\documentclass[11pt]{amsart}

\usepackage{
    amsmath,
    amsfonts,
    amssymb,
    amsthm,
    amscd,
    comment,
    enumitem,
    etoolbox,
    gensymb,    
    mathtools,
    mathdots
}
\usepackage[usenames,dvipsnames]{xcolor}
\usepackage[all]{xy}

\makeatletter
\@namedef{subjclassname@2020}{%
  \textup{2020} Mathematics Subject Classification}
\makeatother

\usepackage[T1]{fontenc}
\usepackage{bbm}                     
\usepackage[colorlinks=true, linkcolor=blue, citecolor=blue, urlcolor=blue, breaklinks=true]{hyperref}

\DeclareFontFamily{OT1}{pzc}{}
\DeclareFontShape{OT1}{pzc}{m}{it}{<-> s * [1.10] pzcmi7t}{}
\DeclareMathAlphabet{\mathpzc}{OT1}{pzc}{m}{it}

\usepackage[capitalize]{cleveref}   

\crefname{defin}{Definition}{Definitions}
\crefname{eg}{Example}{Examples}
\crefname{lem}{Lemma}{Lemmas}
\crefname{prop}{Proposition}{Propositions}
\crefname{theo}{Theorem}{Theorems}
\crefname{equation}{}{}
\crefname{enumi}{}{}

\newcommand\N{\mathbb{N}}

\newcommand\Z{\mathbb{Z}}
\newcommand\kk{\Bbbk}
\newcommand\one{\mathbbm{1}}

\newcommand\cC{\mathcal{C}}
\newcommand\cD{\mathcal{D}}
\newcommand\cE{\mathcal{E}}

\newcommand\cM{\mathcal{M}}

\newcommand\fC{\mathfrak{C}}            
\newcommand\fS{\mathfrak{S}}            
\newcommand\gl{\mathfrak{gl}}

\newcommand\go{\mathsf{I}}              

\newcommand{\md}{\textup{-mod}}
\newcommand\leftdual[1]{\prescript{\vee}{}{\! #1}}

\newcommand\Braid{\mathpzc{Braid}}      
\newcommand\cEnd{\mathpzc{End}}         
\newcommand\cH{\mathpzc{H}}             
\newcommand\FOT{\mathpzc{FOT}}          
\newcommand\FT{\mathpzc{FT}}            
\newcommand\OS{\mathpzc{OS}}            
\newcommand\OT{\mathpzc{OT}}            
\newcommand\KS{\mathpzc{KS}}            

\newcommand\Sym{\mathpzc{Sym}}
\newcommand\T{\mathpzc{T}}              
\newcommand\TL{\mathpzc{TL}}            

\DeclareMathOperator{\Aff}{Aff}
\DeclareMathOperator{\End}{End}

\DeclareMathOperator{\Hom}{Hom}
\DeclareMathOperator{\htr}{Tr_h}                 
\DeclareMathOperator{\id}{id}

\DeclareMathOperator{\Mor}{Mor}

\DeclareMathOperator{\Ob}{Ob}

\DeclareMathOperator{\Rot}{Rot}
\DeclareMathOperator{\Span}{span}
\DeclareMathOperator{\vtr}{Tr_v}                  
\DeclareMathOperator{\writhe}{writhe}

\usepackage{tikz}
\usetikzlibrary{arrows,patterns}
\usepackage{tikz-cd}

\newcommand{\rectcolor}{red!40!white}
\newcommand{\dotlabel}[1]{$\scriptstyle{#1}$}
\newcommand{\braidup}{to[out=up,in=down]}
\newcommand{\braiddown}{to[out=down,in=up]}
\newcommand{\opendot}[1]{\filldraw[fill=white,draw=black] (#1) circle (2pt)}
\newcommand{\posdot}[1]{
    \filldraw[fill=white,draw=black] (#1)+(0.09,0) arc(0:360:0.09);
    \draw (#1)++(-0.09,0) to ++(0.18,0);
    \draw (#1)++(0,-0.09) -- ++(0,0.18)
}
\newcommand{\negdot}[1]{
    \filldraw[fill=white,draw=black] (#1)+(0.09,0) arc(0:360:0.09);
    \draw (#1)++(-0.09,0) to ++(0.18,0)
}
\newcommand{\coupon}[3][0.15]{
    \filldraw[draw=black,fill=white] (#2) circle (#1);
    \node at (#2) {$\scriptscriptstyle{#3}$}
}

\newcommand{\bubright}[1]{
    \draw[->] (#1)+(0.2,0) arc(360:0:0.2)
}
\newcommand{\bubleft}[1]{
    \draw[->] (#1)+(0.2,0) arc(0:360:0.2)
}
\newcommand{\bubun}[1]{
    \draw (#1)+(0.2,0) arc(0:360:0.2)
}
\newcommand{\identify}[4]{
    \draw[red,dashed] (#1,#2) -- (#1,#4);
    \draw[red,dashed] (#3,#2) -- (#3,#4)
}

\newcommand{\poscross}{
    \begin{tikzpicture}[centerzero]
        \draw (0.2,-0.2) -- (-0.2,0.2);
        \draw[wipe] (-0.2,-0.2) -- (0.2,0.2);
        \draw (-0.2,-0.2) -- (0.2,0.2);
    \end{tikzpicture}
}
\newcommand{\negcross}{
    \begin{tikzpicture}[centerzero]
        \draw (-0.2,-0.2) -- (0.2,0.2);
        \draw[wipe] (0.2,-0.2) -- (-0.2,0.2);
        \draw (0.2,-0.2) -- (-0.2,0.2);
    \end{tikzpicture}
}
\newcommand{\posupcross}{
    \begin{tikzpicture}[centerzero]
        \draw[->] (0.2,-0.2) -- (-0.2,0.2);
        \draw[wipe] (-0.2,-0.2) -- (0.2,0.2);
        \draw[->] (-0.2,-0.2) -- (0.2,0.2);
    \end{tikzpicture}
}
\newcommand{\negupcross}{
    \begin{tikzpicture}[centerzero]
        \draw[->] (-0.2,-0.2) -- (0.2,0.2);
        \draw[wipe] (0.2,-0.2) -- (-0.2,0.2);
        \draw[->] (0.2,-0.2) -- (-0.2,0.2);
    \end{tikzpicture}
}
\newcommand{\posrightcross}{
    \begin{tikzpicture}[centerzero]
        \draw[->] (-0.2,-0.2) -- (0.2,0.2);
        \draw[wipe] (0.2,-0.2) -- (-0.2,0.2);
        \draw[<-] (0.2,-0.2) -- (-0.2,0.2);
    \end{tikzpicture}
}
\newcommand{\negrightcross}{
    \begin{tikzpicture}[centerzero]
        \draw[<-] (0.2,-0.2) -- (-0.2,0.2);
        \draw[wipe] (-0.2,-0.2) -- (0.2,0.2);
        \draw[->] (-0.2,-0.2) -- (0.2,0.2);
    \end{tikzpicture}
}

\newcommand{\posdowncross}{
    \begin{tikzpicture}[centerzero]
        \draw[<-] (0.2,-0.2) -- (-0.2,0.2);
        \draw[wipe] (-0.2,-0.2) -- (0.2,0.2);
        \draw[<-] (-0.2,-0.2) -- (0.2,0.2);
    \end{tikzpicture}
}
\newcommand{\negdowncross}{
    \begin{tikzpicture}[centerzero]
        \draw[<-] (-0.2,-0.2) -- (0.2,0.2);
        \draw[wipe] (0.2,-0.2) -- (-0.2,0.2);
        \draw[<-] (0.2,-0.2) -- (-0.2,0.2);
    \end{tikzpicture}
}
\newcommand{\posleftcross}{
    \begin{tikzpicture}[centerzero]
        \draw[->] (0.2,-0.2) -- (-0.2,0.2);
        \draw[wipe] (-0.2,-0.2) -- (0.2,0.2);
        \draw[<-] (-0.2,-0.2) -- (0.2,0.2);
    \end{tikzpicture}
}
\newcommand{\negleftcross}{
    \begin{tikzpicture}[centerzero]
        \draw[<-] (-0.2,-0.2) -- (0.2,0.2);
        \draw[wipe] (0.2,-0.2) -- (-0.2,0.2);
        \draw[->] (0.2,-0.2) -- (-0.2,0.2);
    \end{tikzpicture}
}
\newcommand{\symcross}{
    \begin{tikzpicture}[centerzero]
        \draw[->] (-0.2,-0.2) -- (0.2,0.2);
        \draw[->] (0.2,-0.2) -- (-0.2,0.2);
    \end{tikzpicture}
}

\newcommand{\dotup}{
    \begin{tikzpicture}[centerzero]
        \draw[->] (0,-0.2) -- (0,0.2);
        \opendot{0,0};
    \end{tikzpicture}
}
\newcommand{\dotdown}{
    \begin{tikzpicture}[centerzero]
        \draw[->] (0,0.2) -- (0,-0.2);
        \opendot{0,0};
    \end{tikzpicture}
}

\newcommand{\posdotun}{
    \begin{tikzpicture}[centerzero]
        \draw (0,0.2) -- (0,-0.2);
        \posdot{0,0};
    \end{tikzpicture}
}

\newcommand{\negdotun}{
    \begin{tikzpicture}[centerzero]
        \draw (0,0.2) -- (0,-0.2);
        \negdot{0,0};
    \end{tikzpicture}
}

\newcommand{\rightcup}{
    \begin{tikzpicture}[anchorbase]
        \draw[->] (-0.15,0.15) -- (-0.15,0) arc(180:360:0.15) -- (0.15,0.15);
    \end{tikzpicture}
}
\newcommand{\leftcup}{
    \begin{tikzpicture}[anchorbase]
        \draw[<-] (-0.15,0.15) -- (-0.15,0) arc(180:360:0.15) -- (0.15,0.15);
    \end{tikzpicture}
}
\newcommand{\uncup}{
    \begin{tikzpicture}[anchorbase]
        \draw (-0.15,0.15) -- (-0.15,0) arc(180:360:0.15) -- (0.15,0.15);
    \end{tikzpicture}
}
\newcommand{\rightcap}{
    \begin{tikzpicture}[anchorbase]
        \draw[->] (-0.15,-0.15) -- (-0.15,0) arc(180:0:0.15) -- (0.15,-0.15);
    \end{tikzpicture}
}

\newcommand{\uncap}{
    \begin{tikzpicture}[anchorbase]
        \draw (-0.15,-0.15) -- (-0.15,0) arc(180:0:0.15) -- (0.15,-0.15);
    \end{tikzpicture}
}

\tikzset{anchorbase/.style={>=To,baseline={([yshift=-0.5ex]current bounding box.center)}}}
\tikzset{ 
    centerzero/.style={>=To,baseline={([yshift=-0.5ex](#1))}},
    centerzero/.default={0,0}
}
\tikzset{wipe/.style={white,line width=4pt}}

\newtheorem{theo}{Theorem}[section]

\newtheorem{prop}[theo]{Proposition}
\newtheorem{lem}[theo]{Lemma}
\newtheorem{cor}[theo]{Corollary}

\theoremstyle{definition}
\newtheorem{defin}[theo]{Definition}
\newtheorem{rem}[theo]{Remark}
\newtheorem{eg}[theo]{Example}

\numberwithin{equation}{section}
\allowdisplaybreaks

\setenumerate[1]{label=(\alph*)}          

\newtoggle{comments}
\newtoggle{details}
\newtoggle{detailsnote}

\iftoggle{comments}{%
  \newcommand{\acomments}[1]{
    \ \\
    {\color{red}
      \textbf{AS:} #1
    }
    \ \\
    }
  \newcommand{\question}[1]{
    \ \\
    {\color{blue}
      \textbf{Question:} #1
    }
    \ \\
    }
  \newcommand{\ycomments}[1]{
    \ \\
    {\color{red}
      \textbf{YM:} #1
    }
    \ \\
    }
}{%
  \newcommand{\acomments}[1]{}
  \newcommand{\question}[1]{}
  \newcommand{\ycomments}[1]{}
}

\iftoggle{details}{%
  \newcommand{\details}[1]{
      \ \\
      {\color{OliveGreen}
        \textbf{Details:} #1
      }
      \\
  }
}{%
  \newcommand{\details}[1]{}
}

\begin{document}

\title{Affinization of monoidal categories}

\author{Youssef Mousaaid}
\address[Y.M.]{
  Department of Mathematics and Statistics \\
  University of Ottawa \\
  Ottawa, ON K1N 6N5, Canada
}
\email{ymous016@uottawa.ca}

\author{Alistair Savage}
\address[A.S.]{
  Department of Mathematics and Statistics \\
  University of Ottawa \\
  Ottawa, ON K1N 6N5, Canada
}
\urladdr{\href{https://alistairsavage.ca}{alistairsavage.ca}, \textrm{\textit{ORCiD}:} \href{https://orcid.org/0000-0002-2859-0239}{orcid.org/0000-0002-2859-0239}}
\email{alistair.savage@uottawa.ca}

\begin{abstract}
    We define the \emph{affinization} of an arbitrary monoidal category $\cC$, corresponding to the category of $\cC$-diagrams on the cylinder.  We also give an alternative characterization in terms of adjoining \emph{dot generators} to $\cC$.  The affinization formalizes and unifies many constructions appearing in the literature.  In particular, we describe a large number of examples coming from Hecke-type algebras, braids, tangles, and knot invariants.   When $\cC$ is rigid, its affinization is isomorphic to its horizontal trace, although the two definitions look quite different.  In general, the affinization and the horizontal trace are not isomorphic.
\end{abstract}



\ifboolexpr{togl{comments} or togl{details}}{%
  {\color{magenta}DETAILS OR COMMENTS ON}
}{%
}

\maketitle
\thispagestyle{empty}


\section{Introduction}

The goal of the current paper is to formalize and unify the concept of \emph{affinization} in certain areas of category theory and representation theory.  In particular, we are interested in the following uses of the term \emph{affine}:
\begin{itemize}[wide]
    \item \emph{Topological/diagrammatic}: The term \emph{affine} is often used to refer to the topological or diagrammatic setting of a torus, annulus, or cylinder.  For example, one often encounters the term \emph{affine braids} to refer to braids on a cylinder.  More generally, the term \emph{affine} is often used to describe either monoidal categories drawn in terms of string diagrams on a cylinder or annulus, or where strings are allowed to carry \emph{dots} with various properties.  The description in terms of dots is often ad hoc.  For example, the dots are sometimes equal to their own mates and sometimes they are not.

    \item \emph{Algebraic}: In the context of Hecke-type algebras, the term \emph{affine} refers to the introduction of a (Laurent) polynomial part of the algebra.  (Of course, this can be explained in terms of the Hecke algebra of an affine Weyl group.)  For example, the affine Hecke algebra $H_r^\text{aff}$ of type $A_{r-1}$ over a commutative ground ring $\kk$ is isomorphic, as a $\kk$-module, to $\kk[x_1^{\pm 1}, \dotsc, x_r^{\pm 1}] \otimes_\kk H_r$, where $H_r$ is the Iwahori--Hecke algebra of type $A_{r-1}$.  Allowing $r$ to vary, the affine Hecke algebras can be organized into a tower of algebras, which can be viewed as a monoidal category.  In the string diagram calculus for monoidal categories, the elements $x_i$ correspond to the \emph{dots} mentioned above.

    \item \emph{Representation theoretic}: The term \emph{affine} often appears in the context of duality statements in representation theory.  For example, quantum Schur--Weyl duality states that there is a surjective algebra homomorphism $H_r \to \End_{U_q(\gl_n)}(V^{\otimes r})$, where $V$ is the quantum analogue of the natural representation of $\gl_n$.  This induces an algebra homomorphism from $H_r$ to the endomorphism algebra of the functor $V^{\otimes r} \otimes - \colon U_q(\gl_n)\md \to U_q(\gl_n)\md$.  This homomorphism is not surjective in general.  Instead the braiding on $U_q(\gl_n)\md$ coming from the $R$-matrix allows one to extend it to an algebra homomorphism from $H_r^\text{aff}$ to the endomorphism algebra of the functor $V^{\otimes r} \otimes -$.
\end{itemize}

Preference for one of the languages listed above often depends on one's point of view.  Topologists interested in skein theory and knot invariants will prefer the topological point of view, while representation theorists may prefer the algebraic or representation theoretic point of view.  In some cases, the translation between the languages is well known.  For instance, this is the case for the affine Hecke algebras mentioned above.  Our aim is to completely unify the different languages with a general approach.  The fact that the various viewpoints coincide gives a rich interplay between topology and representation theory.

Our starting point is the definition of the \emph{affinization} $\Aff(\cC)$ of an arbitrary strict monoidal category $\cC$.  In terms of the usual string diagram calculus for monoidal categories, $\Aff(\cC)$ should be thought of as the category of $\cC$-diagrams on the cylinder.  Its definition (\cref{affdef}) involves adjoining to $\cC$ invertible morphisms corresponding to strings wrapping around the cylinder, subject to natural relations.  If $\cC$ is braided, then $\Aff(\cC)$ is a strict monoidal category, with the tensor product corresponding to nesting of cylinders.  Furthermore, in this case, we can give an equivalent definition of the affinization involving the addition of \emph{dot generators} on strands (\cref{plain}).  The equivalence of these two definitions is a very general and precise statement of the correspondence between the topological/diagrammatic and algebraic notions of \emph{affine} described above.  A significant advantage of the description of $\Aff(\cC)$ in terms of dot generators is that string diagrams become easier to draw, since we no longer need to draw them on a cylinder.

We next turn our attention to categorical actions.  If $F \colon \cC \to \cM$ is a monoidal functor, then $\cC$ acts on $\cM$ via the action
\[
    X \cdot M = F(X) \otimes M,\quad f \cdot g = F(f) \otimes g,
\]
for $X \in \Ob(\cC)$, $M \in \Ob(\cM)$, $f \in \Mor(\cC)$, and $g \in \Mor(\cM)$.  If $\cC$ is \emph{balanced} (i.e.\ it is braided and has a twist; see \cref{sec:action}), then we show in \cref{salamander} that this action can be extended in a natural way to the affinization $\Aff(\cC)$.  This yields the precise connection to the representation theoretic viewpoint mentioned above; see \cref{gopher}.

Our affinization procedure recovers a large number of examples appearing in the literature, unifying them into a single precise framework.  Examples include the following:
\begin{itemize}
    \item The affinization of the category of braids over the disc is the category of braids over the annulus (\cref{annularbraids}).

    \item The affinization of the tower of Iwahori--Hecke algebras of type $A$, naturally viewed as a monoidal category, is the tower of affine Hecke algebras of type $A$ (\cref{subsec:Hecke}).

    \item The affinization of the category of oriented tangles (respectively, framed oriented tangles) over the disc is the category of oriented tangles (respectively, framed oriented tangles) over the annulus (\cref{crocodile,alligator}).  Analogous results also hold for unoriented tangles (\cref{kiwi,dodo}).

    \item The affinization of the framed HOMFLYPT skein category over the disc is the framed HOMFLYPT skein category over the annulus (\cref{dragon}).

    \item The affinization of the Kauffman skein category over the disc is the Kauffman skein category over the annulus (\cref{dinosaur}).

    \item The affinization of the Temperley--Lieb category is the affine Temperley--Lieb category (\cref{subsec:TL}).
\end{itemize}
In all of the above examples, our general results give two descriptions of the affine categories, one in terms of string diagrams on the cylinder and the other in terms of string diagrams carrying dots.  Some of these presentations have appeared previously in the literature, while others are new.  (See the body of the paper for references to the literature in each case.)  The appeal of our approach is a completely uniform treatment.

Another construction that has appeared in the literature in the context of monoidal categories on the annulus is the \emph{horizontal trace}.  The horizontal trace $\htr(\cC)$ has the same objects as $\cC$, and its morphisms are equivalences classes of certain morphisms in $\cC$; see \cref{sec:htr}.  We show (\cref{donut}) that, when $\cC$ is rigid (i.e.\ it has left and right duals), the affinization $\Aff(\cC)$ and the horizontal trace $\htr(\cC)$ are isomorphic.  However, even in this case, the two definitions are quite different.  The affinization involves adjoining additional morphisms subject to some natural relations, while the horizontal trace involves equivalence classes of morphisms.  This difference makes the affinization easier to work with in many cases.  In general (i.e.\ when $\cC$ is not rigid), the affinization and the horizontal trace are \emph{not} isomorphic, and it is the affinization, and not the horizontal trace, that gives the correct notion of $\cC$-diagrams on a cylinder for two reasons:  First, the interpretation of morphisms in the horizontal trace as string diagrams on the cylinder (or annulus) involves cups and caps that have no precise meaning when $\cC$ is not rigid.  String diagrams in the affinization avoid such cups and caps.  Second, in specific cases, it is the affinization that gives the expected ``affine'' category.  For example, if $\cC$ is the category of braids over the disc, then $\Aff(\cC)$ is the category of braids over the annulus, while $\htr(\cC)$ is quite different; see \cref{eggs}.

The \emph{vertical trace} $\vtr(\cC)$ of a $\kk$-linear category $\cC$ is the $\kk$-module given by linear combinations of endomorphisms in $\cC$ modulo the relation $f \circ g = g \circ f$ for morphisms $f \colon X \to Y$, $g \colon Y \to X$ in $\cC$; see \cref{tracedef}.  If $\cC$ is strict pivotal, elements of the trace are often drawn as diagrams on the annulus.  In \cref{sec:vtr} we discuss how the procedure of taking the vertical trace behaves with respect to the process of affinization.  In particular, if $\cC$ is a balanced strict $\kk$-linear monoidal category, then $\vtr(\Aff(\cC))$ can be viewed as the category of $\cC$-diagrams on the torus, and it acts naturally on $\vtr(\cC)$.  This action corresponds to placing an annular diagram representing a morphism in $\vtr(\cC)$ inside the toroidal diagram representing a morphism in $\vtr(\Aff(\cC))$; see \cref{fire} for a more general statement.  We also show (\cref{cinnamon})  that if $\cC$ is a right-rigid or left-rigid braided strict monoidal category, then $\vtr(\cC)$ is isomorphic, as a $\kk$-algebra, to the center $Z(\Aff(\cC)) := \End_{\Aff(\cC)}(\one)$ of $\Aff(\cC)$.

We conclude the paper with a brief discussion of how the concept of affinization can be extended to the setting of 2-categories (\cref{sec:2aff}).  For rigid 2-categories, the affinization is again isomorphic to the horizontal trace.  However, the two concepts are different in general.

\subsection*{Acknowledgements}

This research of A.~Savage was supported by Discovery Grant RGPIN-2017-03854 from the Natural Sciences and Engineering Research Council of Canada.  Y.~Mousaaid was also supported by this Discovery Grant.  The authors would like to thank A.~Brochier, J.~Brundan, and P.~Samuelson for helpful conversations.

\section{Affinization of monoidal categories\label{sec:affinize}}

Throughout this paper, we will use the usual calculus of string diagrams for monoidal categories.  We assume all categories are essentially small.  For a category $\cC$, we let $\Ob(\cC)$ denote its set of objects, and $\Mor(\cC)$ its set of morphisms.  We let $1_X$ denote the identity morphism of an object $X$.  We use $\kk$ to denote a commutative ground ring.  By a \emph{monoidal functor}, we mean a strong monoidal functor.  We let $\one$ denote the unit object of a monoidal category.  Throughout this section, $\cC$ denotes a strict monoidal category.
\details{
    Consider the strict monoidal category $\cC_G(A)$ of Example~2.3.6 in \href{http://dx.doi.org/10.1090/surv/205}{Tensor categories} by Etingof, Gelaki, Nikshych, and Ostrik.  Then nontrivial 2-cocycles for the trivial action of $G$ on $A$ give rise to non-strict monoidal functors.  More precisely, given a 2-cocycle $\gamma$, we have the identity functor $\mathcal{F} \colon \cC_G(A) \to \cC_G(A)$, the identity morphism $\epsilon \colon \delta_e \to \delta_e$ (where $e$ is the unit element of the group, so that $\delta_e$ is the unit element of $\cC_G(A)$) and the natural transformation $\mu_{g,h} = \gamma(g,h) \id_{gh} \colon gh = \mathcal{F}(\delta_g) \otimes \mathcal{F}(\delta_h) \to \mathcal{F}(\delta_g \otimes \delta_h) = gh$.  Then $(F,\epsilon,\mu)$ is a strong monoidal functor and, if $\gamma$ is nontrivial, it is not strict.
}

\begin{defin}[Affinization of a strict monoidal category] \label{affdef}
    The \emph{affinization} of a strict monoidal category $\cC$ is the category $\Aff(\cC)$ obtained from $\cC$ by adjoining invertible morphisms $\xi_{X,Y} \colon X \otimes Y \to Y \otimes X$ for each pair of objects $X,Y \in \Ob(\cC)$, subject to the relations
    \begin{gather} \label{coilrel1}
        \xi_{X,Y \otimes Z} = \xi_{Z \otimes X,Y} \circ \xi_{X \otimes Y,Z},
        \\ \label{coilrel2}
        \xi_{X_2,Y_2} \circ (g \otimes f) = (f \otimes g) \circ \xi_{X_1,Y_1},
    \end{gather}
    for all $X,X_1,X_2,Y,Y_1,Y_2,Z \in \Ob(\cC)$, $f \in \Hom_\cC(Y_1,Y_2)$, and $g \in \Hom_\cC(X_1,X_2)$.  We call the $\xi_{X,Y}$ \emph{coils}.  We define $\xi_X := \xi_{\one,X}$.  If $\cC$ is a \emph{$\kk$-linear} strict monoidal category, then $\Aff(\cC)$ is also naturally $\kk$-linear.
\end{defin}

It follows from \cref{coilrel1} that $\xi_{X,\one} = \xi_{X,\one \otimes \one} = \xi_{X,\one} \circ \xi_{X,\one}$ and so, since $\xi_{X,\one}$ is invertible, we have
\begin{equation} \label{cable}
    \xi_{X,\one} = 1_X
    \quad \text{for all } X \in \Ob(\cC).
\end{equation}

In terms of string diagrams, we should picture the morphisms of $\Aff(\cC)$ as string diagrams on a cylinder, where the coil $\xi_{X,Y}$ corresponds to a strand labeled $Y$ wrapping around the cylinder:
\begin{equation} \label{pipe}
    \xi_{X,Y}
    =
    \begin{tikzpicture}[anchorbase]
        \draw[\rectcolor] (-0.4,-0.5) -- (-0.4,0.5);
        \draw[\rectcolor] (0.4,-0.5) -- (0.4,0.5);
        \draw[\rectcolor] (0,0.5) ellipse (0.4 and 0.15);
        \draw[\rectcolor] (-0.4,-0.5) arc (180:360:0.4 and 0.15);
        \draw[dashed,\rectcolor] (-0.4,-0.5) arc (180:0:0.4 and 0.15);
        \draw (-0.2,-0.63) node[anchor=north] {\dotlabel{X}} \braidup (0.2,0.37);
        \draw (0.2,-0.63) node[anchor=north] {\dotlabel{Y}} to[out=up,in=down] (0.4,-0.2);
        \draw[dashed] (0.4,-0.2) to[out=up,in=down] (-0.4,0.2);
        \draw (-0.4,0.2) to[out=up,in=down] (-0.2,0.37);
    \end{tikzpicture}
    \ ,\qquad
    \xi_{X,Y}^{-1}
    =
    \begin{tikzpicture}[anchorbase]
        \draw[\rectcolor] (-0.4,-0.5) -- (-0.4,0.5);
        \draw[\rectcolor] (0.4,-0.5) -- (0.4,0.5);
        \draw[\rectcolor] (0,0.5) ellipse (0.4 and 0.15);
        \draw[\rectcolor] (-0.4,-0.5) arc (180:360:0.4 and 0.15);
        \draw[dashed,\rectcolor] (-0.4,-0.5) arc (180:0:0.4 and 0.15);
        \draw (0.2,-0.63) node[anchor=north] {\dotlabel{X}} \braidup (-0.2,0.37);
        \draw (-0.2,-0.63) node[anchor=north] {\dotlabel{Y}} to[out=up,in=down] (-0.4,-0.2);
        \draw[dashed] (-0.4,-0.2) to[out=up,in=down] (0.4,0.2);
        \draw (0.4,0.2) to[out=up,in=down] (0.2,0.37);
    \end{tikzpicture}
    \ ,\qquad
    \xi_X
    =
    \begin{tikzpicture}[anchorbase]
        \draw[\rectcolor] (-0.4,-0.5) -- (-0.4,0.5);
        \draw[\rectcolor] (0.4,-0.5) -- (0.4,0.5);
        \draw[\rectcolor] (0,0.5) ellipse (0.4 and 0.15);
        \draw[\rectcolor] (-0.4,-0.5) arc (180:360:0.4 and 0.15);
        \draw[dashed,\rectcolor] (-0.4,-0.5) arc (180:0:0.4 and 0.15);
        \draw (0,-0.65) node[anchor=north] {\dotlabel{X}} to[out=up,in=down] (0.4,-0.3);
        \draw[dashed] (0.4,-0.3) to[out=up,in=down] (-0.4,0.1);
        \draw (-0.4,0.1) to[out=up,in=down] (0,0.35);
    \end{tikzpicture}
    \ ,\qquad
    \xi_X^{-1}
    =
    \begin{tikzpicture}[anchorbase]
        \draw[\rectcolor] (-0.4,-0.5) -- (-0.4,0.5);
        \draw[\rectcolor] (0.4,-0.5) -- (0.4,0.5);
        \draw[\rectcolor] (0,0.5) ellipse (0.4 and 0.15);
        \draw[\rectcolor] (-0.4,-0.5) arc (180:360:0.4 and 0.15);
        \draw[dashed,\rectcolor] (-0.4,-0.5) arc (180:0:0.4 and 0.15);
        \draw (0,-0.65) node[anchor=north] {\dotlabel{X}} to[out=up,in=down] (-0.4,-0.3);
        \draw[dashed] (-0.4,-0.3) to[out=up,in=down] (0.4,0.1);
        \draw (0.4,0.1) to[out=up,in=down] (0,0.35);
    \end{tikzpicture}
    \ .
\end{equation}

\begin{rem}
    Instead of considering diagrams on the cylinder, one can also consider a ``pole'' on the right-hand side of the diagrams.  Then, rather than strands wrapping around the cylinder, they wrap around this pole.  However, the monoidal structure to be discussed below is more intuitive from the cylindrical point of view.
\end{rem}

\begin{rem} \label{cylinder-vs-annulus}
    The cylinder is topologically equivalent to the annulus, and, in the literature, one often sees ``affine type'' categories drawn in terms of string diagrams on the annulus.  We choose here to use cylinders, since this allows us to draw the coils without drawing caps and cups (i.e.\ critical points with respect to the vertical coordinate).  As we will recall below, cups and caps in string diagrams typically arise from dual objects.  Thus, in categories that do not necessarily have duals, the cylindrical diagrammatics seem more natural.  See \cref{ht-annulus} for a similar situation.
\end{rem}

In order to make cylindrical string diagrams easier to draw, we cut open the cylinder, drawing $\xi_{X,Y}$ and $\xi_{X,Y}^{-1}$ as the string diagrams
\[
    \xi_{X,Y}
    =
    \begin{tikzpicture}[anchorbase]
        \identify{-0.7}{-0.5}{0.7}{0.5};
        \draw (-0.3,-0.5) node[anchor=north] {\dotlabel{X}} -- (0.3,0.5) node[anchor=south] {\dotlabel{X}};
        \draw (0.3,-0.5) node[anchor=north] {\dotlabel{Y}} to[out=up,in=200] (0.7,0);
        \draw (-0.7,0) to[out=20,in=down] (-0.3,0.5) node[anchor=south] {\dotlabel{Y}};
    \end{tikzpicture}
    \ ,\qquad
    \xi_{X,Y}^{-1}
    =
    \begin{tikzpicture}[anchorbase]
        \identify{-0.7}{-0.5}{0.7}{0.5};
        \draw (0.3,-0.5) node[anchor=north] {\dotlabel{X}} -- (-0.3,0.5) node[anchor=south] {\dotlabel{X}};
        \draw (-0.3,-0.5) node[anchor=north] {\dotlabel{Y}} to[out=up,in=-20] (-0.7,0);
        \draw (0.7,0) to[out=160,in=down] (0.3,0.5) node[anchor=south] {\dotlabel{Y}};
    \end{tikzpicture}
    \ .
\]
where the dashed vertical edges are identified.  Then the relations \cref{coilrel1,coilrel2} become
\begin{equation} \label{whip}
    \begin{tikzpicture}[anchorbase]
        \identify{-0.7}{-0.5}{0.7}{0.5};
        \draw (-0.3,-0.5) node[anchor=north] {\dotlabel{X}} -- (0.3,0.5) node[anchor=south] {\dotlabel{X}};
        \draw (0.3,-0.5) node[anchor=north] {\dotlabel{Y \otimes Z}} to[out=up,in=200] (0.7,0);
        \draw (-0.7,0) to[out=20,in=down] (-0.3,0.5) node[anchor=south] {\dotlabel{Y \otimes Z}};
    \end{tikzpicture}
    =
    \begin{tikzpicture}[anchorbase]
        \identify{-0.8}{-0.5}{0.8}{0.5};
        \draw (-0.3,-0.5) node[anchor=north] {\dotlabel{X}} -- (0.3,0.5) node[anchor=south] {\dotlabel{X}};
        \draw (0.45,-0.5) node[anchor=north] {\dotlabel{Z}} to[out=up,in=200] (0.8,-0.15);
        \draw (0.15,-0.5) node[anchor=north] {\dotlabel{Y}} to[out=up,in=200] (0.8,0.15);
        \draw (-0.8,-0.15) to[out=20,in=down] (-0.15,0.5) node[anchor=south] {\dotlabel{Z}};
        \draw (-0.8,0.15) to[out=20,in=down] (-0.45,0.5) node[anchor=south] {\dotlabel{Y}};
    \end{tikzpicture}
    \ ,\qquad
    \begin{tikzpicture}[anchorbase]
        \identify{-0.6}{-0.5}{0.8}{0.5};
        \draw (-0.3,-0.5) node[anchor=north] {\dotlabel{X_1}} -- (0.3,0.5) node[anchor=south] {\dotlabel{X_2}};
        \draw (0.3,-0.5) node[anchor=north] {\dotlabel{Y_1}} to (0.3,-0.2) to[out=up,in=200] (0.8,0.2);
        \draw (-0.6,0.2) to[out=20,in=down] (-0.3,0.5) node[anchor=south] {\dotlabel{Y_2}};
        \coupon{-0.12,-0.2}{g};
        \coupon{0.3,-0.2}{f};
    \end{tikzpicture}
    =
    \begin{tikzpicture}[anchorbase]
        \identify{-0.8}{-0.5}{0.6}{0.5};
        \draw (-0.3,-0.5) node[anchor=north] {\dotlabel{X_1}} -- (0.3,0.5) node[anchor=south] {\dotlabel{X_2}};
        \draw (0.3,-0.5) node[anchor=north] {\dotlabel{Y_1}} to[out=up,in=200] (0.6,-0.2);
        \draw (-0.8,-0.2) to[out=20,in=down] (-0.3,0.2) -- (-0.3,0.5) node[anchor=south] {\dotlabel{Y_2}};
        \filldraw[draw=black,fill=white] (-0.3,0.2) circle (0.15);
        \coupon{-0.3,0.2}{f};
        \coupon{0.12,0.2}{g};
    \end{tikzpicture}
    \ .
\end{equation}
Analogous relations also hold for the inverses $\xi_{X,Y}^{-1}$.  Intuitively, we can slide morphisms in $\cC$ around the cylinder.

\begin{rem}[Dehn twist] \label{Dehn}
    It follows from \cref{coilrel2} that
    \begin{equation}
        \xi_Y \circ f = f \circ \xi_X, \quad
        f \in \Hom_\cC(X,Y).
    \end{equation}
    In other words, $(\xi_X)_{X \in \Ob(\cC)}$ is a natural transformation of the identity functor $\id_\cC \colon \cC \to \cC$.  In terms of string diagrams, this corresponds to a \emph{Dehn twist} of the cylinder.
\end{rem}

We now wish to endow $\Aff(\cC)$ with the structure of strict monoidal category.  Intuitively, viewing morphisms $f,g$ in $\Aff(\cC)$ as diagrams on the cylinder, the tensor product $f \otimes g$ is given by nesting the cylindrical diagram corresponding to $g$ inside the cylindrical diagram corresponding to $f$.  In order for this to make sense, we need $\cC$ to be a \emph{braided} strict monoidal category, so that we can use the braiding to formalize what it means for strands in one diagram to pass over strands in another diagram.

Recall that a strict monoidal category $\cC$ is \emph{braided} if it is equipped with a natural family of isomorphisms $\beta_{X,Y} \colon X \otimes Y \to Y \otimes X$ satisfying
\begin{equation} \label{braidrel}
    \beta_{X,Y \otimes Z} = (1_Y \otimes \beta_{X,Z}) \circ (\beta_{X,Y} \otimes 1_Z), \quad
    \beta_{X \otimes Y, Z} = (\beta_{X,Z} \otimes 1_Y) \circ (1_X \otimes \beta_{Y,Z}),
\end{equation}
for all $X,Y,Z \in \Ob(\cC)$.  We use the standard string diagrams for the braiding:
\[
    \beta_{X,Y} =
    \begin{tikzpicture}[centerzero]
        \draw (0.3,-0.3) node[anchor=north] {\dotlabel{Y}} -- (-0.3,0.3);
        \draw[wipe] (-0.3,-0.3) -- (0.3,0.3);
        \draw (-0.3,-0.3) node[anchor=north] {\dotlabel{X}} -- (0.3,0.3);
    \end{tikzpicture}
    \ ,\qquad
    \beta_{X,Y}^{-1} =
    \begin{tikzpicture}[centerzero]
        \draw (-0.3,-0.3) node[anchor=north] {\dotlabel{Y}} -- (0.3,0.3);
        \draw[wipe] (0.3,-0.3) -- (-0.3,0.3);
        \draw (0.3,-0.3) node[anchor=north] {\dotlabel{X}} -- (-0.3,0.3);
    \end{tikzpicture}
    \ .
\]
Then the equations in \cref{braidrel} become
\begin{equation} \label{braidextend}
    \begin{tikzpicture}[centerzero]
        \draw (0.5,-0.5) node[anchor=north] {\dotlabel{Y \otimes Z}} -- (-0.5,0.5);
        \draw[wipe] (-0.5,-0.5) -- (0.5,0.5);
        \draw (-0.5,-0.5) node[anchor=north] {\dotlabel{X}} -- (0.5,0.5);
    \end{tikzpicture}
    =
    \begin{tikzpicture}[centerzero]
        \draw (0,-0.5) node[anchor=north] {\dotlabel{Y}} -- (-0.5,0.5);
        \draw (0.5,-0.5) node[anchor=north] {\dotlabel{Z}} -- (0,0.5);
        \draw[wipe] (-0.5,-0.5) -- (0.5,0.5);
        \draw (-0.5,-0.5) node[anchor=north] {\dotlabel{X}} -- (0.5,0.5);
    \end{tikzpicture}
    \ ,\quad
    \begin{tikzpicture}[centerzero]
        \draw (0.5,-0.5) node[anchor=north] {\dotlabel{Z}} -- (-0.5,0.5);
        \draw[wipe] (-0.5,-0.5) -- (0.5,0.5);
        \draw (-0.5,-0.5) node[anchor=north] {\dotlabel{X \otimes Y}} -- (0.5,0.5);
    \end{tikzpicture}
    =
    \begin{tikzpicture}[centerzero]
        \draw (0.5,-0.5) node[anchor=north] {\dotlabel{Z}} -- (-0.5,0.5);
        \draw[wipe] (0,-0.5) -- (0.5,0.5);
        \draw (0,-0.5) node[anchor=north] {\dotlabel{Y}} -- (0.5,0.5);
        \draw[wipe] (-0.5,-0.5) -- (0,0.5);
        \draw (-0.5,-0.5) node[anchor=north] {\dotlabel{X}} -- (0,0.5);
    \end{tikzpicture}
    \ .
\end{equation}

In what follows, we will use unlabeled strands to indicate that a relation holds for any labeling of the strands.  So, for instance,
\[
    \begin{tikzpicture}[centerzero]
        \draw (0,-0.3) -- (0,0.3);
        \coupon{0,0}{f};
    \end{tikzpicture}
    \quad \text{represents an arbitrary morphism} \quad
    \begin{tikzpicture}[centerzero]
        \draw (0,-0.3) node[anchor=north] {\dotlabel{X}} -- (0,0.3) node[anchor=south] {\dotlabel{Y}};
        \filldraw[draw=black,fill=white] (0,0) circle (0.15);
        \node at (0,0) {$\scriptscriptstyle{f}$};
    \end{tikzpicture}
    \quad \text{in } \cC,
\]
and we have
\[
    \begin{tikzpicture}[centerzero]
        \draw (0.25,-0.5) to[out=135,in=down] (-0.2,0) to[out=up,in=225] (0.25,0.5);
        \draw[wipe] (-0.25,-0.5) to[out=45,in=down] (0.2,0) to[out=up,in=-45] (-0.25,0.5);
        \draw (-0.25,-0.5) to[out=45,in=down] (0.2,0) to[out=up,in=-45] (-0.25,0.5);
    \end{tikzpicture}
    \ =\
    \begin{tikzpicture}[centerzero]
        \draw (-0.2,-0.5) -- (-0.2,0.5);
        \draw (0.2,-0.5) -- (0.2,0.5);
    \end{tikzpicture}
    \ =\
    \begin{tikzpicture}[centerzero]
        \draw (-0.25,-0.5) to[out=45,in=down] (0.2,0) to[out=up,in=-45] (-0.25,0.5);
        \draw[wipe] (0.25,-0.5) to[out=135,in=down] (-0.2,0) to[out=up,in=225] (0.25,0.5);
        \draw (0.25,-0.5) to[out=135,in=down] (-0.2,0) to[out=up,in=225] (0.25,0.5);
    \end{tikzpicture}
    \ .
\]
The naturality of the braiding means that
\begin{gather} \label{leopard}
    \begin{tikzpicture}[centerzero]
        \draw (0.5,-0.5) -- (-0.5,0.5);
        \draw[wipe] (0,-0.5) to[out=135,in=down] (-0.4,0) to[out=up,in=225] (0,0.5);
        \draw (0,-0.5) to[out=135,in=down] (-0.4,0) to[out=up,in=225] (0,0.5);
        \draw[wipe] (-0.5,-0.5) -- (0.5,0.5);
        \draw (-0.5,-0.5) -- (0.5,0.5);
    \end{tikzpicture}
    \ =\
    \begin{tikzpicture}[centerzero]
        \draw (0.5,-0.5) -- (-0.5,0.5);
        \draw[wipe] (0,-0.5) to[out=45,in=down] (0.4,0) to[out=up,in=-45] (0,0.5);
        \draw (0,-0.5) to[out=45,in=down] (0.4,0) to[out=up,in=-45] (0,0.5);
        \draw[wipe] (-0.5,-0.5) -- (0.5,0.5);
        \draw (-0.5,-0.5) -- (0.5,0.5);
    \end{tikzpicture}
    \ ,
    \\ \label{lynx}
    \begin{tikzpicture}[centerzero]
        \draw (-0.5,-0.5) -- (0.5,0.5);
        \draw[wipe] (0.5,-0.5) -- (-0.5,0.5);
        \draw (0.5,-0.5) -- (-0.5,0.5);
        \coupon{-0.28,-0.28}{f};
    \end{tikzpicture}
    =
    \begin{tikzpicture}[centerzero]
        \draw (-0.5,-0.5) -- (0.5,0.5);
        \draw[wipe] (0.5,-0.5) -- (-0.5,0.5);
        \draw (0.5,-0.5) -- (-0.5,0.5);
        \coupon{0.28,0.28}{f};
    \end{tikzpicture}
    \ ,\quad
    \begin{tikzpicture}[centerzero]
        \draw (0.5,-0.5) -- (-0.5,0.5);
        \draw[wipe] (-0.5,-0.5) -- (0.5,0.5);
        \draw (-0.5,-0.5) -- (0.5,0.5);
        \coupon{-0.25,-0.25}{f};
    \end{tikzpicture}
    =
    \begin{tikzpicture}[centerzero]
        \draw (0.5,-0.5) -- (-0.5,0.5);
        \draw[wipe] (-0.5,-0.5) -- (0.5,0.5);
        \draw (-0.5,-0.5) -- (0.5,0.5);
        \coupon{0.25,0.25}{f};
    \end{tikzpicture}
    \ ,\quad
    \begin{tikzpicture}[centerzero]
        \draw (-0.5,-0.5) -- (0.5,0.5);
        \draw[wipe] (0.5,-0.5) -- (-0.5,0.5);
        \draw (0.5,-0.5) -- (-0.5,0.5);
        \coupon{0.25,-0.25}{f};
    \end{tikzpicture}
    =
    \begin{tikzpicture}[centerzero]
        \draw (-0.5,-0.5) -- (0.5,0.5);
        \draw[wipe] (0.5,-0.5) -- (-0.5,0.5);
        \draw (0.5,-0.5) -- (-0.5,0.5);
        \coupon{-0.25,0.25}{f};
    \end{tikzpicture}
    \ ,\quad
    \begin{tikzpicture}[centerzero]
        \draw (0.5,-0.5) -- (-0.5,0.5);
        \draw[wipe] (-0.5,-0.5) -- (0.5,0.5);
        \draw (-0.5,-0.5) -- (0.5,0.5);
        \coupon{0.28,-0.28}{f};
    \end{tikzpicture}
    =
    \begin{tikzpicture}[centerzero]
        \draw (0.5,-0.5) -- (-0.5,0.5);
        \draw[wipe] (-0.5,-0.5) -- (0.5,0.5);
        \draw (-0.5,-0.5) -- (0.5,0.5);
        \coupon{-0.28,0.28}{f};
    \end{tikzpicture}
    \ .
\end{gather}
In other words, the braid relations are satisfied, and coupons slide through crossings.

\begin{prop} \label{afftensor}
    If $\cC$ is a braided strict monoidal category, then there is a unique way to extend the tensor product of $\cC$ to $\Aff(\cC)$ such that $\Aff(\cC)$ is a strict monoidal category and
    \begin{equation} \label{stack}
        \xi_{X,Y} \otimes 1_Z = \xi_{X \otimes Z, Y} \circ (1_X \otimes \beta_{Y,Z}), \quad
        1_X \otimes \xi_{Y,Z} = (\beta_{X,Z}^{-1} \otimes 1_Y) \circ \xi_{X \otimes Y,Z},
    \end{equation}
    for all $X,Y,Z \in \Ob(\cC)$.
\end{prop}

\begin{proof}
    Any tensor product on $\Aff(\cC)$ must satisfy
    \begin{equation} \label{deer}
        (g \circ f) \otimes 1_W
        = (g \otimes 1_W) \circ (f \otimes 1_W)
    \end{equation}
    for all morphisms $f \colon X \to Y$ and $g \colon Y \to Z$ in $\Aff(\cC)$ and $W \in \Ob(\cC)$.  Since
    \begin{itemize}
        \item $f \otimes 1_W$ must be the tensor product in $\cC$ if $f$ is a morphism in $\cC$, and is given by \cref{stack} for $f$ a coil, and
        \item morphisms in $\cC$, together with coils, generate the morphisms in $\Aff(\cC)$ under composition,
    \end{itemize}
    there is a unique way to define $f \otimes 1_W$ for a morphism $f$ in $\Aff(\cC)$.  Similarly, there is a unique way to define $1_W \otimes f$.  Then, for morphisms $f \colon X \to Y$ and $g \colon Z \to W$ in $\Aff(\cC)$, we must have
    \[
        g \otimes f = (g \otimes 1_Y) \circ (1_Z \otimes f).
    \]
    This proves the uniqueness statement in the proposition.  It is then a straightforward verification to check that the tensor product, extended to $\Aff(\cC)$ as above, endows $\Aff(\cC)$ with the structure of a strict monoidal category.
\end{proof}

Diagrammatically, the equations in \cref{stack} become
\begin{equation} \label{bridge}
    \begin{tikzpicture}[centerzero]
        \identify{-0.7}{-0.5}{0.7}{0.5};
        \draw (-0.3,-0.5) node[anchor=north] {\dotlabel{X}} -- (0.3,0.5) node[anchor=south] {\dotlabel{X}};
        \draw (0.3,-0.5) node[anchor=north] {\dotlabel{Y}} to[out=up,in=200] (0.7,0);
        \draw (-0.7,0) to[out=20,in=down] (-0.3,0.5) node[anchor=south] {\dotlabel{Y}};
    \end{tikzpicture}
    \ \otimes\
    \begin{tikzpicture}[centerzero]
        \identify{-0.3}{-0.5}{0.3}{0.5};
        \draw (0,-0.5) node[anchor=north] {\dotlabel{Z}} -- (0,0.5);
    \end{tikzpicture}
    \ =\
    \begin{tikzpicture}[centerzero]
        \draw (0.7,-0.5) node[anchor=north] {\dotlabel{Z}} -- (0.7,0.5);
        \draw (-0.3,-0.5) node[anchor=north] {\dotlabel{X}} -- (0.3,0.5) node[anchor=south] {\dotlabel{X}};
        \draw[wipe] (0.3,-0.5) to[out=up,in=200] (1,0.1);
        \draw (0.3,-0.5) node[anchor=north] {\dotlabel{Y}} to[out=up,in=200] (1,0.1);
        \draw (-0.7,0.1) to[out=20,in=down] (-0.3,0.5) node[anchor=south] {\dotlabel{Y}};
        \identify{-0.7}{-0.5}{1}{0.5};
    \end{tikzpicture}
    \ , \qquad
    \begin{tikzpicture}[centerzero]
        \identify{-0.3}{-0.5}{0.3}{0.5};
        \draw (0,-0.5) node[anchor=north] {\dotlabel{X}} -- (0,0.5);
    \end{tikzpicture}
    \ \otimes\
    \begin{tikzpicture}[centerzero]
        \identify{-0.7}{-0.5}{0.7}{0.5};
        \draw (-0.3,-0.5) node[anchor=north] {\dotlabel{Y}} -- (0.3,0.5) node[anchor=south] {\dotlabel{Y}};
        \draw (0.3,-0.5) node[anchor=north] {\dotlabel{Z}} to[out=up,in=200] (0.7,0);
        \draw (-0.7,0) to[out=20,in=down] (-0.3,0.5) node[anchor=south] {\dotlabel{Z}};
    \end{tikzpicture}
    \ =\
    \begin{tikzpicture}[centerzero]
        \draw (-0.3,-0.5) node[anchor=north] {\dotlabel{Y}} -- (0.3,0.5) node[anchor=south] {\dotlabel{Y}};
        \draw (0.3,-0.5) node[anchor=north] {\dotlabel{Z}} to[out=up,in=200] (0.7,-0.1);
        \draw (-1,-0.1) to[out=20,in=down] (-0.3,0.5) node[anchor=south] {\dotlabel{Z}};
        \draw[wipe] (-0.7,-0.5) -- (-0.7,0.5);
        \draw (-0.7,-0.5) node[anchor=north] {\dotlabel{X}} -- (-0.7,0.5);
        \identify{-1}{-0.5}{0.7}{0.5};
    \end{tikzpicture}
    \ .
\end{equation}
Note that the strands of the left-hand diagram in the tensor product pass over those of the right-hand diagram.

For the remainder of this section, we assume $\cC$ is braided, and we view $\Aff(\cC)$ as a strict monoidal category with the tensor product of \cref{afftensor}.  Note that, in general, $\Aff(\cC)$ is no longer braided.  (See, for example, \cref{jaguar+} below.)

We now introduce some diagrammatic shorthand that allows us to dispose of cylindrical diagrams.  We use a \emph{positive dot} on a strand to denote the morphism $\xi_X$ and a \emph{negative dot} to denote $\xi_X^{-1}$:
\[
    \begin{tikzpicture}[centerzero]
        \draw (0,-0.3) node[anchor=north] {\dotlabel{X}} -- (0,0.3);
        \posdot{0,0};
    \end{tikzpicture}
    := \xi_X =
    \begin{tikzpicture}[centerzero]
        \identify{-0.3}{-0.3}{0.3}{0.3};
        \draw (0,-0.3) node[anchor=north] {\dotlabel{X}} to[out=up,in=200] (0.3,0);
        \draw (-0.3,0) to[out=20,in=south] (0,0.3) node[anchor=south] {\dotlabel{X}};
    \end{tikzpicture}
    \ ,\qquad
    \begin{tikzpicture}[centerzero]
        \draw (0,-0.3) node[anchor=north] {\dotlabel{X}} -- (0,0.3);
        \negdot{0,0};
    \end{tikzpicture}
    := \xi_X^{-1} =
    \begin{tikzpicture}[centerzero]
        \identify{-0.3}{-0.3}{0.3}{0.3};
        \draw (0,-0.3) node[anchor=north] {\dotlabel{X}} to[out=up,in=-20] (-0.3,0);
        \draw (0.3,0) to[out=160,in=south] (0,0.3) node[anchor=south] {\dotlabel{X}};
    \end{tikzpicture}
    \ .
\]
It then follows from our definition of the tensor product on $\Aff(\cC)$ that
\begin{equation} \label{wrap}
    \begin{tikzpicture}[centerzero]
        \draw (0,-0.5) -- (0,0.5);
        \posdot{0,0};
        \draw (0.3,-0.5) -- (0.3,0.5);
        \draw (0.6,-0.5) -- (0.6,0.5);
        \draw (0.9,-0.5) -- (0.9,0.5);
        \draw (-0.3,-0.5) -- (-0.3,0.5);
        \draw (-0.6,-0.5) -- (-0.6,0.5);
        \draw (-0.9,-0.5) -- (-0.9,0.5);
        \draw (-1.2,-0.5) -- (-1.2,0.5);
    \end{tikzpicture}
    \ =\
    \begin{tikzpicture}[centerzero]
        \draw (0.3,-0.5) -- (0.3,0.5);
        \draw (0.6,-0.5) -- (0.6,0.5);
        \draw (0.9,-0.5) -- (0.9,0.5);
        \draw[wipe] (0,-0.5) to[out=up,in=200] (1.2,0);
        \draw (0,-0.5) to[out=up,in=200] (1.2,0);
        \draw (-1.5,0) to[out=20,in=down] (0,0.5);
        \draw[wipe] (-0.3,-0.5) -- (-0.3,0.5);
        \draw[wipe] (-0.6,-0.5) -- (-0.6,0.5);
        \draw[wipe] (-0.9,-0.5) -- (-0.9,0.5);
        \draw[wipe] (-1.2,-0.5) -- (-1.2,0.5);
        \draw (-0.3,-0.5) -- (-0.3,0.5);
        \draw (-0.6,-0.5) -- (-0.6,0.5);
        \draw (-0.9,-0.5) -- (-0.9,0.5);
        \draw (-1.2,-0.5) -- (-1.2,0.5);
        \identify{-1.5}{-0.5}{1.2}{0.5};
    \end{tikzpicture}
    \ ,\qquad
    \begin{tikzpicture}[centerzero]
        \draw (0,-0.5) -- (0,0.5);
        \negdot{0,0};
        \draw (0.3,-0.5) -- (0.3,0.5);
        \draw (0.6,-0.5) -- (0.6,0.5);
        \draw (0.9,-0.5) -- (0.9,0.5);
        \draw (-0.3,-0.5) -- (-0.3,0.5);
        \draw (-0.6,-0.5) -- (-0.6,0.5);
        \draw (-0.9,-0.5) -- (-0.9,0.5);
        \draw (-1.2,-0.5) -- (-1.2,0.5);
    \end{tikzpicture}
    \ =\
    \begin{tikzpicture}[centerzero]
        \draw (0.3,-0.5) -- (0.3,0.5);
        \draw (0.6,-0.5) -- (0.6,0.5);
        \draw (0.9,-0.5) -- (0.9,0.5);
        \draw (0,-0.5) to[out=up,in=-20] (-1.5,0);
        \draw[wipe] (1.2,0) to[out=160,in=down] (0,0.5);
        \draw (1.2,0) to[out=160,in=down] (0,0.5);
        \draw[wipe] (-0.3,-0.5) -- (-0.3,0.5);
        \draw[wipe] (-0.6,-0.5) -- (-0.6,0.5);
        \draw[wipe] (-0.9,-0.5) -- (-0.9,0.5);
        \draw[wipe] (-1.2,-0.5) -- (-1.2,0.5);
        \draw (-0.3,-0.5) -- (-0.3,0.5);
        \draw (-0.6,-0.5) -- (-0.6,0.5);
        \draw (-0.9,-0.5) -- (-0.9,0.5);
        \draw (-1.2,-0.5) -- (-1.2,0.5);
        \identify{-1.5}{-0.5}{1.2}{0.5};
    \end{tikzpicture}
    \ ,
\end{equation}
where there can be any number of strands to the left and right of the dots.

Using \cref{whip,leopard,wrap}, we have
\begin{gather} \label{jaguar+}
    \begin{tikzpicture}[centerzero]
        \draw (-0.3,-0.3) -- (0.3,0.3);
        \draw[wipe] (0.3,-0.3) -- (-0.3,0.3);
        \draw (0.3,-0.3) -- (-0.3,0.3);
        \posdot{-0.18,-0.18};
    \end{tikzpicture}
    =
    \begin{tikzpicture}[centerzero]
        \draw (0.3,-0.3) -- (-0.3,0.3);
        \draw[wipe] (-0.3,-0.3) -- (0.3,0.3);
        \draw (-0.3,-0.3) -- (0.3,0.3);
        \posdot{0.15,0.15};
    \end{tikzpicture}
    \ ,\quad
    \begin{tikzpicture}[centerzero]
        \draw (0.3,-0.3) -- (-0.3,0.3);
        \draw[wipe] (-0.3,-0.3) -- (0.3,0.3);
        \draw (-0.3,-0.3) -- (0.3,0.3);
        \posdot{0.18,-0.18};
    \end{tikzpicture}
    =
    \begin{tikzpicture}[centerzero]
        \draw (-0.3,-0.3) -- (0.3,0.3);
        \draw[wipe] (0.3,-0.3) -- (-0.3,0.3);
        \draw (0.3,-0.3) -- (-0.3,0.3);
        \posdot{-0.15,0.15};
    \end{tikzpicture}
    \ ,\
    \begin{tikzpicture}[centerzero]
        \draw (0,-0.4) node[anchor=north] {\dotlabel{X \otimes Y}} -- (0,0.4);
        \posdot{0,0};
    \end{tikzpicture}
    =
    \begin{tikzpicture}[centerzero]
        \draw (0.2,-0.4) node[anchor=north] {\dotlabel{Y}} -- (0.2,0.4);
        \draw (-0.2,-0.4) node[anchor=north] {\dotlabel{X}} -- (-0.2,0.4);
        \posdot{-0.2,0};
        \posdot{0.2,0};
    \end{tikzpicture}
    \ ,\quad
    \begin{tikzpicture}[centerzero]
        \draw (0,-0.4) -- (0,0.4);
        \posdot{0,0.2};
        \coupon{0,-0.15}{f};
    \end{tikzpicture}
    =
    \begin{tikzpicture}[centerzero]
        \draw (0,-0.4) -- (0,0.4);
        \posdot{0,-0.2};
        \coupon{0,0.15}{f};
    \end{tikzpicture}
    \ ,
    \\ \label{jaguar-}
    \begin{tikzpicture}[centerzero]
        \draw (0.3,-0.3) -- (-0.3,0.3);
        \draw[wipe] (-0.3,-0.3) -- (0.3,0.3);
        \draw (-0.3,-0.3) -- (0.3,0.3);
        \negdot{-0.15,-0.15};
    \end{tikzpicture}
    =
    \begin{tikzpicture}[centerzero]
        \draw (-0.3,-0.3) -- (0.3,0.3);
        \draw[wipe] (0.3,-0.3) -- (-0.3,0.3);
        \draw (0.3,-0.3) -- (-0.3,0.3);
        \negdot{0.18,0.18};
    \end{tikzpicture}
    \ ,\quad
    \begin{tikzpicture}[centerzero]
        \draw (-0.3,-0.3) -- (0.3,0.3);
        \draw[wipe] (0.3,-0.3) -- (-0.3,0.3);
        \draw (0.3,-0.3) -- (-0.3,0.3);
        \negdot{0.15,-0.15};
    \end{tikzpicture}
    =
    \begin{tikzpicture}[centerzero]
        \draw (0.3,-0.3) -- (-0.3,0.3);
        \draw[wipe] (-0.3,-0.3) -- (0.3,0.3);
        \draw (-0.3,-0.3) -- (0.3,0.3);
        \negdot{-0.18,0.18};
    \end{tikzpicture}
    \ ,\
    \begin{tikzpicture}[centerzero]
        \draw (0,-0.4) node[anchor=north] {\dotlabel{X \otimes Y}} -- (0,0.4);
        \negdot{0,0};
    \end{tikzpicture}
    =
    \begin{tikzpicture}[centerzero]
        \draw (0.2,-0.4) node[anchor=north] {\dotlabel{Y}} -- (0.2,0.4);
        \draw (-0.2,-0.4) node[anchor=north] {\dotlabel{X}} -- (-0.2,0.4);
        \negdot{-0.2,0};
        \negdot{0.2,0};
    \end{tikzpicture}
    \ ,\quad
    \begin{tikzpicture}[centerzero]
        \draw (0,-0.4) -- (0,0.4);
        \negdot{0,0.2};
        \coupon{0,-0.15}{f};
    \end{tikzpicture}
    =
    \begin{tikzpicture}[centerzero]
        \draw (0,-0.4) -- (0,0.4);
        \negdot{0,-0.2};
        \coupon{0,0.15}{f};
    \end{tikzpicture}
    \ ,\quad
    \begin{tikzpicture}[centerzero]
        \draw (0,-0.3) -- (0,0.3);
        \posdot{0,0.12};
        \negdot{0,-0.12};
    \end{tikzpicture}
    =
    \begin{tikzpicture}[centerzero]
        \draw (0,-0.3) -- (0,0.3);
        \negdot{0,0.12};
        \posdot{0,-0.12};
    \end{tikzpicture}
    =
    \begin{tikzpicture}[centerzero]
        \draw (0,-0.3) -- (0,0.3);
    \end{tikzpicture}
    \ ,
\end{gather}
for all $X,Y \in \Ob(\cC)$ and $f \in \Mor(\cC)$.  Note that the relations in \cref{jaguar-} follow immediately from those in \cref{jaguar+}, together with the fact that the negative dot is inverse to the positive dot.  The third and fourth relations in \cref{jaguar+} are equivalent to the assertion that the collection $(\xi_X)_{X \in \Ob(\cC)}$ is a monoidal natural automorphism of the identity functor on $\cC$; see \cref{Dehn}.

The following result shows, in particular, that the positive and negative dots, together with the morphisms of $\cC$, generate all morphisms of $\Aff(\cC)$ under composition and tensor product.

\begin{theo} \label{plain}
    The affinization $\Aff(\cC)$ of a braided strict monoidal category $\cC$ is isomorphic to the strict monoidal category obtained from $\cC$ by adjoining invertible morphisms
    \[
        \xi_X =
        \begin{tikzpicture}[centerzero]
            \draw (0,-0.3) node[anchor=north] {\dotlabel{X}} -- (0,0.3);
            \posdot{0,0};
        \end{tikzpicture}
        \colon X \to X
    \]
    for all objects $X \in \Ob(\cC)$, subject to the first, third, and fourth relations in \cref{jaguar+}.
\end{theo}

\begin{proof}
    First note that the second relation in \cref{jaguar+} follows from the first relation, after composing on the top with $\poscross$ and on the bottom with $\negcross$.  Let $\Aff'(\cC)$ denote the category described in the statement of the proposition and let us temporarily denote the new generators of $\Aff'(\cC)$ by $\xi'_X$ to avoid confusion.  By \cref{jaguar+}, we have a functor $F \colon \Aff'(\cC) \to \Aff(\cC)$ equal to the identity on $\cC$ and sending $\xi'_X$ to $\xi_X$.  We then define a functor $G \colon \Aff(\cC) \to \Aff'(\cC)$ equal to the identity on $\cC$ and sending $\xi_{X,Y}$ to $\beta_{X,Y} \circ (1_X \otimes \xi_Y')$.  Assuming $G$ is well defined, it is straightforward to verify that $F$ and $G$ are mutually inverse using the computation
    \begin{equation} \label{bootstrap}
        \xi_{X,Y}
        =
        \begin{tikzpicture}[centerzero]
            \identify{-0.7}{-0.5}{0.7}{0.5};
            \draw (-0.3,-0.5) node[anchor=north] {\dotlabel{X}} -- (0.3,0.5);
            \draw (0.3,-0.5) node[anchor=north] {\dotlabel{Y}} to[out=up,in=200] (0.7,0);
            \draw (-0.7,0) to[out=20,in=down] (-0.3,0.5) node[anchor=south] {\dotlabel{Y}};
        \end{tikzpicture}
        =
        \begin{tikzpicture}[centerzero={0,0.2}]
            \draw (0.3,-0.5) node[anchor=north] {\dotlabel{Y}} to[out=up,in=200] (0.7,-0.1);
            \draw (-0.7,-0.1) to[out=20,in=down] (0,0.4) \braidup (-0.3,0.9) node[anchor=south] {\dotlabel{Y}};
            \draw[wipe] (-0.3,-0.5) node[anchor=north] {\dotlabel{X}} -- (-0.3,0.4) \braidup (0.3,0.9);
            \draw (-0.3,-0.5) node[anchor=north] {\dotlabel{X}} -- (-0.3,0.4) \braidup (0.3,0.9);
            \identify{-0.7}{-0.5}{0.7}{0.9};
        \end{tikzpicture}
        = \beta_{X,Y} \circ (1_X \otimes \xi_Y).
    \end{equation}
    To prove that $G$ is well defined, we need to show it preserves the defining relations \cref{coilrel1,coilrel2}.  We first compute that $G$ sends the right-hand side of \cref{coilrel1} to
    \begin{align*}
        \begin{tikzpicture}[centerzero={0,-0.05}]
            \draw (0,-1.1) node[anchor=north] {\dotlabel{Y}} -- (0,-0.6);
            \draw (0.4,-1.1) node[anchor=north] {\dotlabel{Z}} -- (0.4,-0.6) \braidup (-0.4,0);
            \draw[wipe] (0,-0.6) \braidup (0.4,0) \braidup (-0.4,0.6);
            \draw (0,-0.6) \braidup (0.4,0) \braidup (-0.4,0.6);
            \draw[wipe] (-0.4,0) \braidup (0,0.6);
            \draw (-0.4,0) \braidup (0,0.6) -- (0,1);
            \draw (-0.4,0.6) -- (-0.4,1);
            \draw[wipe] (-0.4,-1.1) \braidup (0.4,1);
            \draw (-0.4,-1.1) node[anchor=north] {\dotlabel{X}} \braidup (0.4,1);
            \posdot{0.4,-0.8};
            \posdot{0.4,0};
        \end{tikzpicture}
        =
        \begin{tikzpicture}[centerzero={0,1.05}]
            \draw (0.4,0) node[anchor=north] {\dotlabel{Z}} -- (0.4,0.3) \braidup (0,0.9);
            \draw[wipe] (0,0) -- (0,0.3) \braidup (0.4,0.9);
            \draw (0,0) node[anchor=north] {\dotlabel{Y}} -- (0,0.3) \braidup (0.4,0.9) \braidup (-0.4,2.1);
            \draw[wipe] (0,0.9) \braidup (0.4,1.5);
            \draw (0,0.9) \braidup (0.4,1.5) \braidup (0,2.1);
            \draw[wipe] (-0.4,1.2) \braidup (0.4,2.1);
            \draw (-0.4,0) node[anchor=north] {\dotlabel{X}} -- (-0.4,1.2) \braidup (0.4,2.1);
            \posdot{0.4,0.3};
            \posdot{0.4,0.9};
        \end{tikzpicture}
        \overset{\cref{jaguar+}}{=}
        \begin{tikzpicture}[centerzero={0,1.05}]
            \draw (0,0) node[anchor=north] {\dotlabel{Y}} -- (0,0.3) \braidup (0.4,0.9) \braidup (-0.4,2.1);
            \draw[wipe] (0.4,0) -- (0.4,0.3) \braidup (0,0.9) \braidup (0.4,1.5);
            \draw (0.4,0) node[anchor=north] {\dotlabel{Z}} -- (0.4,0.3) \braidup (0,0.9) \braidup (0.4,1.5) \braidup (0,2.1);
            \draw[wipe] (-0.4,1.2) \braidup (0.4,2.1);
            \draw (-0.4,0) node[anchor=north] {\dotlabel{X}} -- (-0.4,1.2) \braidup (0.4,2.1);
            \posdot{0.4,0.3};
            \posdot{0,0.3};
        \end{tikzpicture}
        =
        \begin{tikzpicture}[centerzero={0,1.05}]
            \draw (0,0) node[anchor=north] {\dotlabel{Y}} -- (0,1) \braidup (-0.4,2.1);
            \draw (0.4,0) node[anchor=north] {\dotlabel{Z}} -- (0.4,1) \braidup (0,2.1);
            \draw[wipe] (-0.4,1) \braidup (0.4,2.1);
            \draw (-0.4,0) node[anchor=north] {\dotlabel{X}} -- (-0.4,1) \braidup (0.4,2.1);
            \posdot{0.4,0.5};
            \posdot{0,0.5};
        \end{tikzpicture}
        \overset{\cref{jaguar+}}{=}
        \begin{tikzpicture}[centerzero]
            \draw (0.3,-0.5) node[anchor=north] {\dotlabel{Y \otimes Z}} -- (-0.3,0.5);
            \draw[wipe] (-0.3,-0.5) -- (0.3,0.5);
            \draw (-0.3,-0.5) node[anchor=north] {\dotlabel{X}} -- (0.3,0.5);
            \posdot{0.17,-0.283};
        \end{tikzpicture}
        ,
    \end{align*}
    which is the image under $G$ of the left-hand side of \cref{coilrel1}.  Finally, for morphisms $f,g$ in $\cC$, $G$ sends the left-hand side of \cref{coilrel2} to
    \[
        \begin{tikzpicture}[centerzero]
            \draw (0.3,-0.6) -- (0.3,0) to[out=up,in=-45] (-0.3,0.5);
            \draw[wipe] (-0.3,-0.5) -- (-0.3,0) to[out=up,in=225] (0.3,0.5);
            \draw (-0.3,-0.6) -- (-0.3,0) to[out=up,in=225] (0.3,0.5);
            \posdot{0.3,0};
            \coupon{0.3,-0.32}{f};
            \coupon{-0.3,-0.32}{g};
        \end{tikzpicture}
        \ \overset{\cref{jaguar+}}{\underset{\cref{lynx}}{=}}\
        \begin{tikzpicture}[centerzero]
            \draw (0.3,-0.5) -- (0.3,-0.25) \braidup (-0.3,0.25) -- (-0.3,0.5);
            \draw[wipe] (-0.3,-0.5) -- (-0.3,-0.25) \braidup (0.3,0.25) -- (0.3,0.5);
            \draw (-0.3,-0.5) -- (-0.3,-0.25) \braidup (0.3,0.25) -- (0.3,0.5);
            \posdot{0.3,-0.25};
            \coupon{-0.3,0.25}{f};
            \coupon{0.3,0.25}{g};
        \end{tikzpicture}
        \ ,
    \]
    which is the image under $G$ of the right-hand side of \cref{coilrel2}.
\end{proof}

For a braided strict monoidal category $\cC$, we now have two diagrammatic ways of viewing the affinization: the ``cylindrical viewpoint'', where we picture string diagrams on the cylinder, and the ``dot viewpoint'', where the coils are depicted using the dot generators.  Both points of view can be useful, depending on the particular application.  One advantage of the dot viewpoint is that we no longer need to work with diagrams on the cylinder; so we can dispense with the dashed vertical lines at the sides of diagrams.

We conclude this section by examining the functoriality of the affinization procedure.  Recall that a \emph{strict} monoidal functor $F \colon \cC \to \cD$  is required to satisfy $F(X \otimes Y) = F(X) \otimes F(Y)$ for all $X,Y \in \Ob(\cC)$.  Note that a strong monoidal functor between strict monoidal categories is not necessarily a \emph{strict} monoidal functor.

\begin{prop}
    Any strict monoidal functor $F \colon \cC \to \cD$ between strict monoidal categories induces a functor $\Aff(F) \colon \Aff(\cC) \to \Aff(\cD)$ by defining
    \begin{gather*}
        \Aff(F)(X) = F(X),\ X \in \Ob(\cC),\qquad
         \Aff(F)(f) = F(f),\ f \in \Mor(\cC),
         \\
        \Aff(F)(\xi_{X,Y}) := \xi_{F(X),F(Y)},\ X,Y \in \Ob(\cC).
    \end{gather*}
    If $F \colon \cC \to \cD$ and $G \colon \cD \to \cE$ are strict monoidal functors between strict monoidal categories, we have $\Aff(G \circ F) = \Aff(G) \circ \Aff(F)$.
\end{prop}

\begin{proof}
    To verify that $\Aff(F)$ is a well-defined functor, it suffices to verify that $\Aff(F)$ preserves the relations \cref{coilrel1,coilrel2}, which follows immediately from the fact that $F$ is a strict monoidal functor.
    \details{
        For example, we have
        \begin{multline*}
            \Aff(X)(\xi_{X,Y \otimes Z})
            = \xi_{F(X),F(Y) \otimes F(Z)}
            \\
            = \xi_{F(Z) \otimes F(X), F(Y)} \circ \xi_{F(X) \otimes F(Y), F(Z)}
            = \Aff(F)(\xi_{Z \otimes X, Y}) \circ \Aff(F)(\xi_{X \otimes Y, Z})
            \\
            = \Aff(F)(\xi_{Z \otimes X, Y} \circ \xi_{X \otimes Y,Z)}
        \end{multline*}
        and
    }
    The final statement of the proposition also follows immediately from the definitions.
\end{proof}

The proof of the following result is a straightforward exercise.

\begin{prop}
    If $F \colon \cC \to \cD$ is a braided strict monoidal functor between braided strict monoidal categories, then $\Aff(F)$ is a strict monoidal functor.
\end{prop}

\begin{rem}
    We expect that the concept of the affinization of a monoidal category can be extended to the setting of monoidal categories which are not necessarily strict.  We have chosen to focus on strict monoidal categories since it significantly simplifies the exposition and all of the examples and applications that we have in mind are strict monoidal categories.  Furthermore, by the Mac Lane coherence theorem, every monoidal category is monoidally equivalent to a strict one.
\end{rem}

\section{Actions\label{sec:action}}

Recall that, for any category $\cM$, the category $\cEnd(\cM)$ of endofunctors and natural transformations is a strict monoidal category.  An \emph{action} of a strict monoidal category $\cC$ on a monoidal category $\cM$ is a monoidal functor $A \colon \cC \to \cEnd(\cM)$.  We adopt the notation $X \cdot M = A(X)(M)$ for $X \in \Ob(\cC)$ and $M \in \Ob(\cM)$.

If $F \colon \cC \to \cM$ is a monoidal functor, then $\cC$ acts on $\cM$ via the action
\begin{equation} \label{viper}
    X \cdot M = F(X) \otimes M,\quad
    f \cdot g = F(f) \otimes g,
\end{equation}
for $X \in \Ob(\cC)$, $M \in \Ob(\cM$), $f \in \Mor(\cC)$, $g \in \Mor(\cM)$.  The goal of this section is to extend this action to the affinization $\Aff(\cC)$.  Intuitively, this action corresponds to placing string diagrams from $\cM$ inside the cylinder corresponding to diagrams from $\Aff(\cC)$, then using the functor $F$ to interpret this as a diagram in $\cM$.  In order for this action to be well defined, we need to make the additional assumption that $\cC$ is a \emph{balanced strict monoidal category}.

Recall that a strict monoidal category $\cC$ is \emph{balanced} if it is braided and has a \emph{twist}, which is a natural transformation $\theta \colon \id_\cC \to \id_\cC$ (recall that $\id_\cC$ is the identity functor on $\cC$), whose components we will denote
\[
    \theta_X =
    \begin{tikzpicture}[centerzero]
         \draw (0,-0.3) node[anchor=north] {\dotlabel{X}} -- (0,0.3);
         \coupon{0,0}{\theta};
    \end{tikzpicture}
    \colon X \to X,\quad X \in \Ob(\cC),
\]
satisfying $\theta_\one = 1_\one$ and
\begin{equation} \label{swirl}
    \begin{tikzpicture}[centerzero]
        \draw (0,-0.4) node[anchor=north] {\dotlabel{X \otimes Y}} -- (0,0.4);
        \coupon{0,0}{\theta};
    \end{tikzpicture}
    =
    \begin{tikzpicture}[centerzero]
        \draw (0.2,-0.4) node[anchor=north] {\dotlabel{Y}} to[out=135,in=down] (-0.2,0);
        \draw[wipe] (-0.2,-0.4) to[out=45,in=down] (0.2,0);
        \draw (-0.2,-0.4) node[anchor=north] {\dotlabel{X}} to[out=45,in=down] (0.2,0) to[out=up,in=-45] (-0.2,0.4);
        \draw[wipe] (-0.2,0) to[out=up,in=225] (0.2,0.4);
        \draw (-0.2,0) to[out=up,in=225] (0.2,0.4);
        \coupon{-0.2,0}{\theta};
        \coupon{0.2,0}{\theta};
    \end{tikzpicture}
    \ ,\quad X,Y \in \Ob(\cC).
\end{equation}
The fact that the family is natural means that the twists commute with morphisms:
\begin{equation} \label{spin}
    \begin{tikzpicture}[centerzero]
        \draw (0,-0.5) node[anchor=north] {\dotlabel{X}} -- (0,0.5) node[anchor=south] {\dotlabel{Y}};
        \coupon{0,-0.2}{f};
        \coupon{0,0.2}{\theta};
    \end{tikzpicture}
    =
    \begin{tikzpicture}[centerzero]
        \draw (0,-0.5) node[anchor=north] {\dotlabel{X}} -- (0,0.5) node[anchor=south] {\dotlabel{Y}};
        \coupon{0,0.2}{f};
        \coupon{0,-0.2}{\theta};
    \end{tikzpicture}
    \ ,\quad f \colon X \to Y \text{ in } \cC.
\end{equation}

\begin{rem}
    In most of the examples to be considered in the current paper, the objects of the category $\cC$ will be freely generated by some set of generating objects.  In this case, any twist on $\cC$ is uniquely determined by the twists of the generating objects and \cref{swirl}.   Furthermore, the category $\cC$ will often be a braided strict pivotal category, in which case we have a twist given by a ``curl''; see \cref{looptwist}.
\end{rem}

The following result extends the action of \cref{viper} to the affinization $\Aff(\cC)$.

\begin{theo} \label{salamander}
    Suppose $\cC$ is a braided strict monoidal category, $\cM$ is a balanced strict monoidal category, and $F \colon \cC \to \cM$ is a braided monoidal functor.  Then there is an action of $\Aff(\cC)$ on $\cM$ uniquely determined by
    \begin{gather} \label{gecko}
        X \cdot M := F(X) \otimes M, \quad
        f \cdot g := F(f) \otimes g, \\
        \xi_X \cdot g := \beta_{N,F(X)} \circ (g \otimes \theta_{F(X)}) \circ \beta_{F(X),M},\quad
        \xi_X^{-1} \cdot g := \beta_{F(X),N}^{-1} \circ (g \otimes \theta_{F(X)}^{-1}) \circ \beta_{M,F(X)}^{-1}, \label{snake}
    \end{gather}
    for all $X \in \Ob(\cC) = \Ob(\Aff(\cC))$, $f \in \Mor(\cC)$, $M,N \in \Ob(\cM)$, $g \in \Hom_\cM(M,N)$.
\end{theo}

In terms of string diagrams, \cref{snake} becomes
\[
    \xi_X \cdot g =
    \begin{tikzpicture}[centerzero]
        \draw (0,-0.6) node[anchor=north] {\dotlabel{M}} -- (0,0);
        \draw[wipe] (-0.5,-0.6) to[out=45,in=down] (0.5,0);
        \draw (-0.5,-0.6) node[anchor=north east] {\dotlabel{F(X)}} to[out=45,in=down] (0.5,0);
        \draw (0.5,0) to[out=up,in=-45] (-0.5,0.6);
        \draw[wipe] (0,0) -- (0,0.6);
        \draw (0,0) -- (0,0.6) node[anchor=south] {\dotlabel{N}};
        \coupon{0,0}{g};
        \coupon{0.5,0}{\theta};
    \end{tikzpicture}
    \ ,\qquad
    \xi_X^{-1} \cdot g =
    \begin{tikzpicture}[centerzero]
        \draw (-0.5,-0.6) node[anchor=north east] {\dotlabel{F(X)}} to[out=45,in=down] (0.5,0);
        \draw[wipe] (0,-0.6) -- (0,0);
        \draw (0,-0.6) node[anchor=north] {\dotlabel{M}} -- (0,0);
        \draw (0,0) -- (0,0.6) node[anchor=south] {\dotlabel{N}};
        \draw[wipe] (0.5,0) to[out=up,in=-45] (-0.5,0.6);
        \draw (0.5,0) to[out=up,in=-45] (-0.5,0.6);
        \coupon{0,0}{g};
        \coupon[0.27]{0.5,0}{\theta^{-1}};
    \end{tikzpicture}
    \ .
\]

\begin{proof}
    Let $W,X,Y,Z$ be objects, and let $f \colon Y \to Z$, $f' \colon Z \to W$, $g$, and $g'$ be morphisms in $\cC$.  It is straightforward to verify that
    \begin{gather*}
        (\xi_X \cdot g) \circ (\xi_X^{-1} \cdot g')
        = 1_X \cdot (g \circ g')
        = (\xi_X^{-1} \cdot g) \circ (\xi_X \cdot g')
        \quad \text{and}
        \\
        (f \cdot g) \circ (f' \cdot g')
        = (f \circ f') \cdot (g \circ g')
    \end{gather*}
    for all objects $X$ in $\cC$, morphisms $f,f'$ in $\cC$, and morphisms $g,g'$ in $\cM$ such that the above compositions are defined.

    It remains to show that \cref{gecko,snake} respect the first, third, and fourth relations in \cref{jaguar+}.  Since $f \cdot g = (f \cdot 1) \circ (1 \cdot g)$ for $f \in \Mor(\Aff(\cC))$ and $g \in \Mor(\cM)$, it suffices to consider the action on identity morphisms.  For the first relation in \cref{jaguar+}, we compute
    \[
        \begin{tikzpicture}[centerzero]
            \draw (-0.3,-0.3) -- (0.3,0.3);
            \draw[wipe] (0.3,-0.3) -- (-0.3,0.3);
            \draw (0.3,-0.3) -- (-0.3,0.3);
            \posdot{-0.18,-0.18};
        \end{tikzpicture}
        \cdot 1_M
        =
        \begin{tikzpicture}[anchorbase]
            \draw (0.5,0) -- (0.5,0.6);
            \draw (1,0) node[anchor=north] {\dotlabel{M}} -- (1,0.6);
            \draw[wipe] (0,0) to[out=45,in=down] (1.5,0.6);
            \draw (0,0) to[out=45,in=down] (1.5,0.6) to[out=up,in=down] (0,1.4) \braidup (0.5,1.8);
            \draw[wipe] (0.5,0.6) -- (0.5,1.2) \braidup (0,1.8);
            \draw (0.5,0.6) -- (0.5,1.2) \braidup (0,1.8);
            \draw[wipe] (1,0.6) -- (1,1.8);
            \draw (1,0.6) -- (1,1.8);
            \coupon{1.5,0.6}{\theta};
        \end{tikzpicture}
        =
        \begin{tikzpicture}[anchorbase]
            \draw (0.5,0) -- (0.5,0.6);
            \draw (1,0) node[anchor=north] {\dotlabel{M}} -- (1,0.6);
            \draw[wipe] (0,0) to[out=45,in=down] (1.5,0.6);
            \draw (0,0) to[out=45,in=down] (1.5,0.6) to[out=up,in=down] (0.5,1.5) -- (0.5,1.8);
            \draw[wipe] (0.5,0.6) \braidup (0,1.2) -- (0,1.8);
            \draw (0.5,0.6) \braidup (0,1.2) -- (0,1.8);
            \draw[wipe] (1,0.6) -- (1,1.8);
            \draw (1,0.6) -- (1,1.8);
            \coupon{1.5,0.6}{\theta};
        \end{tikzpicture}
        =
        \begin{tikzpicture}[centerzero]
            \draw (0.3,-0.3) -- (-0.3,0.3);
            \draw[wipe] (-0.3,-0.3) -- (0.3,0.3);
            \draw (-0.3,-0.3) -- (0.3,0.3);
            \posdot{0.15,0.15};
        \end{tikzpicture}
        \cdot 1_M.
    \]
    For the third relation, setting $X' = F(X)$ and $Y' = F(Y)$, we have
    \[
        \left(
            \begin{tikzpicture}[anchorbase]
                \draw (0.2,-0.4) node[anchor=north] {\dotlabel{Y}} -- (0.2,0.4);
                \draw (-0.2,-0.4) node[anchor=north] {\dotlabel{X}} -- (-0.2,0.4);
                \posdot{-0.2,0};
                \posdot{0.2,0};
            \end{tikzpicture}
        \right)
        \cdot 1_Z
        =
        \begin{tikzpicture}[anchorbase]
            \draw (0.25,0) -- (0.25,0.6);
            \draw[wipe] (-0.25,0) \braidup (0.75,0.6);
            \draw (-0.25,0) node[anchor=north] {\dotlabel{Y'}} \braidup (0.75,0.6) \braidup (-0.25,1.2) -- (-0.25,1.8);
            \draw[wipe] (0.25,0.6) -- (0.25,1.8);
            \draw (0.25,0.6) -- (0.25,1.8);
            \draw[wipe] (-0.75,0.8) \braidup (0.75,1.6);
            \draw (-0.75,0) node[anchor=north] {\dotlabel{X'}} -- (-0.75,0.8) \braidup (0.75,1.6) \braidup (-0.75,2.4);
            \draw[wipe] (-0.25,1.8) -- (-0.25,2.4);
            \draw (-0.25,1.8) -- (-0.25,2.4);
            \draw[wipe] (0.25,1.8) -- (0.25,2.4);
            \draw (0.25,1.8) -- (0.25,2.4) node[anchor=south] {\dotlabel{Z}};
            \coupon{0.75,0.6}{\theta};
            \coupon{0.75,1.6}{\theta};
        \end{tikzpicture}
        =
        \begin{tikzpicture}[anchorbase]
            \draw (0,0) -- (0,1.2);
            \draw[wipe] (-0.5,0) to[out=45,in=down] (0.7,0.7);
            \draw (-0.5,0) node[anchor=north] {\dotlabel{Y'}} to[out=45,in=down] (0.7,0.7) \braidup (0.5,1.2);
            \draw[wipe] (-1,0) \braidup (1,1.2);
            \draw (-1,0) node[anchor=north] {\dotlabel{X'}} \braidup (1,1.2) \braidup (-1,2.4);
            \draw[wipe] (0.5,1.2) \braidup (0.7,1.8);
            \draw (0.5,1.2) \braidup (0.7,1.8) to[out=up,in=-45] (-0.5,2.4);
            \draw[wipe] (0,1.2) -- (0,2.4);
            \draw (0,1.2) -- (0,2.4) node[anchor=south] {\dotlabel{Z}};
            \coupon{0.5,1.2}{\theta};
            \coupon{1,1.2}{\theta};
        \end{tikzpicture}
        \overset{\cref{swirl}}{=}
        \begin{tikzpicture}[centerzero]
            \draw (0,-0.6) -- (0,0);
            \draw[wipe] (-0.5,-0.6) to[out=45,in=down] (0.5,0);
            \draw (-0.5,-0.6) node[anchor=north] {\dotlabel{X' \otimes Y'}} to[out=45,in=down] (0.5,0);
            \draw (0.5,0) to[out=up,in=-45] (-0.5,0.6);
            \draw[wipe] (0,0) -- (0,0.6);
            \draw (0,0) -- (0,0.6) node[anchor=south] {\dotlabel{Z}};
            \coupon{0.5,0}{\theta};
        \end{tikzpicture}
        =
        \begin{tikzpicture}[centerzero]
            \draw (0,-0.4) node[anchor=north] {\dotlabel{X \otimes Y}} -- (0,0.4);
            \posdot{0,0};
        \end{tikzpicture}
        \cdot 1_Z.
    \]
    The fourth relation in \cref{jaguar+} is straightforward to verify using \cref{spin}.
\end{proof}

The following result relates affinization of a monoidal category to the representation theoretic approach to affinization mentioned in the introduction.

\begin{cor} \label{gopher}
    Suppose $\cC$ is a braided strict monoidal category, $\cM$ is a balanced strict monoidal category, and $F \colon \cC \to \cM$ is a braided monoidal functor.  Let $A \colon \cC \to \cEnd(\cM)$ be the action functor of \cref{viper}.
    \begin{enumerate}
        \item For each $X \in \Ob(\cC)$, the collection $\Xi_X := (\xi_X \cdot 1_M)_{M \in \Ob(\cM)}$ is an endomorphism of the functor $A(X) \colon \cM \to \cM$ (i.e.\ a natural transformation from $A(X)$ to itself).
        \item The collection $(\Xi_X)_{X \in \Ob(\cC)}$ is an endomorphism of $A$.
    \end{enumerate}
\end{cor}

\begin{proof}
    \begin{enumerate}[wide]
        \item For a morphism $g \colon M \to N$ in $\cM$, we have
            \[
                \big( A(X)(g) \big) \circ (\xi_X \cdot 1_M)
                =
                \begin{tikzpicture}[anchorbase]
                    \draw (0,-0.6) node[anchor=north] {\dotlabel{M}} -- (0,0);
                    \draw[wipe] (-0.5,-0.6) to[out=45,in=down] (0.5,0);
                    \draw (-0.5,-0.6) node[anchor=north east] {\dotlabel{F(X)}} to[out=45,in=down] (0.5,0);
                    \draw (0.5,0) to[out=up,in=down] (-0.5,0.8) -- (-0.5,1);
                    \draw[wipe] (0,0) -- (0,1);
                    \draw (0,0) -- (0,1) node[anchor=south] {\dotlabel{N}};
                    \coupon{0,0.7}{g};
                    \coupon{0.5,0}{\theta};
                \end{tikzpicture}
                =
                \begin{tikzpicture}[anchorbase]
                    \draw (0,-1) node[anchor=north] {\dotlabel{M}} -- (0,0);
                    \draw[wipe] (-0.5,-0.6) to[out=45,in=down] (0.5,0);
                    \draw (-0.5,-1) node[anchor=north east] {\dotlabel{F(X)}} -- (-0.5,-0.8) to[out=up,in=down] (0.5,0);
                    \draw (0.5,0) to[out=up,in=-45] (-0.5,0.6);
                    \draw[wipe] (0,0) -- (0,0.6);
                    \draw (0,0) -- (0,0.6) node[anchor=south] {\dotlabel{N}};
                    \coupon{0,-0.7}{g};
                    \coupon{0.5,0}{\theta};
                \end{tikzpicture}
                =
                (\xi_X \cdot 1_N) \circ \big( A(X)(g) \big).
            \]
            Hence $\Xi_X$ is a natural transformation from $A(X)$ to itself.

        \item This follows from \cref{salamander} and the last relation in \cref{jaguar+}.
    \end{enumerate}
\end{proof}

For the remainder of this section, suppose that $\cC$ is a balanced strict monoidal category.  Taking $\cM = \cC$ and $F = \id_\cC$ to be the identity functor, \cref{salamander} implies that $\Aff(\cC)$ acts on $\cC$.  This extends the natural action of $\cC$ on itself given by the tensor product.  Considering the action on the unit object and its identity morphism, it then follows that we have a functor
\begin{equation} \label{flatten}
    \Aff(\cC) \to \cC,\quad
    X \mapsto X,\quad
    f \mapsto f \cdot 1_\one,\quad
    X \in \Ob(\Aff(\cC)),\ f \in \Mor(\Aff(\cC)).
\end{equation}
Note that, under the functor \cref{flatten}, we have
\[
    \xi_X \mapsto \theta_X,\quad
    \xi_{X,Y} \mapsto \beta_{X,Y} \circ (1_X \otimes \theta_Y),
\]
where we use \cref{bootstrap}.

\begin{cor} \label{inject}
    If $\cC$ is a balanced strict monoidal category, then the functor $\cC \to \Aff(\cC)$ that is the identity on objects and sends $f \in \Mor(\cC)$ to the corresponding morphism in $\Aff(\cC)$ is faithful.
\end{cor}

\begin{proof}
    It is straightforward to verify that the composite $\cC \to \Aff(\cC) \xrightarrow{\cref{flatten}} \cC$ is the identity functor, which implies the result.
\end{proof}

In light of \cref{inject}, we will view $\cC$ as a subcategory of $\Aff(\cC)$.  We expect that the natural functor $\cC \to \Aff(\cC)$ is faithful even when $\cC$ is not balanced.

Recall that if $\cC$ and $\cD$ are $\kk$-linear categories, then $\cC \boxtimes \cD$ is the $\kk$-linear category whose objects are pairs $(X,Y)$ with $X \in \Ob(\cC)$ and $Y \in \Ob(\cD)$.  Morphisms are given by
\[
    \Hom_{\cC \boxtimes \cD}((X_1,Y_1),(X_2,Y_2))
    = \Hom_\cC(X_1,X_2) \otimes_\kk \Hom_\cD(Y_1,Y_2).
\]
Composition is componentwise on simple tensors, and extended by linearity.

If $\cC$ is a strict $\kk$-linear monoidal category, then a \emph{module category} over $\cC$ is a $\kk$-linear category $\cM$, together with a $\kk$-linear monoidal functor $\cC \to \cEnd_\kk(\cM)$, where $\cEnd_\kk(\cM)$ denotes the strict $\kk$-linear monoidal category with objects that are the $\kk$-linear endofunctors of $\cM$ and morphisms that are natural transformations.  Equivalently, it is a $\kk$-linear functor $- \otimes - \colon \cC \boxtimes \cM \to \cM$ satisfying associativity and unity axioms.  All of the results of the current section go through in this linear setting with the obvious modifications.

\section{Pivotal structures}

We continue to assume throughout this section that $\cC$ is a strict monoidal category.  Recall that $\cC$ is said to be \emph{right rigid} if all objects have right duals.  This means that, for every object $X \in \Ob(\cC)$, there is another object $X^\vee \in \Ob(\cC)$, called the \emph{right dual} of $X$, together with unit and counit morphisms
\begin{equation} \label{rcps}
    \eta_X =
    \begin{tikzpicture}[centerzero={0,-0.15}]
        \draw[->] (-0.3,0) to[out=down,in=down,looseness=2] (0.3,0) node[anchor=south] {\dotlabel{X}};
    \end{tikzpicture}
    \ \colon \one \to X^\vee \otimes X,\qquad
    \epsilon_X =
    \begin{tikzpicture}[centerzero={0,0.15}]
        \draw[->] (-0.3,0) node[anchor=north] {\dotlabel{X}} to[out=up,in=up,looseness=2] (0.3,0);
    \end{tikzpicture}
    \ \colon X \otimes X^\vee \to \one,
\end{equation}
such that
\begin{equation} \label{zigright}
    \begin{tikzpicture}[centerzero={0.3,0.5}]
        \draw[<-] (0.6,0) node[anchor=north] {\dotlabel{X}} to (0.6,0.5) to[out=up,in=up,looseness=2] (0.3,0.5) to[out=down,in=down,looseness=2] (0,0.5) to (0,1);
    \end{tikzpicture}
    \ =\
    \begin{tikzpicture}[centerzero={0,0.5}]
        \draw[<-] (0,0) node[anchor=north] {\dotlabel{X}} to (0,1);
    \end{tikzpicture}
    \ ,\quad
    \begin{tikzpicture}[centerzero={0.3,0.5}]
        \draw[<-] (0.6,1) to (0.6,0.5) to[out=down,in=down,looseness=2] (0.3,0.5) to[out=up,in=up,looseness=2] (0,0.5) to (0,0) node[anchor=north] {\dotlabel{X}};
    \end{tikzpicture}
    \ =\
    \begin{tikzpicture}[centerzero={0,0.5}]
        \draw[->] (0,0) node[anchor=north] {\dotlabel{X}} to (0,1);
    \end{tikzpicture}
    \ .
\end{equation}
Here we use an upward oriented string to denote the identity morphism of an object $X$ and a downward oriented string to denote the identity morphism of its right dual $X^\vee$.

A strict monoidal category $\cC$ is \emph{left rigid} if all objects have left duals.  This means that, for every object $X \in \Ob(\cC)$, there is another object $\leftdual{X} \in \Ob(\cC)$, called the \emph{left dual} of $X$, together with morphisms
\begin{equation} \label{lcps}
    \eta_X' =
    \begin{tikzpicture}[centerzero={0,-0.15}]
        \draw[<-] (-0.3,0) node[anchor=south] {\dotlabel{X}} to[out=down,in=down,looseness=2] (0.3,0);
    \end{tikzpicture}
    \ \colon \one \to X \otimes \leftdual{X},\qquad
    \epsilon_X' =
    \begin{tikzpicture}[centerzero={0,0.15}]
        \draw[<-] (-0.3,0) to[out=up,in=up,looseness=2] (0.3,0) node[anchor=north] {\dotlabel{X}};
    \end{tikzpicture}
    \ \colon \leftdual{X} \otimes X \to \one,
\end{equation}
such that
\begin{equation} \label{zigleft}
    \begin{tikzpicture}[centerzero={0.3,0.5}]
        \draw[->] (0.6,0) node[anchor=north] {\dotlabel{X}} to (0.6,0.5) to[out=up,in=up,looseness=2] (0.3,0.5) to[out=down,in=down,looseness=2] (0,0.5) to (0,1);
    \end{tikzpicture}
    \ =\
    \begin{tikzpicture}[centerzero={0,0.5}]
        \draw[->] (0,0) node[anchor=north] {\dotlabel{X}} to (0,1);
    \end{tikzpicture}
    \ ,\quad
    \begin{tikzpicture}[centerzero={0.3,0.5}]
        \draw[->] (0.6,1) to (0.6,0.5) to[out=down,in=down,looseness=2] (0.3,0.5) to[out=up,in=up,looseness=2] (0,0.5) to (0,0) node[anchor=north] {\dotlabel{X}};
    \end{tikzpicture}
    \ =\
    \begin{tikzpicture}[centerzero={0,0.5}]
        \draw[<-] (0,0) node[anchor=north] {\dotlabel{X}} to (0,1);
    \end{tikzpicture}
    \ .
\end{equation}
Here we use an upward oriented string to denote the identity morphism of an object $X$ and a downward oriented string to denote the identity morphism of its left dual $\leftdual{X}$.

The strict monoidal category $\cC$ is \emph{rigid} if it is both left rigid and right rigid.  If $\cC$ is a \emph{braided} strict monoidal category, then it is left rigid if and only if it is right rigid (hence rigid).  In this case, the left and right duals of $X$ are isomorphic.  Hence there is no ambiguity in using a downward strand labeled $X$ to denote the dual.

\begin{rem} \label{selfdual}
    If $X \in \Ob(\cC)$ is \emph{self-dual}, i.e.\ $X = X^\vee = \leftdual{X}$, then the upward and downward oriented strands are equal, and it is natural to draw these strands without orientation.  In particular, the left and right caps (resp.\ cups) above are equal.
\end{rem}

The following result is straightforward.

\begin{lem}
    If $X$ is right dual (resp.\ left dual) to $Y$ in $\cC$, then the same is true in $\Aff(\cC)$, using the same units and counits.  In particular, if $\cC$ is left rigid (resp.\ right rigid, rigid) then so is $\Aff(\cC)$.
\end{lem}

The category $\cC$ is \emph{strict pivotal} if it is a rigid strict monoidal category and we have the following:
\begin{enumerate}
    \item For all objects $X$ and $Y$ in $\cC$,
        \[
            (X^\vee)^\vee = X,\quad
            (X \otimes Y)^\vee = Y^\vee \otimes X^\vee,\quad
            \one^\vee = \one.
        \]

    \item For all objects $X$ and $Y$ in $\cC$, we have
        \[
            \begin{tikzpicture}[anchorbase]
                \draw[->] (-0.45,0) node[anchor=south] {\dotlabel{X \otimes Y}} to[out=down,in=down,looseness=2] (0.45,0);
            \end{tikzpicture}
            \ =
            \begin{tikzpicture}[anchorbase]
                \draw[->] (-0.6,0) node[anchor=south] {\dotlabel{X}} to[out=down,in=down,looseness=2] (0.6,0);
                \draw[->] (-0.3,0) node[anchor=south] {\dotlabel{Y}} to[out=down,in=down,looseness=2] (0.3,0);
            \end{tikzpicture}
            \qquad \text{and} \qquad
            \begin{tikzpicture}[anchorbase]
                \draw[->] (-0.45,0) node[anchor=north] {\dotlabel{X \otimes Y}} to[out=up,in=up,looseness=2] (0.45,0);
            \end{tikzpicture}
            \ =
            \begin{tikzpicture}[anchorbase]
                \draw[->] (-0.6,0) node[anchor=north] {\dotlabel{X}} to[out=up,in=up,looseness=2] (0.6,0);
                \draw[->] (-0.3,0) node[anchor=north] {\dotlabel{Y}} to[out=up,in=up,looseness=2] (0.3,0);
            \end{tikzpicture}
            \ .
      \]

    \item For every morphism $f \colon X \to Y$ in $\cC$, its right and left mates are equal:
        \[
            f^\vee :=
            \begin{tikzpicture}[anchorbase]
                \draw[<-] (0.4,-0.4) node[anchor=north] {\dotlabel{Y}} to (0.4,0) to[out=up,in=up,looseness=2] (0,0) to[out=down,in=down,looseness=2] (-0.4,0) to (-0.4,0.4) node[anchor=south] {\dotlabel{X}};
                \coupon{0,0}{f};
            \end{tikzpicture}
            \ =\
            \begin{tikzpicture}[anchorbase]
                \draw[<-] (-0.4,-0.4) node[anchor=north] {\dotlabel{Y}} to (-0.4,0) to[out=up,in=up,looseness=2] (0,0) to[out=down,in=down,looseness=2] (0.4,0) to (0.4,0.4) node[anchor=south] {\dotlabel{X}};
                \coupon{0,0}{f};
            \end{tikzpicture}
            \ .
        \]
\end{enumerate}

A braided pivotal category is the same as a balanced rigid category; see, for example, \cite[Cor.~4.21]{Sel11}.  In particular, if $\cC$ is a braided strict pivotal category, then it has a twist given by
\begin{equation} \label{looptwist}
    \theta_X :=
    \begin{tikzpicture}[centerzero]
    	\draw[<-] (0,0.6) to (0,0.3);
    	\draw (0.3,-0.2) to [out=0,in=-90](.5,0);
    	\draw (0.5,0) to [out=90,in=0](.3,0.2);
    	\draw (0,-0.3) to (0,-0.6) node[anchor=north] {\dotlabel{X}};
    	\draw (0,0.3) to [out=-90,in=180] (.3,-0.2);
    	\draw[wipe] (0.3,.2) to [out=180,in=90](0,-0.3);
    	\draw (0.3,.2) to [out=180,in=90](0,-0.3);
    \end{tikzpicture}
    \ ,\quad X \in \Ob(\cC).
\end{equation}
A \emph{ribbon category} (also called a \emph{tortile category}) is a braided pivotal category satisfying
\begin{equation}
    \begin{tikzpicture}[centerzero]
    	\draw[<-] (0,0.6) to (0,0.3);
    	\draw (-0.3,-0.2) to[out=180,in=-90](-0.5,0);
    	\draw (-0.5,0) to[out=90,in=180] (-0.3,0.2);
    	\draw (0,-0.3) to (0,-0.6) node[anchor=north] {\dotlabel{X}};
    	\draw (-0.3,.2) to[out=0,in=90] (0,-0.3);
    	\draw[wipe] (0,0.3) to[out=-90,in=0] (-0.3,-0.2);
    	\draw (0,0.3) to[out=-90,in=0] (-0.3,-0.2);
    \end{tikzpicture}
    \ =\
    \begin{tikzpicture}[centerzero]
    	\draw[<-] (0,0.6) to (0,0.3);
    	\draw (0.3,-0.2) to [out=0,in=-90](.5,0);
    	\draw (0.5,0) to [out=90,in=0](0.3,0.2);
    	\draw (0,-0.3) to (0,-0.6) node[anchor=north] {\dotlabel{X}};
    	\draw (0,0.3) to [out=-90,in=180] (0.3,-0.2);
    	\draw[wipe] (0.3,.2) to [out=180,in=90](0,-0.3);
    	\draw (0.3,.2) to [out=180,in=90](0,-0.3);
    \end{tikzpicture}
    \ ,\quad X \in \Ob(\cC).
\end{equation}
We say a category is a \emph{strict ribbon category} if is it a ribbon category that is strict pivotal.

\begin{theo} \label{rainbow}
    If $\cC$ is a braided strict pivotal category, then the same units and counits endow $\Aff(\cC)$ with the structure of a strict pivotal category.  Furthermore, we have
    \begin{equation} \label{dotrot}
        \xi_X^\vee = \xi_{X^\vee}^{-1},\quad
        X \in \Ob(\cC).
    \end{equation}
\end{theo}

\begin{proof}
    The left and right mates of the morphisms of $\cC$ are equal since $\cC$ is strict pivotal.  Thus, to show that $\Aff(\cC)$ is strict pivotal, it suffices to show that the left and right mates of the positive dots are equal.  (It then automatically follows that the left and right mates of the negative dots are equal, since they are inverse to the positive dots.)  So it is enough to show that
    \begin{equation} \label{dotrotpic}
        \begin{tikzpicture}[centerzero]
            \draw[<-] (0.4,-0.4) node[anchor=north] {\dotlabel{X}} to (0.4,0) to[out=up,in=up,looseness=2] (0,0) to[out=down,in=down,looseness=2] (-0.4,0) to (-0.4,0.4);
            \posdot{0,0};
        \end{tikzpicture}
        \ =\
        \begin{tikzpicture}[centerzero]
            \draw[<-] (-0.4,-0.4) node[anchor=north] {\dotlabel{X}} to (-0.4,0) to[out=up,in=up,looseness=2] (0,0) to[out=down,in=down,looseness=2] (0.4,0) to (0.4,0.4);
            \posdot{0,0};
        \end{tikzpicture}
        \ =\
        \begin{tikzpicture}[anchorbase]
            \draw[<-] (0,-0.3) -- (0,0.3);
            \negdot{0,0};
        \end{tikzpicture}
        \ .
    \end{equation}
    To illustrate the two viewpoints, topological and algebraic, we give two proofs of these identities.

    For a topological proof, we compute
    \begin{gather*}
        \begin{tikzpicture}[centerzero]
            \draw[<-] (0.4,-0.5) node[anchor=north] {\dotlabel{X}} to (0.4,0.1) to[out=up,in=up,looseness=2] (0,0.1) -- (0,-0.1) to[out=down,in=down,looseness=2] (-0.4,-0.1) to (-0.4,0.5);
            \posdot{0,0};
            \draw (-0.7,-0.5) -- (-0.7,0.5);
            \draw (-1,-0.5) -- (-1,0.5);
            \draw (-1.3,-0.5) -- (-1.3,0.5);
            \draw (-1.6,-0.5) -- (-1.6,0.5);
            \draw (0.7,-0.5) -- (0.7,0.5);
            \draw (1,-0.5) -- (1,0.5);
            \draw (1.3,-0.5) -- (1.3,0.5);
        \end{tikzpicture}
        \ =\
        \begin{tikzpicture}[centerzero]
            \draw[<-] (0.4,-0.5) node[anchor=north] {\dotlabel{X}} -- (0.4,0.1) arc(0:90:0.2) to[out=west,in=east] (-1.9,0);
            \draw[wipe] (-0.7,-0.5) -- (-0.7,0.5);
            \draw[wipe] (-1,-0.5) -- (-1,0.5);
            \draw[wipe] (-1.3,-0.5) -- (-1.3,0.5);
            \draw[wipe] (-1.6,-0.5) -- (-1.6,0.5);
            \draw (-0.7,-0.5) -- (-0.7,0.5);
            \draw (-1,-0.5) -- (-1,0.5);
            \draw (-1.3,-0.5) -- (-1.3,0.5);
            \draw (-1.6,-0.5) -- (-1.6,0.5);
            \draw (0.7,-0.5) -- (0.7,0.5);
            \draw (1,-0.5) -- (1,0.5);
            \draw (1.3,-0.5) -- (1.3,0.5);
            \draw[wipe] (-0.4,0.5) -- (-0.4,-0.1) arc(180:270:0.2) to[out=east,in=west] (1.6,0);
            \draw[->] (-0.4,0.5) -- (-0.4,-0.1) arc(180:270:0.2) to[out=east,in=west] (1.6,0);
            \identify{-1.9}{-0.5}{1.6}{0.5};
        \end{tikzpicture}
        \ =\
        \begin{tikzpicture}[centerzero]
            \draw[<-] (0,-0.5) node[anchor=north] {\dotlabel{X}} -- (0,0.5);
            \negdot{0,0};
            \draw (-0.3,-0.5) -- (-0.3,0.5);
            \draw (-0.6,-0.5) -- (-0.6,0.5);
            \draw (-0.9,-0.5) -- (-0.9,0.5);
            \draw (-1.2,-0.5) -- (-1.2,0.5);
            \draw (0.3,-0.5) -- (0.3,0.5);
            \draw (0.6,-0.5) -- (0.6,0.5);
            \draw (0.9,-0.5) -- (0.9,0.5);
        \end{tikzpicture}
        \ ,
        \\
        \begin{tikzpicture}[centerzero]
            \draw[<-] (-0.4,-0.5) node[anchor=north] {\dotlabel{X}} to (-0.4,0.1) to[out=up,in=up,looseness=2] (0,0.1) -- (0,-0.1) to[out=down,in=down,looseness=2] (0.4,-0.1) to (0.4,0.5);
            \posdot{0,0};
            \draw (-0.7,-0.5) -- (-0.7,0.5);
            \draw (-1,-0.5) -- (-1,0.5);
            \draw (-1.3,-0.5) -- (-1.3,0.5);
            \draw (-1.6,-0.5) -- (-1.6,0.5);
            \draw (0.7,-0.5) -- (0.7,0.5);
            \draw (1,-0.5) -- (1,0.5);
            \draw (1.3,-0.5) -- (1.3,0.5);
        \end{tikzpicture}
        \ =\
        \begin{tikzpicture}[centerzero]
            \draw (-0.25,0.1) to[out=west,in=east] (-1.9,0);
            \draw[wipe] (-0.4,-0.5) -- (-0.4,0.25) arc(180:-90:0.15);
            \draw[<-] (-0.4,-0.5) node[anchor=north] {\dotlabel{X}} -- (-0.4,0.25) arc(180:-90:0.15);
            \draw[wipe] (-0.7,-0.5) -- (-0.7,0.5);
            \draw[wipe] (-1,-0.5) -- (-1,0.5);
            \draw[wipe] (-1.3,-0.5) -- (-1.3,0.5);
            \draw[wipe] (-1.6,-0.5) -- (-1.6,0.5);
            \draw (-0.7,-0.5) -- (-0.7,0.5);
            \draw (-1,-0.5) -- (-1,0.5);
            \draw (-1.3,-0.5) -- (-1.3,0.5);
            \draw (-1.6,-0.5) -- (-1.6,0.5);
            \draw (0.7,-0.5) -- (0.7,0.5);
            \draw (1,-0.5) -- (1,0.5);
            \draw (1.3,-0.5) -- (1.3,0.5);
            \draw (0.4,0.5) -- (0.4,-0.25) arc(360:90:0.15);
            \draw[wipe] (0.25,-0.1) to[out=east,in=west] (1.6,0);
            \draw[->] (0.25,-0.1) to[out=east,in=west] (1.6,0);
            \identify{-1.9}{-0.5}{1.6}{0.5};
        \end{tikzpicture}
        \ \overset{\cref{whip}}{=}\
        \begin{tikzpicture}[centerzero]
            \draw (-1.9,0) -- (-0.35,0) arc(90:-180:0.15);
            \draw[wipe] (-0.5,-0.15) to[out=up,in=west] (-0.2,0.15);
            \draw (-0.5,-0.15) to[out=up,in=west] (-0.2,0.15) to[out=east,in=west] (0.15,-0.15) arc(-90:90:0.15);
            \draw[wipe] (0.15,0.15) to[out=west,in=up] (-0.1,-0.5);
            \draw[->] (0.15,0.15) to[out=west,in=up] (-0.1,-0.5) node[anchor=north] {\dotlabel{X}};
            \draw[wipe] (-0.7,-0.5) -- (-0.7,0.5);
            \draw[wipe] (-1,-0.5) -- (-1,0.5);
            \draw[wipe] (-1.3,-0.5) -- (-1.3,0.5);
            \draw[wipe] (-1.6,-0.5) -- (-1.6,0.5);
            \draw (-0.7,-0.5) -- (-0.7,0.5);
            \draw (-1,-0.5) -- (-1,0.5);
            \draw (-1.3,-0.5) -- (-1.3,0.5);
            \draw (-1.6,-0.5) -- (-1.6,0.5);
            \draw (0.7,-0.5) -- (0.7,0.5);
            \draw (1,-0.5) -- (1,0.5);
            \draw (1.3,-0.5) -- (1.3,0.5);
            \draw[wipe] (0.4,0.5) -- (0.4,0.3) to[out=down,in=west] (1.6,0);
            \draw[->] (0.4,0.5) -- (0.4,0.3) to[out=down,in=west] (1.6,0);
            \identify{-1.9}{-0.5}{1.6}{0.5};
        \end{tikzpicture}
        \ =\
        \begin{tikzpicture}[centerzero]
            \draw[<-] (0,-0.5) node[anchor=north] {\dotlabel{X}} -- (0,0.5);
            \negdot{0,0};
            \draw (-0.3,-0.5) -- (-0.3,0.5);
            \draw (-0.6,-0.5) -- (-0.6,0.5);
            \draw (-0.9,-0.5) -- (-0.9,0.5);
            \draw (-1.2,-0.5) -- (-1.2,0.5);
            \draw (0.3,-0.5) -- (0.3,0.5);
            \draw (0.6,-0.5) -- (0.6,0.5);
            \draw (0.9,-0.5) -- (0.9,0.5);
        \end{tikzpicture}
        \ ,
    \end{gather*}
    where, in the last equality, we used the fact that the two curls appearing in the penultimate diagram are inverses of each other (see \cref{spiral}).

    For an algebraic proof of \cref{dotrotpic}, we first note that by attaching appropriate cups and caps and using \cref{zigright,zigleft}, the identities \cref{dotrotpic} are equivalent to
    \begin{equation}
        \begin{tikzpicture}[centerzero]
            \draw[->] (-0.2,-0.25) -- (-0.2,0) arc(180:0:0.2) -- (0.2,-0.25);
            \posdot{-0.2,0};
        \end{tikzpicture}
        =
        \begin{tikzpicture}[centerzero]
            \draw[->] (-0.2,-0.25) -- (-0.2,0) arc(180:0:0.2) -- (0.2,-0.25);
            \negdot{0.2,0};
        \end{tikzpicture}
        \ ,\quad
        \begin{tikzpicture}[centerzero]
            \draw[<-] (-0.2,-0.25) -- (-0.2,0) arc(180:0:0.2) -- (0.2,-0.25);
            \posdot{0.2,0};
        \end{tikzpicture}
        =
        \begin{tikzpicture}[centerzero]
            \draw[<-] (-0.2,-0.25) -- (-0.2,0) arc(180:0:0.2) -- (0.2,-0.25);
            \negdot{-0.2,0};
        \end{tikzpicture}
        \ .
    \end{equation}
    We prove the first identity, since the proof of the second is analogous.  Adding a positive dot to the bottom-right strand, we see that it suffices to prove
    \[
        \begin{tikzpicture}[centerzero]
            \draw[->] (-0.2,-0.25) node[anchor=north] {\dotlabel{X}} -- (-0.2,0) arc(180:0:0.2) -- (0.2,-0.25);
            \posdot{-0.2,0};
            \posdot{0.2,0};
        \end{tikzpicture}
        =
        \begin{tikzpicture}[centerzero]
            \draw[->] (-0.2,-0.25) node[anchor=north] {\dotlabel{X}} -- (-0.2,0) arc(180:0:0.2) -- (0.2,-0.25);
        \end{tikzpicture}
        \ .
    \]
    By the third relation in \cref{jaguar+}, the left-hand side is the composite of a single positive dot on a strand labelled $X \otimes X^\vee$ and the counit (the right cap).  Then, by the fourth relation in \cref{jaguar+}, we can slide this positive dot above the counit.  Since $\xi_\one = 1_\one$, this gives the right-hand side above.
\end{proof}

If $\cC$ is a braided strict pivotal category and $X \in \Ob(\cC)$, then we define the invertible \emph{dots}
\begin{equation}
    \begin{tikzpicture}[centerzero]
        \draw[->] (0,-0.3) node[anchor=north] {\dotlabel{X}} -- (0,0.3);
        \opendot{0,0};
    \end{tikzpicture}
    :=
    \begin{tikzpicture}[centerzero]
        \draw[->] (0,-0.3) node[anchor=north] {\dotlabel{X}} -- (0,0.3);
        \posdot{0,0};
    \end{tikzpicture}
    \ ,\qquad
    \begin{tikzpicture}[centerzero]
        \draw[<-] (0,-0.3) node[anchor=north] {\dotlabel{X}} -- (0,0.3);
        \opendot{0,0};
    \end{tikzpicture}
    :=
    \begin{tikzpicture}[centerzero]
        \draw[<-] (0,-0.3) node[anchor=north] {\dotlabel{X}} -- (0,0.3);
        \negdot{0,0};
    \end{tikzpicture}
    \ .
\end{equation}
It then follows from \cref{dotrotpic} that dots slide over cups and caps:
\begin{equation} \label{porcupine}
    \begin{tikzpicture}[centerzero]
        \draw[->] (-0.2,-0.2) -- (-0.2,0) arc(180:0:0.2) -- (0.2,-0.2);
        \opendot{-0.2,0};
    \end{tikzpicture}
    =
    \begin{tikzpicture}[centerzero]
        \draw[->] (-0.2,-0.2) -- (-0.2,0) arc(180:0:0.2) -- (0.2,-0.2);
        \opendot{0.2,0};
    \end{tikzpicture}
    \ ,\quad
    \begin{tikzpicture}[centerzero]
        \draw[->] (-0.2,0.2) -- (-0.2,0) arc(180:360:0.2) -- (0.2,0.2);
        \opendot{-0.2,0};
    \end{tikzpicture}
    =
    \begin{tikzpicture}[centerzero]
        \draw[->] (-0.2,0.2) -- (-0.2,0) arc(180:360:0.2) -- (0.2,0.2);
        \opendot{0.2,0};
    \end{tikzpicture}
    \ ,\quad
    \begin{tikzpicture}[centerzero]
        \draw[<-] (-0.2,-0.2) -- (-0.2,0) arc(180:0:0.2) -- (0.2,-0.2);
        \opendot{-0.2,0};
    \end{tikzpicture}
    =
    \begin{tikzpicture}[centerzero]
        \draw[<-] (-0.2,-0.2) -- (-0.2,0) arc(180:0:0.2) -- (0.2,-0.2);
        \opendot{0.2,0};
    \end{tikzpicture}
    \ ,\quad
    \begin{tikzpicture}[centerzero]
        \draw[<-] (-0.2,0.2) -- (-0.2,0) arc(180:360:0.2) -- (0.2,0.2);
        \opendot{-0.2,0};
    \end{tikzpicture}
    =
    \begin{tikzpicture}[centerzero]
        \draw[<-] (-0.2,0.2) -- (-0.2,0) arc(180:360:0.2) -- (0.2,0.2);
        \opendot{0.2,0};
    \end{tikzpicture}
    \ .
\end{equation}

If $X$ is self-dual, so that $X = X^\vee$, the above convention for open dots is still valid, but not very useful in practice.  As noted in \cref{selfdual}, one typically denotes the identity of a self-dual object by an unoriented strand, in which case we cannot introduce the open dot as above.  Instead, we have
\begin{equation} \label{wolverine}
    \begin{tikzpicture}[centerzero]
        \draw (-0.2,-0.2) -- (-0.2,0) arc(180:0:0.2) -- (0.2,-0.2);
        \posdot{-0.2,0};
    \end{tikzpicture}
    =
    \begin{tikzpicture}[centerzero]
        \draw (-0.2,-0.2) -- (-0.2,0) arc(180:0:0.2) -- (0.2,-0.2);
        \negdot{0.2,0};
    \end{tikzpicture}
    \ ,\quad
    \begin{tikzpicture}[centerzero]
        \draw (-0.2,-0.2) -- (-0.2,0) arc(180:0:0.2) -- (0.2,-0.2);
        \negdot{-0.2,0};
    \end{tikzpicture}
    =
    \begin{tikzpicture}[centerzero]
        \draw (-0.2,-0.2) -- (-0.2,0) arc(180:0:0.2) -- (0.2,-0.2);
        \posdot{0.2,0};
    \end{tikzpicture}
    \ ,\quad
    \begin{tikzpicture}[centerzero]
        \draw (-0.2,0.2) -- (-0.2,0) arc(180:360:0.2) -- (0.2,0.2);
        \posdot{-0.2,0};
    \end{tikzpicture}
    =
    \begin{tikzpicture}[centerzero]
        \draw (-0.2,0.2) -- (-0.2,0) arc(180:360:0.2) -- (0.2,0.2);
        \negdot{0.2,0};
    \end{tikzpicture}
    \ ,\quad
    \begin{tikzpicture}[centerzero]
        \draw (-0.2,0.2) -- (-0.2,0) arc(180:360:0.2) -- (0.2,0.2);
        \negdot{-0.2,0};
    \end{tikzpicture}
    =
    \begin{tikzpicture}[centerzero]
        \draw (-0.2,0.2) -- (-0.2,0) arc(180:360:0.2) -- (0.2,0.2);
        \posdot{0.2,0};
    \end{tikzpicture}
    \ .
\end{equation}

\begin{rem}
    Many of our categories will have generating morphisms given by braidings for the generating objects and possibly cups and caps coming from dual objects.  We will want to give a presentation for the affinization of such a category $\cC$, following \cref{plain}, by adjoining the $\xi_X$ for $X$ ranging over the set of generating objects.  Then $\xi_X$ is defined on all objects using the third relation in \cref{jaguar+}.  We impose the first relation in \cref{jaguar+} (and then the second relation follows) and then it remains to impose the fourth relation in \cref{jaguar+} with the coupon there ranging over the generating morphisms of $\cC$.  When the coupon is a braiding on generating objects $X$ and $Y$, we have, using the first two relations in \cref{jaguar+}
    \[
        \begin{tikzpicture}[centerzero]
            \draw (-0.3,-0.3) node[anchor=north] {\dotlabel{X}} -- (0.3,0.3);
            \draw[wipe] (0.3,-0.3) -- (-0.3,0.3);
            \draw (0.3,-0.3) node[anchor=north] {\dotlabel{Y}} -- (-0.3,0.3);
            \posdot{-0.18,-0.18};
            \posdot{0.18,-0.18};
        \end{tikzpicture}
        =
        \begin{tikzpicture}[centerzero]
            \draw (0.3,-0.3) node[anchor=north] {\dotlabel{Y}} -- (-0.3,0.3);
            \draw[wipe] (-0.3,-0.3) -- (0.3,0.3);
            \draw (-0.3,-0.3) node[anchor=north] {\dotlabel{X}} -- (0.3,0.3);
            \posdot{0.15,0.15};
            \posdot{0.18,-0.18};
        \end{tikzpicture}
        =
        \begin{tikzpicture}[centerzero]
            \draw (-0.3,-0.3) node[anchor=north] {\dotlabel{X}} -- (0.3,0.3);
            \draw[wipe] (0.3,-0.3) -- (-0.3,0.3);
            \draw (0.3,-0.3) node[anchor=north] {\dotlabel{Y}} -- (-0.3,0.3);
            \posdot{-0.18,0.18};
            \posdot{0.18,0.18};
        \end{tikzpicture}
        \ .
    \]
    So the fourth relation in \cref{jaguar+} is automatically satisfied when the coupon is a braiding.  If the coupon is a right cap, then, as explained in the proof of \cref{rainbow}, the fourth relation in \cref{jaguar+} is equivalent to
    \[
        \begin{tikzpicture}[centerzero]
            \draw[->] (-0.2,-0.25) node[anchor=north] {\dotlabel{X}} -- (-0.2,0) arc(180:0:0.2) -- (0.2,-0.25);
            \posdot{-0.2,0};
        \end{tikzpicture}
        =
        \begin{tikzpicture}[centerzero]
            \draw[->] (-0.2,-0.25) node[anchor=north] {\dotlabel{X}} -- (-0.2,0) arc(180:0:0.2) -- (0.2,-0.25);
            \negdot{0.2,0};
        \end{tikzpicture}
        \ ,
    \]
    and similarly for left caps.  Thus, it suffices to impose the cap-slide relations in \cref{porcupine} (for non-self-dual objects) or \cref{wolverine} (for self-dual objects); the corresponding cup relations then follow using \cref{zigright,zigleft}.
\end{rem}

\section{Tangles\label{sec:tangles}}

Many of the examples to be discussed in the sequel will be constructed from categories of tangles.  In this section, we fix our conventions and recall some concepts that will be common to many of these examples.

Let
\[
    D = (0,1)^2
    \qquad \text{and} \qquad
    A = S^1 \times (0,1)
\]
be the disc and the annulus, respectively.  In order to make various categories of tangles strict, we need to fix a countable number of points in $D$, which will be the possible endpoints of our tangles.  We choose the points
\begin{equation} \label{endpoints}
    P_n = \left( 1 - \tfrac{1}{2^n}, \tfrac{1}{2} \right),\quad n \in \Z_{>0}.
\end{equation}
We make the identification
\begin{equation} \label{lasso}
    A = ([0,1] \times (0,1))/\sim,
\end{equation}
where $\sim$ is the relation given by $(0,b) \sim (1,b)$ for all $b \in (0,1)$, and we also view the $P_n$ as points in $A$.  Up to isomorphism, our categories will not depend on the particular choice of points.  We will typically draw them as equally spaced, or adjust the spacing to the particular tangle we draw.

We let $\FOT(D)$ be the category of framed oriented tangles over $D$.  Its objects are finite sequences $(\varepsilon_1,\dotsc,\varepsilon_n)$ of elements of $\{\uparrow,\downarrow\}$.  The unit object $\one$ is the empty sequence.  Morphisms in $\FOT(D)$ from $(\varepsilon_1,\dotsc,\varepsilon_m)$ to $(\varepsilon'_1,\dotsc,\varepsilon'_n)$ are framed oriented tangles in $D \times [0,1]$, up to ambient isotopy, with endpoints
\[
    \left( \{P_1,\dotsc,P_m\} \times \{0\} \right) \cup \left( \{P_1,\dotsc,P_n\} \times \{1\} \right)
\]
such that the orientation of the tangle at each $P_i \times \{0\}$ agrees with $\varepsilon_i$, the orientation at each $P_i' \times \{1\}$ agrees with $\varepsilon'_i$, and the framing at the point $P_i \times \{0\}$ (respectively, $P_i \times \{1\}$) points towards $P_{i+1} \times \{0\}$ (respectively, $P_{i+1} \times \{1\}$).  We allow tangles to have closed components.  For example,
\[
    \begin{tikzpicture}[anchorbase]
        \draw[->] (0.5,-0.6) to[out=up,in=up,looseness=2] (1.5,-0.6);
        \draw[wipe] (1,-0.6) -- (1,-0.3) to[out=up,in=down] (-0.5,1);
        \draw[->] (1,-0.6) -- (1,-0.3) to[out=up,in=down] (-0.5,1);
    	\draw[wipe] (0,1) to (0,0.3);
    	\draw[<-] (0,1) to (0,0.3);
    	\draw (0.3,-0.2) to[out=0,in=-90](.5,0);
    	\draw (0.5,0) to[out=90,in=0](0.3,0.2);
    	\draw (0,-0.3) to (0,-0.6);
    	\draw (0,0.3) to[out=-90,in=180] (0.3,-0.2);
    	\draw[wipe] (0.3,0.2) to[out=180,in=90](0,-0.3);
    	\draw (0.3,0.2) to[out=180,in=90](0,-0.3);
        \draw (1.2,0.6) to[out=up,in=up,looseness=1.5] (1.6,0.6);
        \draw[wipe] (0.5,1) to[out=down,in=down,looseness=2] (1.5,1);
        \draw[->] (0.5,1) to[out=down,in=down,looseness=2] (1.5,1);
        \draw[wipe] (1.2,0.6) to[out=down,in=down,looseness=1.5] (1.6,0.6);
        \draw[->] (1.2,0.6) to[out=down,in=down,looseness=1.5] (1.6,0.6);
        \draw[->] (-0.5,-0.6) to[out=up,in=up,looseness=2] (-1,-0.6);
    \end{tikzpicture}
    \in \Hom_{\FOT(D)}((\downarrow,\uparrow,\uparrow,\uparrow,\uparrow,\downarrow), (\uparrow,\uparrow,\downarrow,\uparrow)),
\]
where we adopt the convention of blackboard framing (i.e.\ the framing is parallel to the page).  The composite $f \circ g$ is given by placing $f$ above $g$ and rescaling the vertical coordinate.  The category $\FOT(D)$ is a strict monoidal category, with tensor product $f \otimes g$ given by placing the $f$ to the left of $g$ and rescaling.  Objects can be written as (possibly empty) tensor products of the objects $\uparrow$ and $\downarrow$.

We let $\OT(D)$ denote the strict monoidal category of oriented tangles over $D$.  This is defined as above, but without the framing.  Similarly, forgetting orientations, we let $\FT(D)$ and $\T(D)$ denote the strict monoidal categories of framed tangles over $D$ and tangles over $D$, respectively.  Here the objects are natural numbers, since we have no orientations of the endpoints.

Replacing the disc $D$ by the annulus $A$, we obtain the categories $\FOT(A)$, $\OT(A)$, $\FT(A)$, and $\T(A)$ of framed oriented tangles, oriented tangles, framed tangles, and tangles over the annulus.  We draw these by cutting along $\{(0,b,c) : b \in (0,1),\ c \in [0,1]\}$ (see \cref{lasso}) in order to draw, for example,
\begin{equation} \label{pizza}
    \begin{tikzpicture}[anchorbase]
        \draw[->] (0.5,-0.6) to[out=up,in=225] (1.8,0.2);
        \draw[->] (-1.3,0.2) to[out=45,in=down] (-1,1);
        \draw[wipe] (1,-0.6) -- (1,-0.3) to[out=up,in=down] (-0.5,1);
        \draw[<-] (1,-0.6) -- (1,-0.3) to[out=up,in=down] (-0.5,1);
    	\draw[wipe] (0,1) to (0,0.3);
    	\draw[<-] (0,1) to (0,0.3);
    	\draw (0.3,-0.2) to[out=0,in=-90](.5,0);
    	\draw (0.5,0) to[out=90,in=0](0.3,0.2);
    	\draw (0,-0.3) to (0,-0.6);
    	\draw (0,0.3) to[out=-90,in=180] (0.3,-0.2);
    	\draw[wipe] (0.3,0.2) to[out=180,in=90](0,-0.3);
    	\draw (0.3,0.2) to[out=180,in=90](0,-0.3);
        \draw (1.2,0.6) to[out=up,in=up,looseness=1.5] (1.6,0.6);
        \draw[wipe] (0.5,1) to[out=down,in=down,looseness=2] (1.5,1);
        \draw[->] (0.5,1) to[out=down,in=down,looseness=2] (1.5,1);
        \draw[wipe] (1.2,0.6) to[out=down,in=down,looseness=1.5] (1.6,0.6);
        \draw[->] (1.2,0.6) to[out=down,in=down,looseness=1.5] (1.6,0.6);
        \identify{-1.3}{-0.6}{1.8}{1};
    \end{tikzpicture}
    \in \Hom_{\FOT(A)}(\uparrow \otimes \uparrow \otimes \downarrow, \uparrow \otimes \downarrow \otimes \uparrow \otimes \downarrow \otimes \uparrow),
\end{equation}
where we identify the dashed vertical edges.  We always isotope tangles so that they intersect the cut transversely.

The categories $\FOT(A)$, $\OT(A)$, $\FT(A)$, and $\T(A)$ are also strict monoidal categories, although some care must be taken with the tensor product.  Viewing $A \times [0,1]$ as the cylinder, the tensor product $f \otimes g$ is given by placing the cylinder for $g$ inside the cylinder for $f$, then rescaling and isotoping the endpoints of the tangles so that the endpoints of $g$ are to the right of those of $f$ (preserving the relative order of the endpoints in $f$ and the endpoints in $g$).  In terms of diagrams as in \cref{pizza}, this corresponds to placing the diagram of $g$ to the right of the diagram of $f$, and then passing all strands of $f$ exiting the right side of its diagram over the diagram for $g$ and all strands of $g$ exiting the left side of its diagram under the diagram for $f$.  For example,
\[
    \begin{tikzpicture}[centerzero]
        \identify{-0.7}{-0.5}{0.7}{0.5};
        \draw[->] (-0.3,-0.5) \braidup (0.3,0.5);
        \draw[->] (0.3,-0.5) to[out=up,in=200] (0.7,0);
        \draw[->] (-0.7,0) to[out=20,in=down] (-0.3,0.5);
    \end{tikzpicture}
    \ \otimes\
    \begin{tikzpicture}[anchorbase]
        \identify{-0.7}{-0.5}{0.7}{0.5};
        \draw[->] (0,-0.5) \braidup (-0.3,0.5);
        \draw[<-] (0.3,-0.5) \braidup (0,0.5);
        \draw[->] (-0.3,-0.5) to[out=up,in=-20] (-0.7,0);
        \draw[->] (0.7,0) to[out=160,in=down] (0.3,0.5);
    \end{tikzpicture}
    \ =\
    \begin{tikzpicture}[centerzero]
        \identify{-0.9}{-0.5}{0.9}{0.5};
        \draw[->] (0,-0.5) to[out=up,in=-45] (-0.9,-0.15);
        \draw[->] (0.3,-0.5) \braidup (0,0.5);
        \draw[<-] (0.6,-0.5) \braidup (0.3,0.5);
        \draw[->] (0.9,-0.15) to[out=135,in=down] (0.6,0.5);
        \draw[wipe] (-0.3,-0.5) to[out=up,in=225] (0.9,0.15);
        \draw[->] (-0.3,-0.5) to[out=up,in=225] (0.9,0.15);
        \draw[->] (-0.9,0.15) to[out=45,in=down] (-0.6,0.5);
        \draw[wipe] (-0.6,-0.5) \braidup (-0.3,0.5);
        \draw[->] (-0.6,-0.5) \braidup (-0.3,0.5);
    \end{tikzpicture}
    \ .
\]

The category $\Braid(D)$ of braids over $D$ is the strict monoidal subcategory of $\T(D)$ whose morphisms have no closed components.  Equivalently, $\Braid(D)$ is the strict monoidal subcategory of $\OT(D)$ with objects generated by the object $\uparrow$ and morphisms containing no closed components.  We define $\Braid(A)$ similarly.

For a commutative ring $\kk$, we use a subscript $\kk$ to denote the $\kk$-linearization of a strict monoidal category.  For example, $\FOT(D)_\kk$ is the $\kk$-linearization of $\FOT(D)$.

Many of the examples to be introduced in the sections to follow share some conventions and defining relations.  In order to make our presentation efficient, we introduce these first here.

The ``oriented'' categories to follow will be generated by two objects, $\uparrow$ and $\downarrow$, whose identity morphisms we denote by an upward and downward strand, respectively.  Morphisms will be string diagrams with oriented strings.  The domain and codomain of such diagrams can be read from the orientations of the strands at the bottom and top of the diagram in the usual way.  The ``unoriented'' categories will be generated by a single object $\go$, whose identity morphism we denote by an unoriented strand.  String diagrams will then involve unoriented strings.

The following relations will be important (here $z,t,\delta \in \kk$):
\begin{gather}\tag{R0} \label{R0}
    \begin{tikzpicture}[centerzero={0.3,0.5}]
        \draw (0.6,0) -- (0.6,0.5) to[out=up,in=up,looseness=2] (0.3,0.5) to[out=down,in=down,looseness=2] (0,0.5) -- (0,1);
    \end{tikzpicture}
    \ =\
    \begin{tikzpicture}[centerzero={0,0.5}]
        \draw (0,0) -- (0,1);
    \end{tikzpicture}
    \ =\
    \begin{tikzpicture}[centerzero={0.3,0.5}]
        \draw (0.6,1) -- (0.6,0.5) to[out=down,in=down,looseness=2] (0.3,0.5) to[out=up,in=up,looseness=2] (0,0.5) -- (0,0);
    \end{tikzpicture}
    \ ,\qquad
    \begin{tikzpicture}[centerzero]
        \draw (-0.2,-0.3) -- (-0.2,-0.1) arc(180:0:0.2) -- (0.2,-0.3);
        \draw[wipe] (-0.3,0.3) \braiddown (0,-0.3);
        \draw (-0.3,0.3) \braiddown (0,-0.3);
    \end{tikzpicture}
    =
    \begin{tikzpicture}[centerzero]
        \draw (-0.2,-0.3) -- (-0.2,-0.1) arc(180:0:0.2) -- (0.2,-0.3);
        \draw[wipe] (0.3,0.3) \braiddown (0,-0.3);
        \draw (0.3,0.3) \braiddown (0,-0.3);
    \end{tikzpicture}
    \ ,\qquad
    \begin{tikzpicture}[centerzero]
        \draw (-0.3,0.3) \braiddown (0,-0.3);
        \draw[wipe] (-0.2,-0.3) -- (-0.2,-0.1) arc(180:0:0.2) -- (0.2,-0.3);
        \draw (-0.2,-0.3) -- (-0.2,-0.1) arc(180:0:0.2) -- (0.2,-0.3);
    \end{tikzpicture}
    =
    \begin{tikzpicture}[centerzero]
        \draw (0.3,0.3) \braiddown (0,-0.3);
        \draw[wipe] (-0.2,-0.3) -- (-0.2,-0.1) arc(180:0:0.2) -- (0.2,-0.3);
        \draw (-0.2,-0.3) -- (-0.2,-0.1) arc(180:0:0.2) -- (0.2,-0.3);
    \end{tikzpicture}
    \ ,
    \\ \tag{R1} \label{R1}
    \begin{tikzpicture}[centerzero]
    	\draw (0,0.6) to (0,0.3);
    	\draw (-0.3,-0.2) to[out=180,in=-90](-0.5,0);
    	\draw (-0.5,0) to[out=90,in=180] (-0.3,0.2);
    	\draw (0,-0.3) to (0,-0.6);
    	\draw (-0.3,.2) to[out=0,in=90] (0,-0.3);
    	\draw[wipe] (0,0.3) to[out=-90,in=0] (-0.3,-0.2);
    	\draw (0,0.3) to[out=-90,in=0] (-0.3,-0.2);
    \end{tikzpicture}
    \ =\
    \begin{tikzpicture}[centerzero]
        \draw (0,-0.6) -- (0,0.6);
    \end{tikzpicture}
    \ =\
    \begin{tikzpicture}[centerzero]
    	\draw (0,0.6) to (0,0.3);
    	\draw (0.3,-0.2) to [out=0,in=-90](.5,0);
    	\draw (0.5,0) to [out=90,in=0](0.3,0.2);
    	\draw (0,-0.3) to (0,-0.6);
    	\draw (0,0.3) to [out=-90,in=180] (0.3,-0.2);
    	\draw[wipe] (0.3,.2) to [out=180,in=90](0,-0.3);
    	\draw (0.3,.2) to [out=180,in=90](0,-0.3);
    \end{tikzpicture}
    \ ,
    \\ \tag{FR1} \label{FR1}
    \begin{tikzpicture}[centerzero]
    	\draw (0,0.6) to (0,0.3);
    	\draw (-0.3,-0.2) to[out=180,in=-90](-0.5,0);
    	\draw (-0.5,0) to[out=90,in=180] (-0.3,0.2);
    	\draw (0,-0.3) to (0,-0.6);
    	\draw (-0.3,.2) to[out=0,in=90] (0,-0.3);
    	\draw[wipe] (0,0.3) to[out=-90,in=0] (-0.3,-0.2);
    	\draw (0,0.3) to[out=-90,in=0] (-0.3,-0.2);
    \end{tikzpicture}
    \ =\
    \begin{tikzpicture}[centerzero]
    	\draw (0,0.6) to (0,0.3);
    	\draw (0.3,-0.2) to [out=0,in=-90](.5,0);
    	\draw (0.5,0) to [out=90,in=0](0.3,0.2);
    	\draw (0,-0.3) to (0,-0.6);
    	\draw (0,0.3) to [out=-90,in=180] (0.3,-0.2);
    	\draw[wipe] (0.3,.2) to [out=180,in=90](0,-0.3);
    	\draw (0.3,.2) to [out=180,in=90](0,-0.3);
    \end{tikzpicture}
    \ ,
    \\ \tag{R2} \label{R2}
    \begin{tikzpicture}[centerzero]
        \draw (0.25,-0.5) to[out=135,in=down] (-0.2,0) to[out=up,in=225] (0.25,0.5);
        \draw[wipe] (-0.25,-0.5) to[out=45,in=down] (0.2,0) to[out=up,in=-45] (-0.25,0.5);
        \draw (-0.25,-0.5) to[out=45,in=down] (0.2,0) to[out=up,in=-45] (-0.25,0.5);
    \end{tikzpicture}
    \ =\
    \begin{tikzpicture}[centerzero]
        \draw (-0.2,-0.5) -- (-0.2,0.5);
        \draw (0.2,-0.5) -- (0.2,0.5);
    \end{tikzpicture}
    \ =\
    \begin{tikzpicture}[centerzero]
        \draw (-0.25,-0.5) to[out=45,in=down] (0.2,0) to[out=up,in=-45] (-0.25,0.5);
        \draw[wipe] (0.25,-0.5) to[out=135,in=down] (-0.2,0) to[out=up,in=225] (0.25,0.5);
        \draw (0.25,-0.5) to[out=135,in=down] (-0.2,0) to[out=up,in=225] (0.25,0.5);
    \end{tikzpicture}
    \ ,
    \\ \tag{R3} \label{R3}
    \begin{tikzpicture}[centerzero]
        \draw (0.5,-0.5) -- (-0.5,0.5);
        \draw[wipe] (0,-0.5) to[out=135,in=down] (-0.4,0) to[out=up,in=225] (0,0.5);
        \draw (0,-0.5) to[out=135,in=down] (-0.4,0) to[out=up,in=225] (0,0.5);
        \draw[wipe] (-0.5,-0.5) -- (0.5,0.5);
        \draw (-0.5,-0.5) -- (0.5,0.5);
    \end{tikzpicture}
    \ =\
    \begin{tikzpicture}[centerzero]
        \draw (0.5,-0.5) -- (-0.5,0.5);
        \draw[wipe] (0,-0.5) to[out=45,in=down] (0.4,0) to[out=up,in=-45] (0,0.5);
        \draw (0,-0.5) to[out=45,in=down] (0.4,0) to[out=up,in=-45] (0,0.5);
        \draw[wipe] (-0.5,-0.5) -- (0.5,0.5);
        \draw (-0.5,-0.5) -- (0.5,0.5);
    \end{tikzpicture}
    \ ,
    \\ \tag{{$\text{KS}_+$}} \label{KS+}
    \begin{tikzpicture}[centerzero]
        \draw (0.3,-0.4) -- (-0.3,0.4);
        \draw[wipe] (-0.3,-0.4) -- (0.3,0.4);
        \draw (-0.3,-0.4) -- (0.3,0.4);
    \end{tikzpicture}
    +
    \begin{tikzpicture}[centerzero]
        \draw (-0.3,-0.4) -- (0.3,0.4);
        \draw[wipe] (0.3,-0.4) -- (-0.3,0.4);
        \draw (0.3,-0.4) -- (-0.3,0.4);
    \end{tikzpicture}
    = z\
    \begin{tikzpicture}[anchorbase]
        \draw (-0.2,-0.4) -- (-0.2,0.4);
        \draw (0.2,-0.4) -- (0.2,0.4);
    \end{tikzpicture}
    + z\
    \begin{tikzpicture}[anchorbase]
        \draw (-0.2,0.4) -- (-0.2,0.3) arc(180:360:0.2) -- (0.2,0.4);
        \draw (-0.2,-0.4) -- (-0.2,-0.3) arc(180:0:0.2) -- (0.2,-0.4);
    \end{tikzpicture}
    \ ,
    \\ \tag{{$\text{KS}_-$}} \label{KS-}
    \begin{tikzpicture}[centerzero]
        \draw (0.3,-0.4) -- (-0.3,0.4);
        \draw[wipe] (-0.3,-0.4) -- (0.3,0.4);
        \draw (-0.3,-0.4) -- (0.3,0.4);
    \end{tikzpicture}
    -
    \begin{tikzpicture}[centerzero]
        \draw (-0.3,-0.4) -- (0.3,0.4);
        \draw[wipe] (0.3,-0.4) -- (-0.3,0.4);
        \draw (0.3,-0.4) -- (-0.3,0.4);
    \end{tikzpicture}
    = z\
    \begin{tikzpicture}[anchorbase]
        \draw (-0.2,-0.4) -- (-0.2,0.4);
        \draw (0.2,-0.4) -- (0.2,0.4);
    \end{tikzpicture}
    - z\
    \begin{tikzpicture}[anchorbase]
        \draw (-0.2,0.4) -- (-0.2,0.3) arc(180:360:0.2) -- (0.2,0.4);
        \draw (-0.2,-0.4) -- (-0.2,-0.3) arc(180:0:0.2) -- (0.2,-0.4);
    \end{tikzpicture}
    \ ,
    \\ \tag{KB} \label{KB}
    \begin{tikzpicture}[centerzero]
        \draw (0.3,-0.4) -- (-0.3,0.4);
        \draw[wipe] (-0.3,-0.4) -- (0.3,0.4);
        \draw (-0.3,-0.4) -- (0.3,0.4);
    \end{tikzpicture}
    \ := q\
    \begin{tikzpicture}[centerzero]
        \draw (-0.2,-0.4) -- (-0.2,0.4);
        \draw (0.2,-0.4) -- (0.2,0.4);
    \end{tikzpicture}
    \ + q^{-1}\
    \begin{tikzpicture}[centerzero]
        \draw (-0.2,0.4) -- (-0.2,0.3) arc(180:360:0.2) -- (0.2,0.4);
        \draw (-0.2,-0.4) -- (-0.2,-0.3) arc(180:0:0.2) -- (0.2,-0.4);
    \end{tikzpicture}
    \ ,
    \\ \tag{CS} \label{CS}
    \posupcross - \negupcross = z\
    \begin{tikzpicture}[anchorbase]
        \draw[->] (-0.15,-0.2) -- (-0.15,0.2);
        \draw[->] (0.15,-0.2) -- (0.15,0.2);
    \end{tikzpicture}
    \ ,
    \\ \tag{T} \label{T}
    \begin{tikzpicture}[centerzero]
    	\draw (0,0.6) to (0,0.3);
    	\draw (0.3,-0.2) to [out=0,in=-90](.5,0);
    	\draw (0.5,0) to [out=90,in=0](0.3,0.2);
    	\draw (0,-0.3) to (0,-0.6);
    	\draw (0,0.3) to [out=-90,in=180] (0.3,-0.2);
    	\draw[wipe] (0.3,.2) to [out=180,in=90](0,-0.3);
    	\draw (0.3,.2) to [out=180,in=90](0,-0.3);
    \end{tikzpicture}
    = t\
    \begin{tikzpicture}[centerzero]
        \draw (0,-0.6) -- (0,0.6);
    \end{tikzpicture}
    \ ,
    \\ \tag{D} \label{D}
    \begin{tikzpicture}[centerzero]
        \bubun{0,0};
    \end{tikzpicture}
    = \delta 1_\one.
    \\ \tag{UA} \label{UA}
    \begin{tikzpicture}[centerzero]
        \draw (-0.3,-0.3) -- (0.3,0.3);
        \draw[wipe] (0.3,-0.3) -- (-0.3,0.3);
        \draw (0.3,-0.3) -- (-0.3,0.3);
        \posdot{-0.18,-0.18};
    \end{tikzpicture}
    =
    \begin{tikzpicture}[centerzero]
        \draw (0.3,-0.3) -- (-0.3,0.3);
        \draw[wipe] (-0.3,-0.3) -- (0.3,0.3);
        \draw (-0.3,-0.3) -- (0.3,0.3);
        \posdot{0.15,0.15};
    \end{tikzpicture}
    \ ,\qquad
    \begin{tikzpicture}[centerzero]
        \draw (-0.2,-0.2) -- (-0.2,0) arc(180:0:0.2) -- (0.2,-0.2);
        \posdot{-0.2,0};
    \end{tikzpicture}
    =
    \begin{tikzpicture}[centerzero]
        \draw (-0.2,-0.2) -- (-0.2,0) arc(180:0:0.2) -- (0.2,-0.2);
        \negdot{0.2,0};
    \end{tikzpicture}
    \ ,\qquad
    \begin{tikzpicture}[centerzero]
        \draw (0,-0.3) -- (0,0.3);
        \posdot{0,0.12};
        \negdot{0,-0.12};
    \end{tikzpicture}
    =
    \begin{tikzpicture}[centerzero]
        \draw (0,-0.3) -- (0,0.3);
        \negdot{0,0.12};
        \posdot{0,-0.12};
    \end{tikzpicture}
    =
    \begin{tikzpicture}[centerzero]
        \draw (0,-0.3) -- (0,0.3);
    \end{tikzpicture}
    \ ,
    \\ \tag{OA} \label{OA}
    \begin{tikzpicture}[centerzero]
        \draw[->] (-0.3,-0.3) -- (0.3,0.3);
        \draw[wipe] (0.3,-0.3) -- (-0.3,0.3);
        \draw[->] (0.3,-0.3) -- (-0.3,0.3);
        \opendot{-0.18,-0.18};
    \end{tikzpicture}
    =
    \begin{tikzpicture}[centerzero]
        \draw[->] (0.3,-0.3) -- (-0.3,0.3);
        \draw[wipe] (-0.3,-0.3) -- (0.3,0.3);
        \draw[->] (-0.3,-0.3) -- (0.3,0.3);
        \opendot{0.15,0.15};
    \end{tikzpicture}
    \ ,\qquad
    \begin{tikzpicture}[centerzero]
        \draw[->] (-0.2,-0.2) -- (-0.2,0) arc(180:0:0.2) -- (0.2,-0.2);
        \opendot{-0.2,0};
    \end{tikzpicture}
    =
    \begin{tikzpicture}[centerzero]
        \draw[->] (-0.2,-0.2) -- (-0.2,0) arc(180:0:0.2) -- (0.2,-0.2);
        \opendot{0.2,0};
    \end{tikzpicture}
    \ ,\qquad
    \begin{tikzpicture}[centerzero]
        \draw[<-] (-0.2,-0.2) -- (-0.2,0) arc(180:0:0.2) -- (0.2,-0.2);
        \opendot{-0.2,0};
    \end{tikzpicture}
    =
    \begin{tikzpicture}[centerzero]
        \draw[<-] (-0.2,-0.2) -- (-0.2,0) arc(180:0:0.2) -- (0.2,-0.2);
        \opendot{0.2,0};
    \end{tikzpicture}
    \ ,\qquad
    \dotup \text{ is invertible}.
\end{gather}
We will also use oriented versions of \cref{R0}, \cref{R1}, \cref{FR1}, \cref{R2}, and \cref{R3}, referring to them by the same labels.  We will interpret the above both as relations in strict monoidal categories described by generators and relations, and also as relations on ($\kk$-linearizations of) tangle categories over the disc or annulus.  When viewing them as relations on tangle categories, they are local relations drawn using the conventions outlined earlier in this section.  The relations \cref{KS+}, \cref{KS-}, \cref{KB}, and \cref{CS} are called the \emph{Kauffman skein relation}, \emph{Dubrovnik skein relation} (or \emph{Kauffman skein relation in its Dubrovnik form}), \emph{Kauffman bracket skein relation} and \emph{Conway skein relation}, respectively.

We now recall some other relations that follow from various combinations of the above.  First, note that \cref{R0} implies the ``windmill relation'':
\begin{equation} \tag{W} \label{W}
    \cref{R0}
    \implies
    \begin{tikzpicture}[centerzero]
        \draw (-0.2,-0.2) -- (0.2,0.2);
        \draw[wipe] (0.2,-0.2) --  (-0.2,0.2);
        \draw (0.2,-0.2) --  (-0.2,0.2);
    \end{tikzpicture}
    =
    \begin{tikzpicture}[centerzero]
        \draw (0.4,0.3) -- (0.4,0.1) to[out=down,in=right] (0.2,-0.2) to[out=left,in=right] (-0.2,0.2) to[out=left,in=up] (-0.4,-0.1) -- (-0.4,-0.3);
        \draw[wipe] (-0.2,-0.3) \braidup (0.2,0.3);
        \draw (-0.2,-0.3) \braidup (0.2,0.3);
    \end{tikzpicture}
    \ ,\quad
    \begin{tikzpicture}[centerzero]
        \draw (0.2,-0.2) --  (-0.2,0.2);
        \draw[wipe] (-0.2,-0.2) -- (0.2,0.2);
        \draw (-0.2,-0.2) -- (0.2,0.2);
    \end{tikzpicture}
    =
    \begin{tikzpicture}[centerzero]
        \draw (-0.2,-0.3) \braidup (0.2,0.3);
        \draw[wipe] (0.4,0.3) -- (0.4,0.1) to[out=down,in=right] (0.2,-0.2) to[out=left,in=right] (-0.2,0.2) to[out=left,in=up] (-0.4,-0.1) -- (-0.4,-0.3);
        \draw (0.4,0.3) -- (0.4,0.1) to[out=down,in=right] (0.2,-0.2) to[out=left,in=right] (-0.2,0.2) to[out=left,in=up] (-0.4,-0.1) -- (-0.4,-0.3);
    \end{tikzpicture}
    \ .
\end{equation}

Also, it is a straightforward exercise to see that
\begin{equation} \label{spiral}
    \cref{R0},\cref{R2},\cref{R3}
    \implies
    \begin{tikzpicture}[centerzero]
    	\draw (0,0.6) to (0,0.3);
    	\draw (-0.3,-0.2) to[out=180,in=-90](-0.5,0);
    	\draw (-0.5,0) to[out=90,in=180] (-0.3,0.2);
    	\draw (0,-0.3) to (0,-0.6);
    	\draw (0,0.3) to[out=-90,in=0] (-0.3,-0.2);
    	\draw[wipe] (-0.3,.2) to[out=0,in=90] (0,-0.3);
    	\draw (-0.3,.2) to[out=0,in=90] (0,-0.3);
    \end{tikzpicture}
    \ =\
    \left(
        \begin{tikzpicture}[centerzero]
        	\draw (0,0.6) to (0,0.3);
        	\draw (0.3,-0.2) to [out=0,in=-90](.5,0);
        	\draw (0.5,0) to [out=90,in=0](0.3,0.2);
        	\draw (0,-0.3) to (0,-0.6);
        	\draw (0,0.3) to [out=-90,in=180] (0.3,-0.2);
        	\draw[wipe] (0.3,.2) to [out=180,in=90](0,-0.3);
        	\draw (0.3,.2) to [out=180,in=90](0,-0.3);
        \end{tikzpicture}
    \right)^{-1},
    \quad
    \begin{tikzpicture}[centerzero]
    	\draw (0,0.6) to (0,0.3);
    	\draw (-0.3,-0.2) to[out=180,in=-90](-0.5,0);
    	\draw (-0.5,0) to[out=90,in=180] (-0.3,0.2);
    	\draw (0,-0.3) to (0,-0.6);
    	\draw (-0.3,.2) to[out=0,in=90] (0,-0.3);
    	\draw[wipe] (0,0.3) to[out=-90,in=0] (-0.3,-0.2);
    	\draw (0,0.3) to[out=-90,in=0] (-0.3,-0.2);
    \end{tikzpicture}
    \ =\
    \left(
        \begin{tikzpicture}[centerzero]
        	\draw (0,0.6) to (0,0.3);
        	\draw (0.3,-0.2) to [out=0,in=-90](.5,0);
        	\draw (0.5,0) to [out=90,in=0](0.3,0.2);
        	\draw (0,-0.3) to (0,-0.6);
        	\draw (0.3,.2) to [out=180,in=90](0,-0.3);
        	\draw[wipe] (0,0.3) to [out=-90,in=180] (0.3,-0.2);
        	\draw (0,0.3) to [out=-90,in=180] (0.3,-0.2);
        \end{tikzpicture}
    \right)^{-1},
\end{equation}
and this implication also holds for the strands oriented in either direction.  Thus, if \cref{R0}, \cref{R2}, and \cref{R3} hold, then relations \cref{FR1} \and \cref{T} suffice to straighten all twists, in both the oriented and unoriented settings.

Next we consider bubbles.  Relations \cref{R0} and \cref{R2} imply that the bubble is strictly central:
\[
    \cref{R0} + \cref{R2}
    \implies
    \begin{tikzpicture}[anchorbase]
        \bubun{-0.4,0};
        \draw (0,-0.3) -- (0,0.3);
    \end{tikzpicture}
    =
    \begin{tikzpicture}[anchorbase]
        \bubun{0.4,0};
        \draw (0,-0.3) -- (0,0.3);
    \end{tikzpicture}
\]
(The same implication holds in the oriented case.)  Hence it is natural to impose the dimension relation \cref{D}.  In fact, adjoining an indeterminate $\delta$ to the coefficient ring and then imposing \cref{D} simply corresponds to viewing the bubble as an element of the coefficient ring.  In the oriented case, \cref{FR1,T} imply that the clockwise and counterclockwise bubbles are equal:
\begin{equation} \label{hoppy}
    \cref{FR1} + \cref{T}
    \implies
    \begin{tikzpicture}[centerzero]
        \bubright{0,0};
    \end{tikzpicture}
    = t^{-1}\
    \begin{tikzpicture}[centerzero]
        \draw (0.5,0) arc(360:270:0.2) to[out=left,in=right] (-0.3,0.2) arc(90:270:0.2);
        \draw[wipe] (-0.3,-0.2) to[out=right,in=left] (0.3,0.2);
        \draw[->] (-0.3,-0.2) to[out=right,in=left] (0.3,0.2) arc(90:0:0.2);
    \end{tikzpicture}
    =
    \begin{tikzpicture}[centerzero]
        \bubleft{0,0};
    \end{tikzpicture}
    \ .
\end{equation}
Thus, in the presence of \cref{FR1,T}, imposing \cref{D} for the clockwise bubble automatically means \cref{D} is also satisfied for the counterclockwise bubble.

\section{Towers of algebras\label{sec:towers}}

Many families of algebras appearing in representation theory can be combined into a \emph{tower}.  We discuss some of these families here, using the language of strict monoidal categories.  We assume that all categories are linear over a commutative ground ring $\kk$.

Suppose $\cC$ is a strict monoidal category such that
\begin{itemize}
    \item the objects of $\cC$ are generated by a single object $X$, and
    \item we have $\Hom_\cC(X^{\otimes n}, X^{\otimes m}) = 0$ when $m \ne n$.
\end{itemize}
For each $n \in \N$, we have the endomorphism algebra $\cC(n) := \End_\cC(X^{\otimes n})$.  The collection $\cC(n)$, $n \in \N$, is sometimes called a \emph{tower of algebras} in the literature.

If $\cC$ is balanced, then, under the functor \cref{flatten}, the elements $1_X^{\otimes (n-i)} \otimes \xi_X \otimes 1_X^{\otimes (i-1)}$, $1 \le i \le n$, are mapped to
\begin{equation} \label{JMdef}
    J_{i,n} := 1_X^{\otimes (n-i)} \otimes \left( \beta_{X^{\otimes (i-1)},X} \circ (1_X^{\otimes (i-1)} \otimes \theta_X) \circ \beta_{X,X^{\otimes (i-1)}} \right) \in \cC(n).
\end{equation}
For fixed $n$, the elements $J_{1,n}, \dotsc, J_{n,n}$ are pairwise commuting elements of $\cC(n)$.  We call them the \emph{Jucys--Murphy elements} of $\cC(n)$ since, in our examples below, they will correspond to classical Jucys--Murphy elements.

In fact, in many of the examples in this section, we have $\theta_X = 1_X$.  Then the twist on arbitrary objects $X^{\otimes n}$ is determined recursively by \cref{swirl}.  In this case, we have \begin{equation} \label{JMeasy}
    J_{i,n} = 1_X^{\otimes (n-i)} \otimes \left( \beta_{X^{\otimes (i-1)},X} \circ \beta_{X,X^{\otimes (i-1)}} \right).
\end{equation}

\subsection{Category of braids\label{subsec:braids}}

The category $\Braid(D)$ of braids over the disc is isomorphic to the strict monoidal category generated by a single object $\uparrow$, and morphisms
\[
    \posupcross,\ \negupcross,
\]
subject to the relations \cref{R2} and \cref{R3}, with the strands oriented upwards.  (The orientation of the strands does not play an important role here.)  We will identify these two categories.  It follows that $\Braid(D)$ is the free braided monoidal category generated by a single object; see, for example, \cite[Th.~5.6]{Yet01}.  It is balanced, with twist determined by $\theta_\uparrow = 1_\uparrow$.  The endomorphism algebra $\End_{\Braid(D)}(\uparrow^{\otimes n})$ is the group algebra of the braid group of type $A_{n-1}$ and the elements $J_{i,n}$ of \cref{JMeasy} are the usual Jucys--Murphy elements.

\begin{prop} \label{annularbraids}
    The affinization $\Aff(\Braid(D))$ of the category of braids over the disc is isomorphic to the category $\Braid(A)$ of braids over the annulus.
\end{prop}

\begin{proof}
    Consider the functor $F \colon \Aff(\Braid(D)) \to \Braid(A)$ that is the identity on objects, and is given on morphisms as follows.  Given a braid $f \in \Braid(D)$, we can naturally view it as a braid over $A$ via \cref{lasso}.  We define
    \[
        F(\xi_\uparrow) =
        \begin{tikzpicture}[centerzero]
            \draw[->] (0,-0.3) to[out=up,in=west] (0.3,0);
            \draw[->] (-0.3,0) to[out=east,in=south] (0,0.3);
            \identify{-0.3}{-0.3}{0.3}{0.3};
        \end{tikzpicture}
        \qquad \text{and} \qquad
        F(\xi_\uparrow^{-1}) =
        \begin{tikzpicture}[centerzero]
            \draw[->] (0,-0.3) to[out=up,in=east] (-0.3,0);
            \draw[->] (0.3,0) to[out=west,in=south] (0,0.3);
            \identify{-0.3}{-0.3}{0.3}{0.3};
        \end{tikzpicture}
        \ ,
    \]
    where we view these as braids over $A$ by identifying the vertical edges, as in \cref{pizza}.  It is straightforward to verify that $F$ respects the relations \cref{coilrel1,coilrel2} and that it is a monoidal functor.

    Now consider the functor $G \colon \Braid(A) \to \Aff(\Braid(D))$ that is the identity on objects, and is given on morphisms as follows.  Given a braid $f$ over the annulus, we cut the annulus as explained in \cref{sec:tangles} to obtain a tangle drawn in a rectangle with vertical edges identified.  After isotoping if necessary, this diagram is a composite of diagrams not intersecting the vertical edges and diagrams of the form
    \begin{equation} \label{spirit}
        \begin{tikzpicture}[centerzero]
            \draw (-0.4,-0.4) \braidup (-0.2,0.4);
            \draw (-0.2,-0.4) \braidup (0,0.4);
            \draw (0,-0.4) \braidup (0.2,0.4);
            \draw (0.2,-0.4) \braidup (0.4,0.4);
            \draw[->] (0.4,-0.4) to[out=up,in=225] (0.6,0);
            \draw[->] (-0.6,0) to[out=45,in=down] (-0.4,0.4);
            \identify{-0.6}{-0.4}{0.6}{0.4};
        \end{tikzpicture}
        \qquad \text{and} \qquad
        \begin{tikzpicture}[centerzero]
            \draw (-0.2,-0.4) \braidup (-0.4,0.4);
            \draw (0,-0.4) \braidup (-0.2,0.4);
            \draw (0.2,-0.4) \braidup (0,0.4);
            \draw (0.4,-0.4) \braidup (0.2,0.4);
            \draw[->] (-0.4,-0.4) to[out=up,in=-45] (-0.6,0);
            \draw[->] (0.6,0) to[out=135,in=down] (0.4,0.4);
            \identify{-0.6}{-0.4}{0.6}{0.4};
        \end{tikzpicture}
        \ ,
    \end{equation}
    where there can be an arbitrary number of strands (including zero) in the middle (i.e.\ not intersecting the vertical edges of the rectangle).  We then define $G$ on such a morphism by declaring that it sends tangles not intersecting the vertical edges to the same tangles, naturally interpreted as tangles over the disc, and the tangles \cref{spirit} to $\xi_{\uparrow^{\otimes n},\uparrow}$ and $\xi_{\uparrow^{\otimes n},\uparrow}^{-1}$, respectively, where $n$ is the number of strands in the middle of the diagrams (not intersecting the vertical edges).  The relations \cref{coilrel2} ensure that $G$ is well defined, and it is straightforward to verify that $F$ and $G$ are mutually inverse.
\end{proof}

\begin{cor}
    The category $\Braid(A)$ of braids over the annulus is isomorphic to the strict monoidal category generated by a single object $\uparrow$, morphisms $\posupcross$, $\negupcross$, and an invertible morphism $\dotup$, subject to the relations \cref{R2}, \cref{R3}, and the first relation in \cref{OA}.
\end{cor}

The endomorphism algebra $\End_{\Braid(A)}(\uparrow^{\otimes n})$ is the extended affine braid group of type $A_{n-1}$, which is isomorphic to the braid group of type $B_n$.  We refer the reader to \cite[\S 2]{GL03} for further discussion of viewing this braid group in terms of cylindrical braids.

\subsection{Hecke algebras\label{subsec:Hecke}}

Let $\cH(D)$ be the strict $\kk$-linear monoidal category obtained from $\Braid(D)_\kk$ by imposing the Conway skein relation \cref{CS}.  Then the endomorphism algebra $\End_{\cH(D)}(\uparrow^{\otimes n})$ is the Iwahori--Hecke algebra of type $A_{n-1}$.  (One often sees the definition with $z = q-q^{-1}$ for some $q \in \kk^\times$.)

Applying our affinization procedure, we obtain the category $\Aff(\cH(D))$.  By \cref{plain}, this is the strict monoidal category generated by a single object $\uparrow$, and morphisms
\[
    \posupcross,\quad \negupcross,\quad \dotup,
\]
subject to the relations \cref{R2} and \cref{R2} with the strands oriented upwards, the relation \cref{CS}, and
\[
    \begin{tikzpicture}[centerzero]
        \draw[->] (-0.3,-0.3) -- (0.3,0.3);
        \draw[wipe] (0.3,-0.3) -- (-0.3,0.3);
        \draw[->] (0.3,-0.3) -- (-0.3,0.3);
        \opendot{-0.18,-0.18};
    \end{tikzpicture}
    =
    \begin{tikzpicture}[centerzero]
        \draw[->] (0.3,-0.3) -- (-0.3,0.3);
        \draw[wipe] (-0.3,-0.3) -- (0.3,0.3);
        \draw[->] (-0.3,-0.3) -- (0.3,0.3);
        \opendot{0.15,0.15};
    \end{tikzpicture}
    \ ,\qquad
    \dotup \text{ is invertible}.
\]
The endomorphism algebra $\End_{\Aff(\cH(D))}(\uparrow^{\otimes n})$ is the affine Hecke algebra of type $A_{n-1}$.  The category $\cH(D)$ is balanced, with twist determined by $\theta_\uparrow = 1_\uparrow$.  The elements $J_{i,n}$ of \cref{JMeasy} are the usual Jucys--Murphy elements.  An interpretation of affine Hecke algebras in terms of the cylinder was given in \cite{GL03}.

The above discussion can be generalized to the setting of \emph{quantum wreath product algebras}, which were introduced in \cite[Def.~2.1]{RS20} and depend on a Frobenius algebra $A$.  They are the endomorphism algebras of the \emph{quantum wreath product category} defined in \cite[\S2]{BSW-qFrobHeis}.  Affinization of this category yields the \emph{quantum affine wreath product category} defined there, whose endomorphism algebras are the \emph{quantum affine wreath product algebras} defined in \cite[Def.~2.5]{RS20}.  When $A=\kk$, this recovers the case of (affine) Hecke algebras described above.  When $A$ is the group algebra of a finite cyclic group, it corresponds to the (affine) Yokonuma--Hecke algebras.

\subsection{Symmetric groups}

Let $\Sym$ be the free $\kk$-linear symmetric monoidal category on a single object.  Thus $\Sym$ is the strict $\kk$-linear monoidal category generated by a single object $\uparrow$, and morphism $\symcross$, subject to the relations
\[
    \begin{tikzpicture}[centerzero]
        \draw[->] (0.25,-0.5) to[out=135,in=down] (-0.2,0) to[out=up,in=225] (0.25,0.5);
        \draw[->] (-0.25,-0.5) to[out=45,in=down] (0.2,0) to[out=up,in=-45] (-0.25,0.5);
    \end{tikzpicture}
    \ =\
    \begin{tikzpicture}[centerzero]
        \draw[->] (-0.2,-0.5) -- (-0.2,0.5);
        \draw[->] (0.2,-0.5) -- (0.2,0.5);
    \end{tikzpicture}
    \ =\
    \begin{tikzpicture}[centerzero]
        \draw[->] (-0.25,-0.5) to[out=45,in=down] (0.2,0) to[out=up,in=-45] (-0.25,0.5);
        \draw[->] (0.25,-0.5) to[out=135,in=down] (-0.2,0) to[out=up,in=225] (0.25,0.5);
    \end{tikzpicture}
    \quad \text{and} \quad
    \begin{tikzpicture}[centerzero]
        \draw[->] (0.5,-0.5) -- (-0.5,0.5);
        \draw[->] (0,-0.5) to[out=135,in=down] (-0.4,0) to[out=up,in=225] (0,0.5);
        \draw[->] (-0.5,-0.5) -- (0.5,0.5);
    \end{tikzpicture}
    \ =\
    \begin{tikzpicture}[centerzero]
        \draw[->] (0.5,-0.5) -- (-0.5,0.5);
        \draw[->] (0,-0.5) to[out=45,in=down] (0.4,0) to[out=up,in=-45] (0,0.5);
        \draw[->] (-0.5,-0.5) -- (0.5,0.5);
    \end{tikzpicture}
    \ .
\]
Then $\End_{\Sym}(\uparrow^{\otimes n})$ is the group algebra of the symmetric group $\fS_n$.

Since $\Sym$ is a braided monoidal category, we can apply our affinization procedure.  Is it easy to see, using \cref{plain}, that $\End_{\Aff(\Sym)}(\uparrow^{\otimes n})$ is isomorphic to the wreath product algebra $\kk[x_1^{\pm 1},\dotsc,x_n^{\pm n}] \rtimes \fS_n$, where $x_i$ corresponds to a positive dot on the $i$-th strand, and $\fS_n$ acts on $\kk[x_1^{\pm 1},\dotsc,x_n^{\pm 1}]$ by permutation of the $x_i$.

Note that the endomorphism algebras of $\Aff(\Sym)$ are \emph{not} degenerate affine Hecke algebras.  To obtain the latter, one needs to instead consider a $q \to 1$ degeneration of $\Aff(\cH(D))$ from \cref{subsec:Hecke}.

\section{Oriented examples\label{sec:oriented}}

In this section and the next, we give a number of examples of the affinization of monoidal categories coming from the theory of tangles and skein theory.  In each example, we see that affinization of such a category over the disc yields the corresponding category over the annulus.  Furthermore, \cref{plain} gives us presentations of these annular categories involving dot generators.

\subsection{Oriented tangles}

The category $\OT(D)$ of oriented tangles over the disc is isomorphic to the strict monoidal category generated by objects $\uparrow$, $\downarrow$, and morphisms
\[
    \posupcross,\ \negupcross,\ \posrightcross,\ \negrightcross,\ \posdowncross,\ \negdowncross,\ \posleftcross,\ \negleftcross,\ \rightcup,\ \leftcup,\ \rightcap,\ \rightcup,
\]
subject to the relations \cref{R0}, \cref{R1}, \cref{R2}, and \cref{R3} for all orientations of the strands.  This category is a strict ribbon category.

\begin{rem} \label{efficient}
    There exist more efficient presentations.  For instance, using \cref{W}, it is enough to include the upward crossings as generators.  One can then use \cref{W} to \emph{define} the other crossings.  Also, for example, \cref{R1} and \cref{R2} for downward oriented strands follow from the upward oriented strand analogues, together with \cref{R0}.  See \cite[Th.~3.2]{Tur89}, \cite[Th.~3.5]{FY89}, or \cite[Th.~XII.2.2]{Kas95} for details.  We choose to include all the crossings as generators to emphasize the structure as a braided monoidal category.
\end{rem}

\begin{prop} \label{crocodile}
    The affinization $\Aff(\OT(D))$ of the category of oriented tangles over the disc is isomorphic to the category $\OT(A)$ of oriented tangles over the annulus.
\end{prop}

\begin{proof}
    The proof is analogous to that of \cref{annularbraids}.
\end{proof}

\begin{cor}
    The category $\OT(A)$ of oriented tangles over the annulus is isomorphic to the strict monoidal category generated by objects $\uparrow$, $\downarrow$, and morphisms
    \[
        \posupcross,\ \negupcross,\ \posrightcross,\ \negrightcross,\ \posdowncross,\ \negdowncross,\ \posleftcross,\ \negleftcross,\ \rightcup,\ \leftcup,\ \rightcap,\ \rightcup,\ \dotup,\ \dotdown,
    \]
    subject to relations \cref{R0}, \cref{R1}, \cref{R2}, \cref{R3}, and \cref{OA}.
\end{cor}

\subsection{Framed oriented tangles}

The category $\FOT(D)$ of framed oriented tangles over the disc is isomorphic to the strict monoidal category generated by objects $\uparrow$, $\downarrow$, and morphisms
\[
    \posupcross,\ \negupcross,\ \posrightcross,\ \negrightcross,\ \posdowncross,\ \negdowncross,\ \posleftcross,\ \negleftcross,\ \rightcup,\ \leftcup,\ \rightcap,\ \rightcup,
\]
subject to the relations \cref{R0}, \cref{FR1}, \cref{R2}, and \cref{R3} for all orientations of the strands.  As in \cref{efficient}, there are actually more efficient presentations, i.e.\ presentations with fewer generators.  See, for example, \cite[Th.~3.5]{FY89}.  The category $\FOT(D)$ is the ribbon category freely generated by a single object; see \cite[Th.~6.4]{Shu94} and \cite[Th.~9.1]{Yet01}.

\begin{prop} \label{alligator}
    The affinization $\Aff(\FOT(D))$ of the category of framed oriented tangles over the disc is isomorphic to the category $\FOT(A)$ of oriented tangles over the annulus.
\end{prop}

\begin{proof}
    The proof is analogous to that of \cref{annularbraids}.
\end{proof}

\begin{cor}
    The category $\FOT(A)$ of oriented tangles over the annulus is isomorphic to the strict monoidal category generated by objects $\uparrow$, $\downarrow$, and morphisms
    \[
        \posupcross,\ \negupcross,\ \posrightcross,\ \negrightcross,\ \posdowncross,\ \negdowncross,\ \posleftcross,\ \negleftcross,\ \rightcup,\ \leftcup,\ \rightcap,\ \rightcup,\ \dotup,\ \dotdown,
    \]
    subject to relations \cref{R0}, \cref{FR1}, \cref{R2}, \cref{R3}, and \cref{OA}.
\end{cor}

\subsection{HOMFLYPT skein category\label{subsec:HOMFLYPT}}

Fix $z,\delta \in \kk$ and $t \in \kk^\times$ satisfying $z \delta = t-t^{-1}$.  Generically, we can work over the ring $\kk = \Z[z,t,t^{-1},\delta]/(z\delta - t + t^{-1})$.  The \emph{framed HOMFLYPT skein category} $\OS(D;z,t,\delta)$ (resp.\ $\OS(A;z,t,\delta)$) over the disc (resp.\ over the annulus) is the category obtained from $\FOT(D)_\kk$ (resp.\ from $\FOT(A)_\kk$) by imposing the Conway skein relation \cref{CS}, the twist relation \cref{T} for the upward orientation of the strands, and the dimension relation \cref{D} with either orientation of the bubble (see \cref{hoppy}).  The reason for imposing the condition $z\delta = t-t^{-1}$ in our coefficient ring is that we have
\begin{equation} \label{giraffe}
    (t-t^{-1})\
    \begin{tikzpicture}[centerzero]
        \draw[->] (0,-0.6) -- (0,0.6);
    \end{tikzpicture}
    \overset{\cref{T}}{=}
    \begin{tikzpicture}[centerzero]
    	\draw[<-] (0,0.6) to (0,0.3);
    	\draw (0.3,-0.2) to [out=0,in=-90](.5,0);
    	\draw (0.5,0) to [out=90,in=0](0.3,0.2);
    	\draw (0,-0.3) to (0,-0.6);
    	\draw (0,0.3) to [out=-90,in=180] (0.3,-0.2);
    	\draw[wipe] (0.3,.2) to [out=180,in=90](0,-0.3);
    	\draw (0.3,.2) to [out=180,in=90](0,-0.3);
    \end{tikzpicture}
    -
    \begin{tikzpicture}[centerzero]
    	\draw[<-] (0,0.6) to (0,0.3);
    	\draw (0.3,-0.2) to [out=0,in=-90](.5,0);
    	\draw (0.5,0) to [out=90,in=0](0.3,0.2);
    	\draw (0,-0.3) to (0,-0.6);
    	\draw (0.3,.2) to [out=180,in=90](0,-0.3);
    	\draw[wipe] (0,0.3) to [out=-90,in=180] (0.3,-0.2);
    	\draw (0,0.3) to [out=-90,in=180] (0.3,-0.2);
    \end{tikzpicture}
    \overset{\cref{CS}}{=} z\
    \begin{tikzpicture}[centerzero]
        \draw[->] (0,-0.6) -- (0,0.6);
        \bubright{0.4,0};
    \end{tikzpicture}
    \overset{\cref{D}}{=} z \delta\
    \begin{tikzpicture}[centerzero]
        \draw[->] (0,-0.6) -- (0,0.6);
    \end{tikzpicture}
    \ .
\end{equation}
When $z \in \kk^\times$, we have $\delta = (t-t^{-1})/z$, and so we can omit $\delta$ from the notation.  In this case, we denote the category by $\OS(D;z,t)$ (resp.\ $\OS(A;z,t)$).  The category $\OS(D;z,t,\delta)$ was first introduced in \cite[\S5.2]{Tur89}, where it was called the \emph{Hecke category} (not to be confused with the more modern use of this term, which is related to the category of Soergel bimodules).  Our choice of the notation $\OS$ comes from \emph{oriented skein}.

The category $\OS(D;z,t)$ underpins the HOMFLYPT polynomial in the following sense.  Given an oriented link diagram $L$, define $\writhe(L)$ to be the number of positive crossings minus the number of negative crossings in $L$.  Viewing $L$ as an element of $\End_{\OS(D;z,t)}(\one)$, there is a unique scalar $H_L(z,t) \in \kk$ such that $t^{-\writhe(L)} L = H_L(z,t) \delta 1_\one$.  (The factor of $\delta$ appears here to normalize the polynomial so that $H_L(z,t)=1$ when $L$ is the unknot.)  If $\kk = \Z[z,z^{-1},t,t^{-1}]$, then $H_L(z,t)$ is precisely the HOMFLYPT polynomial of $L$.  Since the writhe number of a knot is independent of its orientation, this gives an invariant of unoriented knots.  The Alexander, Conway, and Jones polynomials are all specializations of the HOMFLYPT polynomial.

\begin{prop} \label{dragon}
    The affinization $\Aff(\OS(D;z,t,\delta))$ of the framed HOMFLYPT skein category over the disc is isomorphic to $\OS(A;z,t,\delta)$, the framed HOMFLYPT skein category over the annulus.
\end{prop}

\begin{proof}
    The proof is analogous to that of \cref{annularbraids}.
\end{proof}

\begin{cor} \label{Brundan}
    The framed HOMFLYPT skein category $\OS(A;z,t,\delta)$ over the annulus is isomorphic to the strict monoidal category generated by objects $\uparrow$, $\downarrow$, and morphisms
    \[
        \posupcross,\ \negupcross,\ \posrightcross,\ \negrightcross,\ \posdowncross,\ \negdowncross,\ \posleftcross,\ \negleftcross,\ \rightcup,\ \leftcup,\ \rightcap,\ \rightcup,\ \dotup,\ \dotdown,
    \]
    subject to relations \cref{R0}, \cref{FR1}, \cref{R2}, \cref{R3}, \cref{CS}, \cref{T}, \cref{D}, and \cref{OA}.
\end{cor}

The presentation from \cref{Brundan} implies that $\OS(A;z,t)$ is isomorphic to the \emph{affine oriented skein category} $\mathpzc{AOS}(z,t)$ of \cite[\S4]{Bru17}, which is also the \emph{quantum Heisenberg category of central charge zero} (see \cite{BSW-qheis}).  The functor $\OS(A;z,t) \to \OS(D;s,t)$ from \cref{flatten} corresponds to the functor described in \cite[Lem.~4.2]{Bru17} after rescaling the dots by a factor of $t$.

\section{Unoriented examples}

In this section we continue our study of examples of the affinization of monoidal categories, now focussing on unoriented categories of tangles and skein categories.

\subsection{Tangles}

The category $\T(D)$ of tangles over the disc is isomorphic to the strict monoidal category generated by a single object $\go$, and morphisms
\[
    \poscross,\ \negcross,\ \uncup,\ \uncap,
\]
subject to the relations \cref{R0}, \cref{R1}, \cref{R2}, and \cref{R3}; see \cite[Th.~3.5]{FY89}.  The category $\T(D)$ is a strict ribbon category.

\begin{prop} \label{kiwi}
    The affinization $\Aff(\T(D))$ of the category of tangles over the disc is isomorphic to the category $\T(A)$ of tangles over the annulus.
\end{prop}

\begin{proof}
    The proof is analogous to that of \cref{annularbraids}.
\end{proof}

\begin{cor}
    The category $\T(A)$ of tangles over the annulus is isomorphic to the strict monoidal category generated by a single object $\go$, and morphisms
    \[
        \poscross,\ \negcross,\ \uncup,\ \uncap,\ \posdotun,\ \negdotun,
    \]
    subject to the relations \cref{R0}, \cref{R1}, \cref{R2}, \cref{R3}, and \cref{UA}.
\end{cor}

\subsection{Framed tangles}

The category $\FT(D)$ of framed tangles over the disc is isomorphic to the strict monoidal category generated by a single object $\go$, and morphisms
\[
    \poscross,\ \negcross,\ \uncup,\ \uncap,
\]
subject to the relations \cref{R0}, \cref{FR1}, \cref{R2}, and \cref{R3}; see \cite[p.~436]{Tur89}.  It is a strict ribbon category.

\begin{prop} \label{dodo}
    The affinization $\Aff(\FT(D))$ of the category of framed tangles over the disc is isomorphic to the category $\FT(A)$ of tangles over the annulus.
\end{prop}

\begin{proof}
    The proof is analogous to that of \cref{annularbraids}.
\end{proof}

\begin{cor}
    The category $\FT(A)$ of framed tangles over the annulus is isomorphic to the strict monoidal category generated by a single object $\go$, and morphisms
    \[
        \poscross,\ \negcross,\ \uncup,\ \uncap,\ \posdotun,\ \negdotun,
    \]
    subject to the relations \cref{R0}, \cref{FR1}, \cref{R2}, \cref{R3}, and \cref{UA}.
\end{cor}

\subsection{Kauffman skein category}

Let $\varepsilon \in \{1,-1\}$, and fix $z,\delta \in \kk$, $t \in \kk^\times$ satisfying $z(\delta + \varepsilon) = t + \varepsilon t^{-1}$.  Generically, we work over the ring $\kk = \Z[z,t,t^{-1},\delta]/(z(\delta + \varepsilon) - t - \varepsilon t^{-1})$.  The \emph{Kauffman skein categories} $\KS_\varepsilon(D;z,t,\delta)$ over the disc are the categories obtained from $\FT(D)_\kk$ by imposing the relation ($\text{KS}_\varepsilon$) (i.e.\ \cref{KS+} when $\varepsilon=1$ and \cref{KS-} when $\varepsilon=-1$), the twist relation \cref{T}, and the dimension relation \cref{D}.  Replacing the disc $D$ by the annulus $A$, we get the analogous categories over the annulus.  One might also use the term \emph{Kauffman skein category} to refer to the choice $\varepsilon=1$ and \emph{Dubrovnik skein category} to refer to the choice $\varepsilon=-1$.

The reason for imposing the condition $z(\delta + \varepsilon) = t + \varepsilon t^{-1}$ in our coefficient ring is that we have
\begin{equation} \label{hippo}
    (t + \varepsilon t^{-1})\
    \begin{tikzpicture}[centerzero]
        \draw (0,-0.6) -- (0,0.6);
    \end{tikzpicture}
    \overset{\cref{T}}{=}
    \begin{tikzpicture}[centerzero]
    	\draw (0,0.6) to (0,0.3);
    	\draw (0.3,-0.2) to [out=0,in=-90](.5,0);
    	\draw (0.5,0) to [out=90,in=0](0.3,0.2);
    	\draw (0,-0.3) to (0,-0.6);
    	\draw (0,0.3) to [out=-90,in=180] (0.3,-0.2);
    	\draw[wipe] (0.3,.2) to [out=180,in=90](0,-0.3);
    	\draw (0.3,.2) to [out=180,in=90](0,-0.3);
    \end{tikzpicture}
    + \varepsilon\
    \begin{tikzpicture}[centerzero]
    	\draw (0,0.6) to (0,0.3);
    	\draw (0.3,-0.2) to [out=0,in=-90](.5,0);
    	\draw (0.5,0) to [out=90,in=0](0.3,0.2);
    	\draw (0,-0.3) to (0,-0.6);
    	\draw (0.3,.2) to [out=180,in=90](0,-0.3);
    	\draw[wipe] (0,0.3) to [out=-90,in=180] (0.3,-0.2);
    	\draw (0,0.3) to [out=-90,in=180] (0.3,-0.2);
    \end{tikzpicture}
    \overset{(\text{KS}_\varepsilon)}{\underset{\cref{R0}}{=}} z\
    \begin{tikzpicture}[centerzero]
        \draw (0,-0.6) -- (0,0.6);
        \bubun{0.4,0};
    \end{tikzpicture}
    + \varepsilon z\
    \begin{tikzpicture}[centerzero]
        \draw (0,-0.6) -- (0,0.6);
    \end{tikzpicture}
    \overset{\cref{D}}{=} z(\delta + \varepsilon)\
    \begin{tikzpicture}[centerzero]
        \draw (0,-0.6) -- (0,0.6);
    \end{tikzpicture}
    \ .
\end{equation}
When $z \in \kk^\times$, we have $\delta = 1 - \varepsilon (t + \varepsilon t^{-1})/z$, and we denote the categories simply by $\KS_\varepsilon(D;z,t)$.  These categories were introduced by Turaev in \cite[\S7.7]{Tur89}.

The endomorphism algebras $\End_{\KS_-(D;z,t)}(\go^{\otimes n})$ are the \emph{Kauffman tangle algebras}, which are isomorphic to the \emph{Birman--Murakami--Wenzl (BMW) algebras}; see \cite{Mor10}.  Explicit bases for the morphism spaces of $\KS_\pm(D;z,t)$ are given in \cite[Th.~7.8]{Tur89}.  In particular, $\End(\one)$ is one-dimensional, and this gives rise to link invariants as for the HOMFLYPT skein category described in \cref{subsec:HOMFLYPT}.   Namely, suppose $\kk = \Z[z,t,t^{-1},\delta]/(z(\delta + \varepsilon) - t - \varepsilon t^{-1})$.   Given an oriented link diagram $L$, we can view $L$ as an element of $\End_{\KS_\varepsilon(D;z,t)}(\one)$.  Then there is a unique scalar $F_{L,\varepsilon}(z,t) \in \kk[z,t,t^{-1}]$ such that $t^{-\writhe(L)} L = F_{L,\varepsilon}(z,t) \delta 1_\one$.  The scalar $F_{L,1}(z,t)$ is the \emph{Kauffman polynomial} of $L$ and $F_{L,-1}(z,t)$ is the \emph{Dubrovnik polynomial} of $L$.

It was noted by Lickorish (see \cite[p.~466]{Kau90}) that the Kauffman and Dubrovnik polynomials are essentially equivalent when one extends scalars to a ring including a square root $i = \sqrt{-1}$ of negative one, in the sense that
\[
    F_{L,1}(z,t) = i^{-\writhe(L)} (-1)^{c(L)+1} F_{L,-1}(-iz,it),
\]
where $c(L)$ is the number of components of $L$.  More generally, one has the following.

\begin{lem}
    Suppose $\kk$ contains a square root $i$ of $-1$.  Then, for each $n \in \N$, we have an isomorphism of algebras
    \[
        \End_{\KS_+(D;z,t,\delta)}(\go^{\otimes n}) \to \End_{\KS_-(D;-iz,it,-\delta)}(\go^{\otimes n}),
    \]
    determined by
    \begin{equation} \label{headlight}
        \begin{aligned}
            1_{\go^{\otimes (j-1)}} \otimes \poscross \otimes 1_{\go^{\otimes (n-j-1)}}
            &\mapsto i 1_{\go^{\otimes (j-1)}} \otimes \poscross \otimes 1_{\go^{\otimes (n-j-1)}},
            \\
            1_{\go^{\otimes (j-1)}} \otimes
            \begin{tikzpicture}[centerzero]
                \draw (-0.15,0.25) -- (-0.15,0.2) arc(180:360:0.15) -- (0.15,0.25);
                \draw (-0.15,-0.25) -- (-0.15,-0.2) arc(180:0:0.15) -- (0.15,-0.25);
            \end{tikzpicture}
            \otimes 1_{\go^{\otimes (n-j-1)}}
            &\mapsto - 1_{\go^{\otimes (j-1)}} \otimes
            \begin{tikzpicture}[centerzero]
                \draw (-0.15,0.25) -- (-0.15,0.2) arc(180:360:0.15) -- (0.15,0.25);
                \draw (-0.15,-0.25) -- (-0.15,-0.2) arc(180:0:0.15) -- (0.15,-0.25);
            \end{tikzpicture}
            \otimes 1_{\go^{\otimes (nj-1)}},
        \end{aligned}
    \end{equation}
    for $1 \le j < n$.
\end{lem}

\begin{proof}
    It is well-known that the elements appearing in \cref{headlight} generate the endomorphism algebras of $\KS_\pm(D;z,t,\delta)$, i.e.\ the BMW algebras in their original and Dubrovnik forms.  Then it is straightforward to verify, using the standard presentation of these algebras, that the given map is an isomorphism.
\end{proof}

Because of the above relationship between the two choices of $\varepsilon$, many authors make the choice that best suits their particular purpose.  In the literature, the two choices are often referred to as ``variants'' of each other.  However, if one is interested in the full monoidal categories, there is no such relationship, as the following result shows.

\begin{prop}
    If $\kk$ is an integral domain of characteristic not equal to two, then there \emph{do not exist any} choices of $z_\pm,t_\pm,\delta_\pm$ for which $\KS_+(D;z_+,t_+,\delta_+)$ and $\KS_-(D;z_-,t_-,\delta_-)$ are isomorphic as monoidal categories.
\end{prop}

\begin{proof}
    Suppose, towards a contradiction, that we have an isomorphism
    \[
        \Psi \colon \KS_+(D;z_+,t_+,\delta_+) \to \KS_-(D;z_-,t_-,\delta_-)
    \]
    of monoidal categories.  Since this induces an isomorphism after we extend the base ring, we may assume that $\kk$ is a field.  By \cite[Th.~7.8]{Tur89}, the morphism space $\Hom(\go^{\otimes 2},\one)$ is one-dimensional (in both categories), spanned by $\uncap$.  Similarly, $\Hom(\one, \go^{\otimes 2})$ is one-dimensional, spanned by $\uncup$.  It follows that $\Psi(\uncap) = d\, \uncap$ and $\Psi(\uncup) = c\, \uncup$ for some $c,d \in \kk$.  Since $\Psi$ must preserve the first equality in \cref{R0}, we conclude that $cd=1$.

    Let $\Rot$ be the linear operator on $\End(\go^{\otimes 2})$ giving by counterclockwise rotation by $90\degree$.   This operation is given by tensoring on the left and right by $\go$, then adding a cap to the top two strands, a cup to the bottom two strands, then using \cref{R0}.  For any $f \in \End_{\KS_+(D;z_+,t_+,\delta_+)}(\go^{\otimes 2})$, we have
    \[
        \Psi(\Rot(f))
        = \Psi
        \left(
            \begin{tikzpicture}[anchorbase]
                \draw (-0.25,0.2) rectangle (0.25,-0.2);
                \node at (0,0) {$\scriptstyle{f}$};
                \draw (-0.4,-0.4) -- (-0.4,0.2) arc (180:0:0.15);
                \draw (0.4,0.4) -- (0.4,-0.2) arc(360:180:0.15);
                \draw (-0.1,-0.4) -- (-0.1,-0.2);
                \draw (0.1,0.4) -- (0.1,0.2);
            \end{tikzpicture}
        \right)
        =
        \begin{tikzpicture}[anchorbase]
            \draw (-0.35,0.2) rectangle (0.35,-0.2);
            \node at (0,0) {$\scriptstyle{\Psi(f)}$};
            \draw (-0.5,-0.4) -- (-0.5,0.2) arc (180:0:0.2);
            \draw (0.5,0.4) -- (0.5,-0.2) arc(360:180:0.2);
            \draw (-0.1,-0.4) -- (-0.1,-0.2);
            \draw (0.1,0.4) -- (0.1,0.2);
        \end{tikzpicture}
        = \Rot(\Psi(f)).
    \]
    In particular, $\Psi$ preserves the eigenspaces of $\Rot$.  However, it follows from \cite[Th.~7.8]{Tur89} that the endomorphism space $\End_{\KS_\varepsilon(D;z_\varepsilon,t_\varepsilon,\delta_\varepsilon)}(\go^{\otimes 2})$ is three-dimensional, with basis
    \[
        \poscross + \negcross,\quad
        \poscross - \negcross,\quad
        \begin{tikzpicture}[centerzero]
            \draw (-0.15,-0.25) -- (-0.15,0.25);
            \draw (0.15,-0.25) -- (0.15,0.25);
        \end{tikzpicture}
        - \varepsilon\
        \begin{tikzpicture}[centerzero]
            \draw (-0.15,0.25) -- (-0.15,0.2) arc(180:360:0.15) -- (0.15,0.25);
            \draw (-0.15,-0.25) -- (-0.15,-0.2) arc(180:0:0.15) -- (0.15,-0.25);
        \end{tikzpicture}
        \ .
    \]
    In particular, the $+1$ eigenspace is one-dimensional for $\varepsilon=1$ and it is two-dimensional for $\varepsilon=-1$.  This contradicts the assumption that $\Psi$ is an isomorphism.
\end{proof}

\begin{prop} \label{dinosaur}
    The affinization $\Aff(\KS_\varepsilon(D;z,t,\delta))$ of the Kauffman skein category over the disc is isomorphic to $\KS_\varepsilon(A;z,t,\delta)$, the Kauffman skein category over the annulus.
\end{prop}

\begin{proof}
    The proof is analogous to that of \cref{annularbraids}.
\end{proof}

By \cref{plain}, $\Aff(\KS_\varepsilon(D;z,t,\delta))$ is the strict monoidal category generated by a single object $\go$, and morphisms
\[
    \poscross,\ \negcross,\ \uncup,\ \uncap,\ \posdotun,\ \negdotun,
\]
subject to the relations \cref{R0}, \cref{FR1}, \cref{R2}, \cref{R3}, ($\text{KS}_\varepsilon$), \cref{T}, \cref{D}, and \cref{UA}.  For the choice $\varepsilon = -1$, the category $\Aff(\KS_-(D;z,t,\delta))$ appeared in \cite[Def.~1.3]{GRS20}, where the authors called it the \emph{affine Kauffman skein category} (although they have made the choice of the Dubrovnik skein relation), and gave a basis theorem for it and its cyclotomic quotients.  The topological interpretation in terms of diagrams on the cylinder does not seem to appear there.  For the choice $\varepsilon=1$, this category does not seem to have been studied yet in the literature.  The properties of the Jucys--Murphy elements do not seem to have been explored in either of the two cases.

\subsection{Temperley--Lieb category\label{subsec:TL}}

For $\delta \in \kk$, the \emph{Temperley--Lieb category} $\TL(\delta)$ is the free $\kk$-linear rigid monoidal category generated by a self-dual object of dimension $\delta$.  It is generated by a single object $\go$, and morphisms $\uncup$ and $\uncap$, subject to the relation \cref{D} and the first two equalities in \cref{R0}.

For the remainder of this section, assume there exists $q \in \kk^\times$ such that
\begin{equation} \label{TLbraiddelt}
    \delta = -q^2 - q^{-2}.
\end{equation}
Then $\TL(\delta)$ is braided with the braiding given by
\begin{equation} \label{TLbraid}
    \begin{tikzpicture}[centerzero]
        \draw (0.3,-0.4) -- (-0.3,0.4);
        \draw[wipe] (-0.3,-0.4) -- (0.3,0.4);
        \draw (-0.3,-0.4) -- (0.3,0.4);
    \end{tikzpicture}
    \ := q\
    \begin{tikzpicture}[centerzero]
        \draw (-0.2,-0.4) -- (-0.2,0.4);
        \draw (0.2,-0.4) -- (0.2,0.4);
    \end{tikzpicture}
    \ + q^{-1}\
    \begin{tikzpicture}[centerzero]
        \draw (-0.2,0.4) -- (-0.2,0.3) arc(180:360:0.2) -- (0.2,0.4);
        \draw (-0.2,-0.4) -- (-0.2,-0.3) arc(180:0:0.2) -- (0.2,-0.4);
    \end{tikzpicture}
    \ ,\qquad
    \begin{tikzpicture}[centerzero]
        \draw (-0.3,-0.4) -- (0.3,0.4);
        \draw[wipe] (0.3,-0.4) -- (-0.3,0.4);
        \draw (0.3,-0.4) -- (-0.3,0.4);
    \end{tikzpicture}
    \ := q^{-1}\
    \begin{tikzpicture}[centerzero]
        \draw (-0.2,-0.4) -- (-0.2,0.4);
        \draw (0.2,-0.4) -- (0.2,0.4);
    \end{tikzpicture}
    \ + q\
    \begin{tikzpicture}[centerzero]
        \draw (-0.2,0.4) -- (-0.2,0.3) arc(180:360:0.2) -- (0.2,0.4);
        \draw (-0.2,-0.4) -- (-0.2,-0.3) arc(180:0:0.2) -- (0.2,-0.4);
    \end{tikzpicture}
    \ ,
\end{equation}
and then extended to arbitrary objects using \cref{braidextend}; see \cite[Prop.~2.6]{BSA18} (note that our $q$ is the $q^{1/2}$ of \cite{BSA18}).  The cups and caps make $\TL(\delta)$ a strict pivotal category, and so we can define a twist as in \cref{looptwist}:
\begin{equation}
    \theta_\go :=
    \begin{tikzpicture}[anchorbase]
    	\draw[-] (0,0.6) to (0,0.3);
    	\draw (0.3,-0.2) to [out=0,in=-90](.5,0);
    	\draw (0.5,0) to [out=90,in=0](.3,0.2);
    	\draw (0,-0.3) to (0,-0.6);
    	\draw (0,0.3) to [out=-90,in=180] (.3,-0.2);
    	\draw[wipe] (0.3,.2) to [out=180,in=90](0,-0.3);
    	\draw (0.3,.2) to [out=180,in=90](0,-0.3);
    \end{tikzpicture}
    =
    - q^3\
    \begin{tikzpicture}[centerzero]
        \draw (0,-0.6) -- (0,0.6);
    \end{tikzpicture}
    \ .
\end{equation}
Extending the twist to a general object $\go^{\otimes n}$ using \cref{swirl} recovers the twist defined in \cite[Prop.~2.10]{BSA18}.

Note that the definition \cref{TLbraid} corresponds precisely to the relation \cref{KB} and its image under clockwise rotation by $90\degree$.   Thus $\TL(-q^2-q^{-2})$ is isomorphic to a quotient of either of the Kauffman skein categories $\KS_\pm(D;q \pm q^{-1},-q^3,-q^2-q^{-2})$ by \cref{KB}.

The category $\TL(-q^2-q^{-2})$ underpins the Jones polynomial just like the oriented skein category underpins the HOMFLYPT polynomial (see \cref{subsec:HOMFLYPT}).  Given an unoriented link $L$, let $\vec{L}$ be an oriented link obtained by choosing some orientation of the components of $L$.  Then there is a unique scalar $V_{\vec{L}}(q) \in \kk$ such that $(-q^3)^{-\writhe(\vec{L})} L = V_{\vec{L}}(q) \delta 1_\one$ in $\TL(-q^2-q^{-2})$.  If $\kk = \Z[q,q^{-1}]$ and $L$ is a knot (in which case the writhe number is independent of the chosen orientation), then $V_L(q)$ is the Jones polynomial of $L$ (a Laurent polynomial in an indeterminate $t$), specialized at $t=q^4$.

We can now apply our general affinization procedure to obtain the \emph{affine Temperley--Lieb category} $\Aff(\TL(\delta))$.  It is the strict monoidal category generated by the object $\go$ and morphisms
\[
    \uncup,\ \uncap,\ \poscross,\ \negcross,\ \posdotun,\ \negdotun,
\]
subject to the relations \cref{R0}, \cref{FR1}, \cref{R2}, \cref{R3}, \cref{KB}, and \cref{UA}. Alternatively, if we want a presentation without crossings, $\Aff(\TL(\delta))$ is the strict monoidal category generated by an object $\go$ and morphisms
\[
    \uncup,\ \uncap,\ \posdotun,\ \negdotun,
\]
subject to the relations from \cref{R0} and \cref{UA} that do not involve crossings, together with the relation
\begin{equation} \label{racoon}
    q\
    \begin{tikzpicture}[centerzero]
        \draw (-0.2,-0.4) -- (-0.2,0.4);
        \draw (0.2,-0.4) -- (0.2,0.4);
        \posdot{0.2,0};
    \end{tikzpicture}
    \ - q^{-1}\
    \begin{tikzpicture}[centerzero]
        \draw (-0.2,-0.4) -- (-0.2,0.4);
        \draw (0.2,-0.4) -- (0.2,0.4);
        \posdot{-0.2,0};
    \end{tikzpicture}
    \ = q\
    \begin{tikzpicture}[centerzero]
        \draw (-0.2,0.5) -- (-0.2,0.3) arc(180:360:0.2) -- (0.2,0.5);
        \draw (-0.2,-0.5) -- (-0.2,-0.3) arc(180:0:0.2) -- (0.2,-0.5);
        \posdot{-0.2,-0.3};
    \end{tikzpicture}
    \ - q^{-1}\
    \begin{tikzpicture}[centerzero]
        \draw (-0.2,0.5) -- (-0.2,0.3) arc(180:360:0.2) -- (0.2,0.5);
        \draw (-0.2,-0.5) -- (-0.2,-0.3) arc(180:0:0.2) -- (0.2,-0.5);
        \posdot{0.2,0.3};
    \end{tikzpicture}
    \ .
\end{equation}
The affine Temperley--Lieb category was introduced in \cite[Def.~2.5]{GL98}, where the morphism spaces were defined to be linear combinations of Temperley--Lieb diagrams on the cylinder.  (The $q$ of \cite{GL98} is our $q^2$.)  We were not able to find the above presentation of the affine Temperley--Lieb category, with the relation \cref{racoon}, in the literature.

One can also form the affine Temperley--Lieb category $\Aff(\TL(\delta))$ for general $\delta$, not necessarily of the form \cref{TLbraiddelt}.  However, one then loses the braided monoidal structure on $\TL(\delta)$.  Thus we cannot use the presentation coming from \cref{plain}, and we lose the monoidal structure on $\Aff(\TL(\delta))$.

\section{Relation to the horizontal trace\label{sec:htr}}

In several places in the literature, the \emph{horizontal trace} is used as the formal concept embodying the idea of monoidal categories on annuli or cylinders.  For example, the horizontal trace is defined in \cite[\S2.4]{BHLZ17}, where it is stated that the horizontal trace of $\cC$ can be naturally regarded as the category of $\cC$-diagrams on the annulus \emph{when $\cC$ admits biadjoints}.  A particular instance of this can be found in \cite[Prop.~4.4]{CK18}, where the authors identify the annular spider category with the horizontal trace of the usual spider category.

In this section, we give a precise relationship between the horizontal trace and the affinization of a monoidal category.  The two are different in general, but are isomorphic \emph{when $\cC$ is rigid}.  This explains why the horizontal trace corresponds to diagrams on the annulus for rigid categories.  For example, the spider category considered in \cite{CK18} is rigid.  However, when $\cC$ is \emph{not} rigid, it is the affinization, and not the horizontal trace, that naturally corresponds to $\cC$-diagrams on the annulus.  Below we give examples to illustrate this distinction.

Suppose $\cC$ is an essentially small monoidal category.  Fix two objects $X,Y$ in $\cC$, and consider pairs $(Z,f)$, where $Z$ is an object of $\cC$ and $f \colon X \otimes Z \to Z \otimes Y$ is a morphism in $\cC$.  We define an equivalence relation on such pairs generated by the relations
\begin{equation} \label{htrel}
    (Z, f \circ (1_X \otimes g)) \sim (Z', (g \otimes 1_Y) \circ f), \quad
    g \in \Hom_\cC(Z,Z'),\ f \in \Hom_\cC(X \otimes Z', Z \otimes Y).
\end{equation}
The \emph{horizontal trace} $\htr(\cC)$ of $\cC$ is the category with the same objects as $\cC$ and with $\Hom_{\htr(\cC)}(X,Y)$ the set of equivalence classes $[Z,f]$ of such pairs $(Z,f)$.  Composition of morphisms is given by
\begin{equation} \label{htcomp}
    [Z', g \colon X \otimes Z' \to Z' \otimes Y]
    \circ [Z, f \colon W \otimes Z \to Z \otimes X]
    := [Z \otimes Z', (1_Z \otimes g) \circ (f \otimes 1_{Z'})].
\end{equation}

If $\cC$ is a braided strict monoidal category, then $\htr(\cC)$ is a strict monoidal category, where the tensor product of objects is the same as in $\cC$, and the tensor product of morphisms is given by
\begin{multline} \label{fly}
    [Z_1, f_1 \colon X_1 \otimes Z_1 \to Z_1 \otimes Y_1]
    \otimes [Z_2, f_2 \colon X_2 \otimes Z_2 \to Z_2 \otimes Y_2]
    \\
    := [Z_1 \otimes Z_2, (1_{Z_1} \otimes \beta_{Y_1,Z_2} \otimes 1_{Y_2}) \circ (f_1 \otimes f_2) \circ (1_{X_1} \otimes \beta_{Z_1,X_2}^{-1} \otimes 1_{Z_2})].
\end{multline}

\begin{rem}
    In the case that $\cC$ is a braided monoidal category, a formula for a tensor product on $\htr(\cC)$ is asserted in \cite[\S2.4.2]{CK18}.  However, the tensor product given there does \emph{not} satisfy the properties required of a tensor product in a braided monoidal category.  For example, when $\cC$ is the category of braids over the disc (see \cref{subsec:braids}) with generating object $\go$ and $f = \eta_\go \circ \epsilon_\go$, then $(f \otimes 1_\go) \circ (1_\go \otimes f) \ne f \otimes f \ne (1_\go \otimes f) \circ (f \otimes 1_\go)$ with the tensor product of \cite[\S2.4.2]{CK18}.
\end{rem}

\begin{rem} \label{red}
    It can be helpful to visualize the morphism $[Z,f]$, $f \colon X \otimes Z \to Z \otimes Y$, as the following diagram on the cylinder:
    \begin{equation} \label{herring}
        \begin{tikzpicture}[anchorbase]
            \identify{-0.6}{-0.6}{0.6}{0.6};
            \draw[rounded corners] (-0.3,-0.2) rectangle (0.3,0.2);
            \node at (0,0) {$\scriptstyle{f}$};
            \draw (-0.15,-0.6) node[anchor=north] {\dotlabel{X}} -- (-0.15,-0.2);
            \draw (0.2,-0.2) arc(180:270:0.15) to[out=0,in=225] (0.6,0) node[anchor=west] {\dotlabel{Z}};
            \draw (-0.6,0) to[out=45,in=180] (-0.35,0.35) arc(90:0:0.15);
            \draw (0.2,0.2) -- (0.2,0.6) node[anchor=south] {\dotlabel{Y}};
        \end{tikzpicture}
    \end{equation}
    where, as usual, we identify the vertical edges.  The relation \cref{htrel} then corresponds to the equality
    \begin{equation}
        \begin{tikzpicture}[anchorbase]
            \identify{-0.6}{-1}{0.7}{0.6};
            \draw[rounded corners] (-0.3,-0.2) rectangle (0.3,0.2);
            \node at (0,0) {$\scriptstyle{f}$};
            \draw (-0.2,-1) node[anchor=north] {\dotlabel{X}} -- (-0.2,-0.2);
            \draw (0.2,-0.2) -- (0.2,-0.7) arc(180:270:0.15) to[out=0,in=225] (0.7,0) node[anchor=west] {\dotlabel{Z}};
            \coupon{0.2,-0.5}{g};
            \draw (-0.6,0) to[out=45,in=180] (-0.35,0.35) arc(90:0:0.15);
            \draw (0.2,0.2) -- (0.2,0.6) node[anchor=south] {\dotlabel{Y}};
        \end{tikzpicture}
        =
        \begin{tikzpicture}[anchorbase]
            \identify{-0.7}{-0.6}{0.6}{1};
            \draw[rounded corners] (-0.3,-0.2) rectangle (0.3,0.2);
            \node at (0,0) {$\scriptstyle{f}$};
            \draw (-0.2,-0.6) node[anchor=north] {\dotlabel{X}} -- (-0.2,-0.2);
            \draw (-0.2,0.2) -- (-0.2,0.7) arc(0:90:0.15) to[out=180,in=45] (-0.7,0);
            \coupon{-0.2,0.5}{g};
            \draw (0.6,0) node[anchor=west] {\dotlabel{Z'}} to[out=225,in=0] (0.35,-0.35) arc(270:180:0.15);
            \draw (0.2,0.2) -- (0.2,1) node[anchor=south] {\dotlabel{Y}};
        \end{tikzpicture}
        \ ,
    \end{equation}
    where we think of sliding the coupon $g$ around the cylinder.  However, in general, \emph{one should not view this as a precise string diagram}.  The reason is that the cups and caps drawn above do not have any precise meaning, since our category may not be rigid!  (This is why we do not orient the strands in \cref{herring}.)  This observation is essential, since it explains the difference between the affinization and horizontal trace constructions.  As we will see below, for non-rigid categories the affinization and horizontal trace differ in general.  We will see in \cref{donut} that if our category is rigid then the affinization and horizontal trace agree, essentially because we can view \cref{herring} as a string diagram, with the cup and cap being a unit and counit morphism as in \cref{rcps}.
\end{rem}

\begin{rem} \label{ht-annulus}
    One also sometimes sees the morphism $[Z,f]$, $f \colon X \otimes Z \to Z \otimes Y$, drawn as a diagram on the annulus:
    \[
        \begin{tikzpicture}[centerzero]
            \draw[densely dotted,green!40!black] (0.2,0) arc(0:360:0.2);
            \draw[densely dotted,green!40!black] (1.2,0) arc(0:360:1.2);
            \draw[rounded corners] (-0.3,0.5) rectangle (0.3,0.9);
            \node at (0,0.7) {$\scriptstyle{f}$};
            \draw (-0.15,0.5) to[out=down,in=up] (0,0.2);
            \draw (0.15,0.9) to[out=up,in=down] (0,1.2);
            \draw (0.15,0.5) to[out=down,in=up] (0.7,0) arc(360:180:0.7) to[out=up,in=down] (-0.4,0.9) arc(180:0:0.125);
        \end{tikzpicture}
    \]
    However, the same pitfalls are present here as they are with the cylindrical depiction of \cref{red}, since the cups and caps appearing in this diagram do not have a precise meaning in general.  See also \cref{cylinder-vs-annulus}.
\end{rem}

\begin{theo} \label{donut}
    \begin{enumerate}
        \item If $\cC$ is right rigid, then there is a functor $\Theta \colon \Aff(\cC) \to \htr(\cC)$ that is the identity on objects and is given on morphisms by
            \begin{equation} \label{donut1}
                f \mapsto [\one,f],\quad
                \xi_{X,Y}
                \mapsto [Y^\vee, \eta_Y \otimes 1_X \otimes \epsilon_Y],\quad
                \xi_{X,Y}^{-1} \mapsto [Y,1_Y \otimes 1_X \otimes 1_Y],
            \end{equation}
            for morphisms $f$ in $\cC$ and objects $X,Y$ in $\cC$.

        \item If $\cC$ is left rigid, then there is a functor $\Theta' \colon \htr(\cC) \to \Aff(\cC)$ that is the identity on objects and is given on morphisms by
            \begin{equation} \label{donut2}
                [Z,f] \mapsto (\epsilon'_Z \otimes 1_Y) \circ (1_{\leftdual{Z}} \otimes f) \circ \xi_{X \otimes Z, \leftdual{Z}} \circ (1_X \otimes \eta'_Z)
                =
                \begin{tikzpicture}[anchorbase]
                    \identify{-0.7}{-0.9}{0.7}{0.6};
                    \draw[rounded corners] (-0.3,-0.2) rectangle (0.3,0.2);
                    \node at (0,0) {$\scriptstyle{f}$};
                    \draw[->] (-0.2,-0.9) node[anchor=north] {\dotlabel{X}} -- (-0.2,-0.2);
                    \draw[<-] (0.2,-0.2) -- (0.2,-0.7) arc(180:270:0.15) to[out=0,in=225] (0.7,-0.4) node[anchor=west] {\dotlabel{Z}};
                    \draw[<-] (-0.7,-0.4) node[anchor=east] {\dotlabel{Z}} to[out=45,in=270] (-0.5,0.2) arc(180:0:0.15);
                    \draw[->] (0.2,0.2) -- (0.2,0.6) node[anchor=south] {\dotlabel{Y}};
                \end{tikzpicture}
            \end{equation}
            for objects $Z$ in $\cC$ and $f \in \Hom_\cC(X \otimes Z, Z \otimes Y)$.

        \item \label{frosting} If $\cC$ is rigid, then the functors $\Theta$ and $\Theta'$ are mutually inverse.  Hence the horizontal trace $\htr(\cC)$ and the affinization $\Aff(\cC)$ are isomorphic.

        \item If $\cC$ is braided and rigid, then the functors $\Theta$ and $\Theta'$ are mutually inverse isomorphisms of monoidal categories.
    \end{enumerate}
\end{theo}

\begin{proof}
    \begin{enumerate}[wide]
        \item Suppose $\cC$ is right rigid.  To verify that $\Theta$ is well defined, we must show that the given images of $\xi_{X,Y}$ and $\xi_{X,Y}^{-1}$ are inverses of each other, and that $\Theta$ preserves the relations \cref{coilrel1,coilrel2}.  For the first assertion, we compute
            \begin{multline*}
                [Y^\vee, \eta_Y \otimes 1_X \otimes \epsilon_Y] \circ
                [Y, 1_Y \otimes 1_X \otimes 1_Y]
                \\
                = [Y \otimes Y^\vee, (1_Y \otimes \eta_Y \otimes 1_X \otimes \epsilon_Y) \circ (1_Y \otimes 1_X \otimes 1_Y \otimes 1_{Y^\vee})]
                \\
                = [Y \otimes Y^\vee, 1_Y \otimes \eta_Y \otimes 1_X \otimes \varepsilon_Y]
                \overset{\cref{htrel}}{=} [\one, (\epsilon_Y \otimes 1_Y) \circ (1_Y \otimes \eta_Y) \otimes 1_X]
                \overset{\cref{zigright}}{=} [\one, 1_Y \otimes 1_X],
            \end{multline*}
            where we used the interchange law in the second equality.  A similar computation shows that
            \[
                [Y, 1_Y \otimes 1_X \otimes 1_Y] \circ
                [Y^\vee, (\eta_Y \otimes 1_X) \circ (1_X \otimes \epsilon_Y)]
                = [\one, 1_X \otimes 1_Y].
            \]
            Thus the images of $\xi_{X,Y}$ and $\xi_{X,Y}^{-1}$ are inverse, as desired.

            Next, note that
            \begin{multline*}
                \Theta(\xi_{X \otimes Y,Z}^{-1}) \circ \Theta(\xi_{Z \otimes X,Y}^{-1})
                = [Z, 1_Z \otimes 1_X \otimes 1_Y \otimes 1_Z] \circ [Y, 1_Y \otimes 1_Z \otimes 1_X \otimes 1_Y]
                \\
                = [Y \otimes Z, 1_Y \otimes 1_Z \otimes 1_X \otimes 1_Y \otimes 1_Z]
                = \Theta(\xi_{X,Y \otimes Z}^{-1}).
            \end{multline*}
            Taking inverses shows that $\Theta$ preserves \cref{coilrel1}.  Next, for $f \in \Hom_\cC(Y_1,Y_2)$ and $g \in \Hom_\cC(X_1,X_2)$, we compute
            \begin{multline*}
                \Theta(\xi_{X_2,Y_2}^{-1}) \circ \Theta(f \otimes g)
                = [Y_2, 1_{Y_2} \otimes 1_{X_2} \otimes 1_{Y_2}] \circ [\one,f \otimes g]
                = [Y_2, f \otimes g \otimes 1_{Y_2}]
                \\
                \overset{\cref{htrel}}{=} [Y_1, 1_{Y_1} \otimes g \otimes f]
                = [\one, g \otimes f] \circ [Y_1, 1_{Y_1} \otimes 1_{X_1} \otimes 1_{Y_1}]
                = \Theta(g \otimes f) \circ \Theta(\xi_{X_1,Y_1}^{-1}),
            \end{multline*}
            and so $\Theta$ preserves \cref{coilrel2}.

        \item Suppose $\cC$ is left rigid.  To verify that the functor $\Theta'$ is well-defined, suppose $g \in \Hom_\cC(Z,Z')$ and $f \in \Hom_\cC(X \otimes Z', Z \otimes Y)$.  Then
            \begin{multline*}
                \Theta'([Z, f \circ (1_X \otimes g)])
                =
                \begin{tikzpicture}[anchorbase]
                    \identify{-0.7}{-1.2}{0.7}{0.6};
                    \draw[rounded corners] (-0.3,-0.2) rectangle (0.3,0.2);
                    \node at (0,0) {$\scriptstyle{f}$};
                    \draw[->] (-0.2,-1.2) node[anchor=north] {\dotlabel{X}} -- (-0.2,-0.2);
                    \draw[<-] (0.2,-0.2) -- (0.2,-0.9) arc(180:270:0.15) to[out=0,in=225] (0.7,-0.7) node[anchor=west] {\dotlabel{Z}};
                    \coupon{0.2,-0.55}{g};
                    \draw[<-] (-0.7,-0.7) to[out=45,in=270] (-0.5,0.2) arc(180:0:0.15);
                    \draw[->] (0.2,0.2) -- (0.2,0.6) node[anchor=south] {\dotlabel{Y}};
                \end{tikzpicture}
                =
                \begin{tikzpicture}[anchorbase]
                    \identify{-0.7}{-1.2}{1.2}{0.6};
                    \draw[rounded corners] (-0.3,-0.2) rectangle (0.3,0.2);
                    \node at (0,0) {$\scriptstyle{f}$};
                    \draw[->] (-0.2,-1.2) node[anchor=north] {\dotlabel{X}} -- (-0.2,-0.2);
                    \draw[<-] (0.2,-0.2) -- (0.2,-0.9) arc(180:360:0.15) -- (0.5,-0.5) arc(180:0:0.15) -- (0.8,-0.9) arc(180:270:0.15) to[out=0,in=225] (1.2,-0.3) node[anchor=west] {\dotlabel{Z}};
                    \coupon{0.8,-0.7}{g};
                    \draw[<-] (-0.7,-0.7) to[out=45,in=270] (-0.5,0.2) arc(180:0:0.15);
                    \draw[->] (0.2,0.2) -- (0.2,0.6) node[anchor=south] {\dotlabel{Y}};
                \end{tikzpicture}
                \overset{\cref{whip}}{=}
                \begin{tikzpicture}[anchorbase]
                    \identify{-1.3}{-1}{0.7}{1.2};
                    \draw[rounded corners] (-0.3,-0.2) rectangle (0.3,0.2);
                    \node at (0,0) {$\scriptstyle{f}$};
                    \draw[->] (-0.2,-1) node[anchor=north] {\dotlabel{X}} -- (-0.2,-0.2);
                    \draw[<-] (0.2,-0.2) -- (0.2,-0.7) arc(180:270:0.15) to[out=0,in=225] (0.7,-0.5) node[anchor=west] {\dotlabel{Z}};
                    \draw[<-] (-1.3,-0.5) to[out=45,in=270] (-1.1,0.9) arc(180:0:0.15) -- (-0.8,0.5) arc(180:360:0.15) -- (-0.5,0.9) arc(180:0:0.15) -- (-0.2,0.2);
                    \coupon{-0.8,0.7}{g};
                    \draw[->] (0.2,0.2) -- (0.2,1.2) node[anchor=south] {\dotlabel{Y}};
                \end{tikzpicture}
                =
                \begin{tikzpicture}[anchorbase]
                    \identify{-0.7}{-1.2}{0.7}{0.6};
                    \draw[rounded corners] (-0.3,-0.5) rectangle (0.3,-0.1);
                    \node at (0,-0.3) {$\scriptstyle{f}$};
                    \draw[->] (-0.2,-1.2) node[anchor=north] {\dotlabel{X}} -- (-0.2,-0.5);
                    \draw[<-] (0.2,-0.5) -- (0.2,-0.9) arc(180:270:0.15) to[out=0,in=225] (0.7,-0.7) node[anchor=west] {\dotlabel{Z}};
                    \draw[<-] (-0.7,-0.7) to[out=45,in=270] (-0.5,0.3) arc(180:0:0.15) -- (-0.2,-0.1);
                    \coupon{-0.2,0.15}{g};
                    \draw[->] (0.2,-0.1) -- (0.2,0.6) node[anchor=south] {\dotlabel{Y}};
                \end{tikzpicture}
                \\
                =
                \Theta'([Z, (g \otimes 1_Y) \circ f)]).
            \end{multline*}

        \item Now suppose $\cC$ is rigid.  We check that the functors $\Theta$ and $\Theta'$ are mutually inverse.  For $f \in \Hom_\cC(X,Y)$, we have
            \[
                \Theta' \circ \Theta(f)
                = \Theta'([\one,f])
                = f \circ \xi_{X,\one}
                \overset{\cref{cable}}{=} f.
            \]
            For objects $X,Y$ in $\cC$, we have
            \[
                \Theta' \circ \Theta(\xi_{X,Y}^{-1})
                = \Theta' ([Y, 1_Y \otimes 1_X \otimes 1_Y])
                =
                \begin{tikzpicture}[centerzero]
                    \identify{-0.6}{-0.6}{0.6}{0.6};
                    \draw[->] (-0.2,-0.6) -- (-0.2,0.2) arc(0:90:0.15) to[out=180,in=45] (-0.6,0) node[anchor=east] {\dotlabel{Y}};
                    \draw[->] (0,-0.6) node[anchor=north] {\dotlabel{X}} -- (0,0.6);
                    \draw[<-] (0.2,0.6) -- (0.2,-0.2) arc(180:270:0.15) to[out=0,in=225] (0.6,0) node[anchor=west] {\dotlabel{Y}};
                \end{tikzpicture}
                \overset{\cref{whip}}{=}
                \begin{tikzpicture}[centerzero]
                    \identify{-0.6}{-0.6}{0.6}{0.6};
                    \draw[->] (0,-0.6) node[anchor=north] {\dotlabel{X}} to[out=up,in=down] (-0.3,0.6);
                    \draw[->] (-0.4,-0.6) node[anchor=north] {\dotlabel{Y}} to[out=90,in=-45] (-0.6,-0.2);
                    \draw[<-] (0,0.6) -- (0,0.3) arc(180:360:0.1) arc (180:0:0.1) to[out=-90,in=135] (0.6,-0.2) node[anchor=west] {\dotlabel{Y}};
                \end{tikzpicture}
                =
                \begin{tikzpicture}[anchorbase]
                    \identify{-0.7}{-0.5}{0.7}{0.5};
                    \draw[->] (0.3,-0.5) node[anchor=north] {\dotlabel{X}} \braidup (-0.3,0.5);
                    \draw[->] (-0.3,-0.5) node[anchor=north] {\dotlabel{Y}} to[out=up,in=-20] (-0.7,0);
                    \draw[->] (0.7,0) to[out=160,in=down] (0.3,0.5) node[anchor=south] {\dotlabel{Y}};
                \end{tikzpicture}
                = \xi_{X,Y}^{-1}.
            \]
            Thus $\Theta' \circ \Theta$ is the identity functor.  It is straightforward to verify that $\Theta \circ \Theta'$ is also the identity.

        \item Suppose $\cC$ is braided and rigid.  By part~\ref{frosting}, it suffices to show that the functor $\Theta$ is a monoidal functor.  For morphisms $f,g$ in any monoidal category, we have
            \[
                f \otimes g = (f \otimes 1) \circ (1 \otimes g).
            \]
            Thus, it suffices to show that $\Theta(f) \otimes \Theta(1_Z) = \Theta(f \otimes 1_Z)$ and $\Theta(1_Z) \otimes \Theta(f) = \Theta(1_Z \otimes f)$ for all objects $Z$ and morphisms $f$ in $\Aff(\cC)$.  For $f \in \Hom_\cC(X,Y)$ and $Z \in \Ob(\cC)$, we have
            \[
                \Theta(f) \otimes \Theta(1_Z)
                = [\one,f] \otimes [\one,1_Z]
                = [\one, \beta_{Y,\one} \circ f \circ (1_X \otimes \beta_{\one,\one}^{-1})]
                = [\one,f]
                = \Theta(f \otimes 1_Z).
            \]
            Similarly, $\Theta(1_Z) \otimes \Theta(f) = \Theta(1_Z \otimes f)$.  For $X,Y,Z \in \Ob(\cC)$, we have
            \begin{align*}
                \Theta(\xi_{X,Y}) \otimes \Theta(1_Z)\
                &\overset{\mathclap{\cref{donut1}}}{=}\ [Y^\vee, \eta_Y \otimes 1_X \otimes \epsilon_Y] \otimes [\one, 1_Z]
                \\
                &\overset{\mathclap{\cref{fly}}}{=}\ [Y^\vee, (1_{Y^\vee} \otimes \beta_{Y \otimes X, \one} \otimes 1_Z) \circ (\eta_Y \otimes 1_X \otimes \epsilon_Y \otimes 1_Z) \circ (1_X \otimes 1_Y \otimes \beta_{Y^\vee,Z}^{-1})]
                \\
                &= [Y^\vee, (\eta_Y \otimes 1_X \otimes \epsilon_Y \otimes 1_Z) \circ (1_X \otimes 1_Y \otimes \beta_{Y^\vee,Z}^{-1})]
                \\
                &\overset{\mathclap{\cref{R0}}}{=}\ [Y^\vee, (\eta_Y \otimes 1_X \otimes 1_Z \otimes \epsilon_Y) \circ (1_X \otimes \beta_{Y,Z} \otimes 1_{Y^\vee})]
                \\
                &\overset{\mathclap{\cref{htcomp}}}{=}\ [Y^\vee, \eta_Y \otimes 1_X \otimes 1_Z \otimes \epsilon_Y] \circ [\one, 1_X \otimes \beta_{Y,Z}]
                \\
                &\overset{\mathclap{\cref{donut1}}}{=}\ \Theta(\xi_{X \otimes Z,Y} \circ (1_X \otimes \beta_{Y,Z}))
                \\
                &\overset{\mathclap{\cref{stack}}}{=}\ \Theta(\xi_{X,Y} \otimes 1_Z)
            \end{align*}
            and
            \begin{align*}
                \Theta(1_X) \otimes \Theta(\xi_{Y,Z})\
                &\overset{\mathclap{\cref{donut1}}}{=}\ [\one, 1_X] \otimes [Z^\vee, \eta_Z \otimes 1_Y \otimes \epsilon_Z]
                \\
                &\overset{\mathclap{\cref{fly}}}{=}\ [Z^\vee, (\beta_{X,Z^\vee} \otimes 1_Z \otimes 1_Y) \circ (1_X \otimes \eta_Z \otimes 1_Y \otimes \epsilon_Z) \circ (1_X \otimes \beta_{\one,Y \otimes Z}^{-1} \otimes 1_{Z^\vee})]
                \\
                &= [Z^\vee, (\beta_{X,Z^\vee} \otimes 1_Z \otimes 1_Y) \circ (1_X \otimes \eta_Z \otimes 1_Y \otimes \epsilon_Z)]
                \\
                &\overset{\mathclap{\cref{R0}}}{=}\ [Z^\vee, (1_{Z^\vee} \otimes \beta_{X,Z}^{-1} \otimes 1_Y) \circ (\eta_Z \otimes 1_X \otimes 1_Y \otimes \epsilon_Z)]
                \\
                &\overset{\mathclap{\cref{htcomp}}}{=}\ [\one, \beta_{X,Z}^{-1} \otimes 1_Y] \circ [Z^\vee, \eta_Z \otimes 1_X \otimes 1_Y \otimes \epsilon_Z]
                \\
                &\overset{\mathclap{\cref{donut1}}}{=}\  \Theta ((\beta_{X,Z}^{-1} \otimes 1_Y) \circ \xi_{X \otimes Y,Z})
                \\
                &\overset{\mathclap{\cref{stack}}}{=}\ \Theta(1_X \otimes \xi_{Y,Z}). \qedhere
            \end{align*}
    \end{enumerate}
\end{proof}

We now give an example to show that the affinization and horizontal trace are \emph{not} isomorphic (or even equivalent) in general.

\begin{eg} \label{golf}
    Let $\cC$ be the free monoidal category on one generating object $\go$.  Thus $\Hom_\cC(\go^{\otimes m}, \go^{\otimes n})$ is empty when $m \ne n$ and consists of only $1_\go^{\otimes n}$ when $m=n$.  Then $\Aff(\cC)$ is generated by the morphisms $\xi_{\go^{\otimes m}, \go^{\otimes n}} \colon \go^{\otimes (m+n)} \to \go^{\otimes (m+n)}$, $m,n \in \N$, and their inverses.  In particular, since $\xi_{\one,\one} = 1_\one$ by \cref{cable}, we have $\End_{\Aff(\cC)}(\one) = \{1_\one\}$.  However, $\End_{\htr(\cC)}(\one) = \{[\go^{\otimes n}, \one_\go^{\otimes n}] : n \in \N\}$ has countably many elements.  In terms of cylindrical diagrams, $\htr(\cC)$ contains closed string diagrams wrapping around the cylinder, while $\Aff(\cC)$ does not.

    On the other hand, $\End_{\Aff(\cC)}(\go) = \{\xi_\go^n : n \in \Z\}$ is an infinite cyclic group, while $\End_{\htr(\cC)}(\go) = \{[\go^{\otimes n},1_\go^{\otimes (n+1)}] : n \in \N\}$ is an infinite cyclic monoid (whose generator is not invertible).  In terms of cylindrical diagrams, $\End_{\Aff(\cC)}(\go)$ contains strands that enter at the bottom of the cylinder, wrap around the cylinder in either direction, and exit out the top, while in $\End_{\htr(\cC)}(\go)$ such strands can only wrap in one direction.  This illustrates the asymmetry inherent in the definition of the horizontal trace.
\end{eg}

The next example illustrates how, in general, it is the affinization of $\cC$, and not its horizontal trace, that naturally corresponds to the category of $\cC$-diagrams on the cylinder or annulus.

\begin{eg} \label{eggs}
    As we saw in~\cref{annularbraids} the affinization $\Aff(\Braid(D))$ of the category of braids over the disc is the category $\Braid(A)$ of braids over the annulus.  On the other hand, $\htr(\Braid(D))$ is quite different.  Similar to the situation in \cref{golf}, the ``braids'' in $\htr(\Braid(D))$ can only wrap in one direction around the cylinder, and one also has closed components wrapping around the cylinder.
\end{eg}

\section{The vertical trace\label{sec:vtr}}

The \emph{vertical trace}, or simply the \emph{trace}, of a $\kk$-linear category $\cC$ is the $\kk$-module
\begin{equation} \label{tracedef}
    \vtr(\cC) := \left( \bigoplus_{X \in \Ob(\cC)} \End_\cC(X) \right) / \Span_\kk\{ f \circ g - g \circ f\},
\end{equation}
where $f$ and $g$ run through all pairs of morphisms $f \colon X \to Y$ and $g \colon Y \to X$ in $\cC$.  We let $[f] \in \vtr(\cC)$ denote the class of an endomorphism $f \in \End_\cC(X)$.

\begin{rem}
    One can also consider the vertical trace of categories that are not necessarily linear.  In this case, one replaces \cref{tracedef} by the set
    \[
        \vtr(\cC) = \left( \bigsqcup_{X \in \Ob(\cC)} \End_\cC(X) \right)/\sim,
    \]
    where $\sim$ is the equivalence relation generated by $f \circ g \sim g \circ f$ for all pairs of morphisms $f \colon X \to Y$ and $g \colon Y \to X$ in $\cC$.  Since most of our examples of interest are linear categories, we work in the linear setting in this section.  However, all of the results go through in the not-necessarily-linear setting with the obvious modifications.
\end{rem}

If $\cC$ is strict pivotal (in particular, this means that it is strict monoidal), we can naturally think of the trace as consisting of diagrams on the annulus.  In particular, if $f$ is an endomorphism in $\cC$, then we picture $[f]$ as
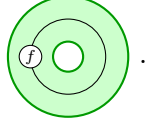
\[
  \begin{tikzpicture}[anchorbase]
    \filldraw[thick,draw=green!60!black,fill=green!20!white] (0.8,0) arc(0:360:0.8);
    \filldraw[thick,draw=green!60!black,fill=white] (0.2,0) arc(0:360:0.2);
    \draw (0.5,0) arc(0:360:0.5);
    \coupon{-0.5,0}{f};
  \end{tikzpicture}
  \ .
\]
The fact that $[f \circ g] = [g \circ f]$ in $\vtr(\cC)$ then corresponds to the fact we can slide diagrams around the annulus:
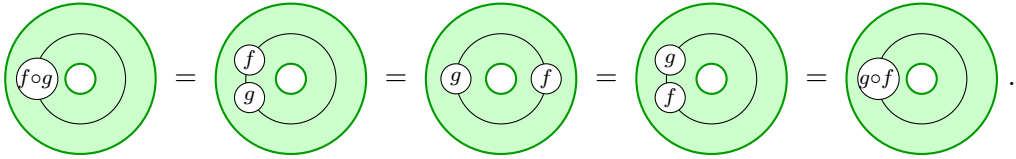
\[
  \begin{tikzpicture}[anchorbase]
    \filldraw[thick,draw=green!60!black,fill=green!20!white] (1,0) arc(0:360:1);
    \filldraw[thick,draw=green!60!black,fill=white] (0.2,0) arc(0:360:0.2);
    \draw (0.6,0) arc(0:360:0.6);
    \filldraw[draw=black,fill=white] (-0.3,0) arc(0:360:0.275);
    \node at (-0.6,0) {\dotlabel{f \circ g}};
  \end{tikzpicture}
  \ =\
  \begin{tikzpicture}[anchorbase]
    \filldraw[thick,draw=green!60!black,fill=green!20!white] (1,0) arc(0:360:1);
    \filldraw[thick,draw=green!60!black,fill=white] (0.2,0) arc(0:360:0.2);
    \draw (0.6,0) arc(0:360:0.6);
    \filldraw[draw=black,fill=white] (-0.35,0.25) arc(0:360:0.2);
    \filldraw[draw=black,fill=white] (-0.35,-0.25) arc(0:360:0.2);
    \node at (-0.55,0.25) {\dotlabel{f}};
    \node at (-0.55,-0.25) {\dotlabel{g}};
  \end{tikzpicture}
  \ =\
  \begin{tikzpicture}[anchorbase]
    \filldraw[thick,draw=green!60!black,fill=green!20!white] (1,0) arc(0:360:1);
    \filldraw[thick,draw=green!60!black,fill=white] (0.2,0) arc(0:360:0.2);
    \draw (0.6,0) arc(0:360:0.6);
    \filldraw[draw=black,fill=white] (-0.4,0) arc(0:360:0.2);
    \node at (-0.6,0) {\dotlabel{g}};
    \filldraw[draw=black,fill=white] (0.8,0) arc(0:360:0.2);
    \node at (0.6,0) {\dotlabel{f}};
  \end{tikzpicture}
  \ =\
  \begin{tikzpicture}[anchorbase]
    \filldraw[thick,draw=green!60!black,fill=green!20!white] (1,0) arc(0:360:1);
    \filldraw[thick,draw=green!60!black,fill=white] (0.2,0) arc(0:360:0.2);
    \draw (0.6,0) arc(0:360:0.6);
    \filldraw[draw=black,fill=white] (-0.35,0.25) arc(0:360:0.2);
    \filldraw[draw=black,fill=white] (-0.35,-0.25) arc(0:360:0.2);
    \node at (-0.55,0.25) {\dotlabel{g}};
    \node at (-0.55,-0.25) {\dotlabel{f}};
  \end{tikzpicture}
  \ =\
  \begin{tikzpicture}[anchorbase]
    \filldraw[thick,draw=green!60!black,fill=green!20!white] (1,0) arc(0:360:1);
    \filldraw[thick,draw=green!60!black,fill=white] (0.2,0) arc(0:360:0.2);
    \draw (0.6,0) arc(0:360:0.6);
    \filldraw[draw=black,fill=white] (-0.3,0) arc(0:360:0.275);
    \node at (-0.6,0) {\dotlabel{g \circ f}};
  \end{tikzpicture}
  \ .
\]
We also have that $\vtr(\cC)$ is a $\kk$-algebra, with multiplication given by $[f] \cdot [g] = [f \otimes g]$.  This corresponds to nesting of annuli:
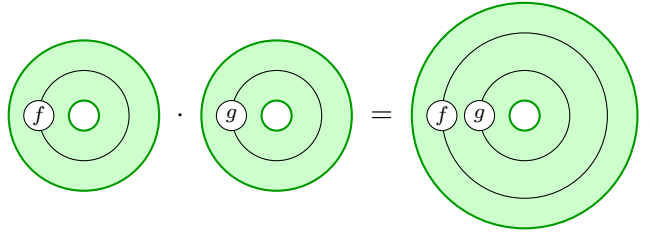
\[
  \begin{tikzpicture}[anchorbase]
    \filldraw[thick,draw=green!60!black,fill=green!20!white] (1,0) arc(0:360:1);
    \filldraw[thick,draw=green!60!black,fill=white] (0.2,0) arc(0:360:0.2);
    \draw (0.6,0) arc(0:360:0.6);
    \filldraw[draw=black,fill=white] (-0.4,0) arc(0:360:0.2);
    \node at (-0.6,0) {\dotlabel{f}};
  \end{tikzpicture}
  \ \cdot\
  \begin{tikzpicture}[anchorbase]
    \filldraw[thick,draw=green!60!black,fill=green!20!white] (1,0) arc(0:360:1);
    \filldraw[thick,draw=green!60!black,fill=white] (0.2,0) arc(0:360:0.2);
    \draw (0.6,0) arc(0:360:0.6);
    \filldraw[draw=black,fill=white] (-0.4,0) arc(0:360:0.2);
    \node at (-0.6,0) {\dotlabel{g}};
  \end{tikzpicture}
  \ =\
  \begin{tikzpicture}[anchorbase]
    \filldraw[thick,draw=green!60!black,fill=green!20!white] (1.5,0) arc(0:360:1.5);
    \filldraw[thick,draw=green!60!black,fill=white] (0.2,0) arc(0:360:0.2);
    \draw (0.6,0) arc(0:360:0.6);
    \filldraw[draw=black,fill=white] (-0.4,0) arc(0:360:0.2);
    \node at (-0.6,0) {\dotlabel{g}};
    \draw (1.1,0) arc(0:360:1.1);
    \filldraw[draw=black,fill=white] (-0.9,0) arc(0:360:0.2);
    \node at (-1.1,0) {\dotlabel{f}};
  \end{tikzpicture}
  \ .
\]
We refer the reader to \cite{BGHL14} for further details and properties of the trace.

Now suppose $\cC$ is a strict $\kk$-linear monoidal category.  Then $\Aff(\cC)$ is a $\kk$-linear category, and we can pass to its vertical trace $\vtr(\Aff(\cC)$.  If $\cC$ is strict pivotal, this naturally corresponds to string diagrams drawn on the torus.  For example, if $X \in \Ob(\cC)$, then $[\xi_X]$ corresponds to the diagram
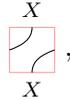
\[
    \begin{tikzpicture}[centerzero]
        \draw[\rectcolor] (-0.3,-0.3) rectangle (0.3,0.3);
        \draw (0,-0.3) node[anchor=north] {\dotlabel{X}} to[out=up,in=200] (0.3,0);
        \draw (-0.3,0) to[out=20,in=south] (0,0.3) node[anchor=south] {\dotlabel{X}};
    \end{tikzpicture}
    \ ,
\]
where we identity the vertical edges and the horizontal edges, thereby obtaining a diagram on the torus.

If $\cC$ is a braided strict $\kk$-linear monoidal category, then $\Aff(\cC)$ is a strict monoidal category, and hence $\vtr(\Aff(\cC))$ is a $\kk$-algebra.  If $\cC$ is also strict pivotal, the product on $\vtr(\Aff(\cC))$ corresponds to the nesting of tori.  This can be visualized as in \cref{bridge}, except that we only consider classes of endomorphisms and we identity the top and bottom of the diagrams (as well as the dashed vertical edges).  For example, for $X, Y \in \Ob(\cC)$, we have
\[
    [\xi_X] \cdot [1_Y]
    =
    \begin{tikzpicture}[centerzero]
        \draw[\rectcolor] (-0.3,-0.3) rectangle (0.3,0.3);
        \draw (0,-0.3) node[anchor=north] {\dotlabel{X}} to[out=up,in=200] (0.3,0);
        \draw (-0.3,0) to[out=20,in=south] (0,0.3) node[anchor=south] {\dotlabel{X}};
    \end{tikzpicture}
    \cdot
    \begin{tikzpicture}[centerzero]
        \draw[\rectcolor] (-0.3,-0.3) rectangle (0.3,0.3);
        \draw (0,-0.3) node[anchor=north] {\dotlabel{Y}} to (0,0.3);
    \end{tikzpicture}
    =
    \begin{tikzpicture}[centerzero]
        \draw[\rectcolor] (-0.45,-0.3) rectangle (0.45,0.3);
        \draw (0.25,-0.3) node[anchor=north] {\dotlabel{Y}} to (0.25,0.3);
        \draw[wipe] (-0.15,-0.3) to[out=up,in=200] (0.45,0);
        \draw (-0.15,-0.3) node[anchor=north] {\dotlabel{X}} to[out=up,in=200] (0.45,0);
        \draw (-0.45,0) to[out=20,in=down] (-0.15,0.3) node[anchor=south] {\dotlabel{X}};
    \end{tikzpicture}
    = [\xi_{Y,X} \circ \beta_{X,Y}]
    \ .
\]

The action described in \cref{sec:action} induces an action of the trace, as we now explain.

\begin{lem} \label{ice}
    If $\cC$ and $\cD$ are $\kk$-linear categories, we have an isomorphism of $\kk$-modules
    \begin{equation} \label{sphinx}
        \vtr(\cC \boxtimes \cD) \xrightarrow{\cong} \vtr(\cC) \otimes \vtr(\cD),\quad
        [f \otimes g] \mapsto [f] \otimes [g],
    \end{equation}
    for $f$ an endomorphism in $\cC$ and $g$ an endomorphism in $\cD$, extended by $\kk$-linearity.  If $\cC$ and $\cD$ are also monoidal, then \cref{sphinx} is an isomorphism of $\kk$-algebras.
\end{lem}

\begin{proof}
    For $f_1,f_2$ endomorphisms in $\cC$ and $g_1,g_2$ endomorphisms in $\cD$, we have
    \[
        ([f_1] \otimes [g_1]) \circ ([f_2] \otimes [g_2])
        = [f_1 \circ f_2] \otimes [g_1 \circ g_2]
        = [f_2 \circ f_1] \otimes [g_2 \circ g_1]
        = ([f_2] \otimes [g_2]) \circ ([f_1] \otimes [g_1]).
    \]
    Thus the map \cref{sphinx} is well defined.  It is similarly straightforward to verify that the map
    \[
        \vtr(\cC) \otimes \vtr(\cD) \to \vtr(\cC \boxtimes \cD),\quad
        [f] \otimes [g] \mapsto [f \otimes g],
    \]
    is also well defined, and inverse to \cref{sphinx}.
    \details{
        To see that this map is well-defined, we compute
        \[
            [(f_1 \circ f_2) \otimes g]
            = [(f_1 \otimes g) \circ (f_2 \otimes 1)]
            = [(f_2 \otimes 1) \circ (f_1 \otimes g)]
            = [(f_2 \circ f_1) \otimes g]
        \]
        and
        \[
            [f \otimes (g_1 \circ g_2)]
            = [(f \otimes g_1) \circ (1 \otimes g_2)]
            = [(1 \otimes g_2) \circ (f \otimes g_1)]
            = [f \otimes (g_2 \circ g_1)].
        \]
    }
\end{proof}

\begin{prop} \label{fire}
    Suppose $\cC$ is a strict $\kk$-linear monoidal category, $\cM$ is a balanced strict $\kk$-linear monoidal category, and $F \colon \cC \to \cM$ is a monoidal functor.  Then $\vtr(\cM)$ is a $\vtr(\Aff(\cC))$-module, with action given by
    \begin{equation}
        [f] \cdot [g] = [f \cdot g],\quad
        f \in \bigoplus_{X \in \Ob(\cC)} \End_\cC(X),\
        g \in \bigoplus_{X \in \Ob(\cM)} \End_\cM(X),
    \end{equation}
    where, on the right-hand side, $f \cdot g$ denotes the action from \cref{salamander}.
\end{prop}

\begin{proof}
    By (the $\kk$-linear version of) \cref{salamander}, we have a $\kk$-linear functor $\Aff(\cC) \boxtimes \cM \to \cM$.  This induces a $\kk$-linear map $\vtr(\Aff(\cC) \boxtimes \cM) \to \vtr(\cM)$.  By \cref{ice}, this induces a $\kk$-linear map $\vtr(\Aff(\cC)) \otimes \vtr(\cM) \to \vtr(\cM)$.  The associativity and unity axioms follow from the corresponding properties of the action of $\Aff(\cC)$ on $\cM$.
\end{proof}

If $\cC$ is a balanced strict $\kk$-linear monoidal category, then we can take $\cM = \cC$ and $F$ the identity functor in \cref{fire} to see that $\vtr(\cC)$ is a $\vtr(\Aff(\cC))$-module.  If $\cC$ is strict pivotal, this action can be interpreted diagrammatically as follows.  As described above, the element $[f] \in \vtr(\Aff(\cC))$ can be viewed as a $\cC$-diagram on the torus, while $[g] \in \vtr(\cC)$ can be viewed as a $\cC$-diagram in the annulus.  Thickening the annulus, we view $[g]$ as a diagram in the solid torus.  Then $[f] \cdot [g]$ is obtained by placing this solid torus inside the torus carrying the diagram of $f$, viewing the result as a diagram in the solid torus, which we then project onto the annulus, using the braiding to formalize what it means for one strand to pass over another.

\begin{eg}
    Recall the framed HOMFLYPT skein category from \cref{subsec:HOMFLYPT}.  In the language of \cite{MS17}, $\vtr(\Aff(\OS(D;z,t,\delta))) = \vtr(\OS(A;z,t,\delta))$ is the \emph{HOMFLYPT skein algebra of the torus}.  As explained above, it acts on $\vtr(\OS(D;z,t,\delta))$, which is the \emph{skein of the solid torus}.  This action is studied in \cite{MS17}.
\end{eg}

Recall that the \emph{center} $Z(\cC)$ of a monoidal category $\cC$ is $\End_\cC(\one)$, the endomorphism algebra of the identity object.  The next result shows that if $\cC$ is a braided strict monoidal category (so that $\Aff(\cC)$ is a strict monoidal category by \cref{afftensor}) that is either left or right rigid, then the vertical trace of $\cC$ is isomorphic to the center of $\Aff(\cC)$.

\begin{theo} \label{cinnamon}
    \begin{enumerate}
        \item \label{rightroll} If $\cC$ is a right-rigid braided strict monoidal category, then the map
            \begin{equation} \label{rightrollin}
                \vtr(\cC) \to Z(\Aff(\cC)),\quad
                [f] \mapsto \epsilon_X \circ (f \otimes 1_{X^\vee}) \circ \xi_{X,X^\vee}^{-1} \circ \eta_X
                =
                \begin{tikzpicture}[centerzero]
                    \identify{-0.5}{-0.4}{0.5}{0.4};
                    \draw[->] (-0.5,0) to[out=-30,in=180] (-0.2,-0.25);
                    \draw[->] (-0.2,-0.25) to[out=0,in=down] (0,0) to[out=up,in=180] (0.2,0.25);
                    \draw (0.2,0.25) to[out=0,in=150] (0.5,0);
                    \coupon{0,0}{f};
                \end{tikzpicture}
                \ ,\ f \in \End_\cC(X),
            \end{equation}
            is an isomorphism of $\kk$-algebras.

        \item \label{leftroll} If $\cC$ is a left-rigid braided strict monoidal category, then the map
            \begin{equation} \label{leftrollin}
                \vtr(\cC) \to Z(\Aff(\cC)),\quad
                [f] \mapsto \epsilon_X' \circ (1_{\leftdual{X}} \otimes f) \circ \xi_{X,\leftdual{X}} \circ \eta_X'
                =
                \begin{tikzpicture}[centerzero]
                    \identify{-0.5}{-0.4}{0.5}{0.4};
                    \draw[->] (0.5,0) to[out=210,in=0] (0.2,-0.25);
                    \draw[->] (0.2,-0.25) to[out=180,in=down] (0,0) to[out=up,in=0] (-0.2,0.25);
                    \draw (-0.2,0.25) to[out=180,in=30] (-0.5,0);
                    \coupon{0,0}{f};
                \end{tikzpicture}
                \ ,\ f \in \End_\cC(X),
            \end{equation}
            is an isomorphism of $\kk$-algebras.
    \end{enumerate}
\end{theo}

\begin{proof}
    We give the proof of \cref{rightroll}, since the proof of \cref{leftroll} is analogous.  Suppose $\cC$ is a right-rigid braided strict monoidal category.  Recall that the \emph{right mate} of a morphism $f \colon X \to Y$ is the morphism
    \[
        f^\vee
        =
        \begin{tikzpicture}[centerzero]
            \draw[->] (-0.4,0.5) node[anchor=south]{\dotlabel{X}} -- (-0.4,-0.1) arc(180:360:0.2) -- (0,0.1) arc(180:0:0.2) -- (0.4,-0.5) node[anchor=north] {\dotlabel{Y}};
            \coupon{0,0}{f};
        \end{tikzpicture}
        \colon Y^\vee \to X^\vee.
    \]
    Suppose $f \colon X \to Y$, $g \colon Y \to X$ are morphisms in $\cC$.  Letting $h := f^\vee$, we have
    \[
        \begin{tikzpicture}[centerzero]
            \draw[\rectcolor] (-0.4,-0.6) rectangle (0.4,0.6);
            \draw[->] (-0.4,0) to[out=-30,in=180] (-0.1,-0.5) arc(270:360:0.1) -- (0,0.4) arc(180:90:0.1) to[out=0,in=150] (0.4,0);
            \coupon{0,0.2}{f};
            \coupon{0,-0.2}{g};
        \end{tikzpicture}
        \overset{\cref{zigright}}{=}
        \begin{tikzpicture}[centerzero]
            \draw[\rectcolor] (-0.4,-0.6) rectangle (0.6,0.6);
            \draw[->] (-0.4,0) to[out=-30,in=180] (-0.1,-0.5) arc(270:360:0.1) -- (0,0.3) arc(180:0:0.2) to[out=270,in=150] (0.6,0);
            \coupon{0.38,0.3}{h};
            \coupon{0,-0.2}{g};
        \end{tikzpicture}
        =
        \begin{tikzpicture}[centerzero]
            \draw[\rectcolor] (-0.6,-0.6) rectangle (0.4,0.6);
            \draw[->] (-0.6,0) to[out=-30,in=180] (-0.4,-0.3) arc(180:360:0.2) -- (0,0.4) arc(180:90:0.1) to[out=0,in=150] (0.4,0);
            \coupon{-0.38,-0.3}{h};
            \coupon{0,0.2}{g};
        \end{tikzpicture}
        \overset{\cref{zigright}}{=}
        \begin{tikzpicture}[centerzero]
            \draw[\rectcolor] (-0.4,-0.6) rectangle (0.4,0.6);
            \draw[->] (-0.4,0) to[out=-30,in=180] (-0.1,-0.5) arc(270:360:0.1) -- (0,0.4) arc(180:90:0.1) to[out=0,in=150] (0.4,0);
            \coupon{0,0.2}{g};
            \coupon{0,-0.2}{f};
        \end{tikzpicture}
        \ ,
    \]
    where the second equality uses the analogue of \cref{whip} for the inverse coil.  It follows that the map \cref{rightrollin} is well defined.

    Now, viewing morphisms in $\Aff(\cC)$ as string diagrams on the cylinder, every element of $Z(\Aff(\cC))$ is a linear combination of elements of the form
    \[
        \begin{tikzpicture}[centerzero]
            \identify{-0.5}{-0.4}{0.5}{0.4};
            \draw[->] (-0.5,0) to[out=-30,in=180] (-0.2,-0.25);
            \draw[->] (-0.2,-0.25) to[out=0,in=down] (0,0) to[out=up,in=180] (0.2,0.25);
            \draw (0.2,0.25) to[out=0,in=150] (0.5,0);
            \coupon{0,0}{f};
        \end{tikzpicture}
        \ ,\quad f \in \End_\cC(X).
    \]
    We define an inverse to \cref{rightrollin} by sending this element to $[f]$.  The argument that this inverse is well defined is analogous to the argument given above that \cref{rightrollin} is well defined.
\end{proof}

\section{Affinization of 2-categories\label{sec:2aff}}

In this final section, we briefly describe how the concept of affinization can be generalized to the setting of 2-categories.  In a 2-category, we will use juxtaposition to denote horizontal composition of 1-morphisms and 2-morphisms, and we will use $\circ$ to denote vertical composition.

\begin{defin}[Affinization of a 2-category] \label{walrus}
    The \emph{affinization} $\Aff(\fC)$ of a 2-category $\fC$ is the category defined as follows. Objects of $\Aff(\fC)$ are 1-endomorphisms $f \colon x \to x$ in $\fC$.  The morphisms of $\Aff(\fC)$ are obtained from the class of 2-morphisms $\sigma \colon f \to g$ in $\fC$ between 1-endomorphisms by adjoining invertible 2-morphisms
    \[
        \xi_{f,g} \colon fg \to gf,\quad f \colon y \to x,\ g \colon x \to y,\quad x,y \in \Ob(\fC),
    \]
    and imposing the relations
    \begin{gather}
        \xi_{f,gh} = \xi_{hf,g} \circ \xi_{fg,h}, \\
        \xi_{f_2,g_2} \circ (\tau \sigma) = (\sigma \tau) \circ \xi_{f_1,g_1},
    \end{gather}
    for all $f \colon z \to x$, $g \colon y \to z$, $h \colon x \to y$, $x,y,z \in \Ob(\fC)$, and all $\sigma \colon g_1 \to g_2$, $\tau \colon f_1 \to f_2$ for $g_1,g_2 \colon x \to y$, $f_1,f_2 \colon y \to x$, $x,y \in \Ob(\fC)$.
\end{defin}

The reader should compare \cref{walrus} to \cref{affdef}.  In particular, a strict monoidal category is the same as a 2-category with one object.  In this case, \cref{walrus,affdef} coincide.

As for the affinization of strict monoidal categories, morphisms in $\Aff(\fC)$ can be naturally interpreted as diagrams on the cylinder.  Cutting open the cylinder as in \cref{sec:affinize}, we draw $\xi_{f,g}$ and $\xi_{f,g}^{-1}$ for $f \colon y \to x$, $g \colon x \to y$ as the string diagrams
\[
    \xi_{f,g}
    =
    \begin{tikzpicture}[anchorbase]
        \identify{-0.7}{-0.5}{0.7}{0.5};
        \draw (-0.3,-0.5) node[anchor=north] {\dotlabel{f}} -- (0.3,0.5) node[anchor=south] {\dotlabel{f}};
        \draw (0.3,-0.5) node[anchor=north] {\dotlabel{g}} to[out=up,in=200] (0.7,0);
        \draw (-0.7,0) to[out=20,in=down] (-0.3,0.5) node[anchor=south] {\dotlabel{g}};
        \node at (0.55,-0.3) {\dotlabel{x}};
        \node at (0.35,0.1) {\dotlabel{y}};
        \node at (-0.35,-0.1) {\dotlabel{x}};
        \node at (-0.55,0.3) {\dotlabel{y}};
    \end{tikzpicture}
    \ ,\qquad
    \xi_{f,g}^{-1}
    =
    \begin{tikzpicture}[anchorbase]
        \identify{-0.7}{-0.5}{0.7}{0.5};
        \draw (0.3,-0.5) node[anchor=north] {\dotlabel{f}} -- (-0.3,0.5) node[anchor=south] {\dotlabel{f}};
        \draw (-0.3,-0.5) node[anchor=north] {\dotlabel{g}} to[out=up,in=-20] (-0.7,0);
        \draw (0.7,0) to[out=160,in=down] (0.3,0.5) node[anchor=south] {\dotlabel{g}};
        \node at (-0.55,-0.3) {\dotlabel{y}};
        \node at (-0.35,0.1) {\dotlabel{x}};
        \node at (0.35,-0.1) {\dotlabel{y}};
        \node at (0.55,0.3) {\dotlabel{x}};
    \end{tikzpicture}
    \ .
\]
Regions of the cylinder are labeled by objects of $\fC$, strings, which are allowed to wrap around the cylinder, are labeled by 1-morphisms in $\fC$ and these strings can carry coupons labeled by 2-morphisms in $\fC$.  For example,
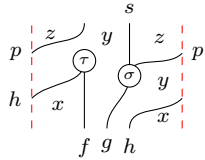
\[
    \begin{tikzpicture}[anchorbase]
        \identify{-1}{-0.7}{1}{0.7};
        \draw (0.3,-0.7) node[anchor=north] {\dotlabel{h}} to[out=up,in=225] (1,-0.3);
        \draw (0,-0.7) node[anchor=north] {\dotlabel{g}} \braidup (0.3,0) -- (0.3,0.7) node[anchor=south] {\dotlabel{s}};
        \draw (0.3,0) to[out=up,in=225] (1,0.3) node[anchor=west] {\dotlabel{p}};
        \draw (-0.3,-0.7) node[anchor=north] {\dotlabel{f}} -- (-0.3,0.2);
        \draw (-1,-0.3) node[anchor=east] {\dotlabel{h}} to[out=45,in=down] (-0.3,0.2);
        \draw (-1,0.3) node[anchor=east] {\dotlabel{p}} to[out=45,in=down] (-0.3,0.7);
        \coupon{-0.3,0.2}{\tau};
        \coupon{0.3,0}{\sigma};
        \node at (0.75,-0.55) {\dotlabel{x}};
        \node at (-0.65,-0.4) {\dotlabel{x}};
        \node at (0.75,-0.1) {\dotlabel{y}};
        \node at (0,0.5) {\dotlabel{y}};
        \node at (0.7,0.5) {\dotlabel{z}};
        \node at (-0.75,0.55) {\dotlabel{z}};
    \end{tikzpicture}
\]
is a 2-morphism in $\Aff(\fC)$ for $x,y,z \in \Ob(\fC)$, $f \colon y \to x$, $g \colon y \to y$, $h \colon x \to y$, $\sigma \colon g \to sp$, and $\tau \colon hf \to 1_y$.

The horizontal trace of a 2-category is defined in \cite[\S2.4]{BHLZ17}.  The discussion of \cref{sec:htr} has a straightforward generalization to the setting of 2-categories.  In particular, the horizontal trace of a 2-category $\fC$ is isomorphic to $\Aff(\fC)$ if $\fC$ is rigid, but the two concepts are different in general.  (Since a strict monoidal category can be considered as a 2-category with one object, \cref{golf} gives an example illustrating the difference.)  For an arbitrary 2-category, it is the affinization that naturally corresponds to the category of $\fC$-diagrams on the cylinder.


\bibliographystyle{alphaurl}
\bibliography{AffMonCat}

\end{document}